\documentclass[12pt, oneside, reqno]{amsart}

\title[DPG loss functions for neural networks]{DPG loss functions for learning parameter-to-solution maps by neural networks}
\author{Pablo Cort{\'e}s Castillo}
\author{Wolfgang Dahmen}
\author{Jay Gopalakrishnan}

\usepackage{amssymb}	
\usepackage{tikz}
\usepackage{tikz-cd}
\usetikzlibrary{positioning}
\usetikzlibrary{calc}
\usepackage{amsmath}
\usepackage{mathrsfs}
\usepackage{amsthm}
\usepackage{verbatim}
\usepackage{subeqnarray}
\usepackage{subcaption}
\usepackage{graphicx}
\usepackage{cancel}
\usepackage{color}
\usepackage{float}
\usepackage{bm}
\usepackage{hyperref}
\usepackage{amsfonts}
\usepackage{amscd}
\usepackage{cases}
\usepackage{xcolor}
\usepackage{mathtools}
\usepackage{epstopdf} 
\usepackage[bbgreekl]{mathbbol}
\usepackage[margin=2.5cm]{geometry}
\usepackage{pgfplots}
\usepackage{multirow}
\usepgfplotslibrary{groupplots}
\pgfplotsset{compat=1.18}

\definecolor{crimson}{RGB}{220, 20, 60}
\definecolor{navy}{RGB}{0, 0, 128}
\definecolor{royalblue}{RGB}{65, 105, 225}
\definecolor{darkcyan}{RGB}{0, 139, 139}
\definecolor{tabblue}{HTML}{1F77B4}
\definecolor{taborange}{HTML}{FF7F0E}
\definecolor{tabgreen}{HTML}{2CA02C}
\definecolor{tabred}{HTML}{D62728}
\definecolor{tabpurple}{HTML}{9467BD}
\definecolor{dgreen}{rgb}{0.2,0.72,0.3}

\newcommand{\bb}[1]{\mathbb{{#1}}}
\newcommand{\cl}[1]{\mathcal{{#1}}}

\newcommand{\Qt}{\tilde{\mathcal{Q}}}

\newcommand{\Real}{\mathbb{R}}

\newcommand{\vertiii}[1]{{\left\vert\kern-0.25ex\left\vert\kern-0.25ex\left\vert #1 
    \right\vert\kern-0.25ex\right\vert\kern-0.25ex\right\vert}}

\newcommand{\vertiiii}[1]{{\left\vert\kern-0.25ex\left\vert\kern-0.25ex\left\vert\kern-0.25ex\left\vert #1 
    \right\vert\kern-0.25ex\right\vert\kern-0.25ex\right\vert\kern-0.25ex\right\vert}}

\newtheorem{theorem}{Theorem}[section]

\newtheorem{proposition}[theorem]{Proposition}
\theoremstyle{definition}
\newtheorem{definition}[theorem]{Definition}
\newtheorem{example}[theorem]{Example}
\theoremstyle{remark}

\newcommand{\Mtest}{M_{\text{test}}}
\newcommand{\pinn}{\mathrm{pinn}}
\newcommand{\dpg}{\mathrm{dpg}}
\newcommand{\Xhdpg}{\cl{X}_h^{\dpg}}
\newcommand{\Xdpg}{\cl{X}^{\dpg}}
\newcommand{\Xdpgo}{\cl{X}_0^{\dpg}}
\newcommand{\Xdpgh}{\hat{\cl{X}}^{\dpg}}

\newcommand{\xdpg}{x_h^{\dpg}}

\newcommand{\Ydpg}{\cl{Y}^\dpg}
\newcommand{\Yhdpg}{\cl{Y}_h^\dpg}
\newcommand{\fdpg}{\cl{F}_h^{\dpg}}

\newcommand{\fosls}{\mathrm{fos}}
\newcommand{\Xhfosls}{\cl{X}_h^{\fosls}}
\newcommand{\Xfosls}{\cl{X}^{\fosls}}

\newcommand{\xfosls}{x_h^{\fosls}}
\newcommand{\ffosls}{\cl{F}_h^{\fosls}}

\newcommand{\oh}{\varOmega_h}
\newcommand{\Xh}{{\cl{X}_h}}
\newcommand{\Fh}{{\cl{F}_h}}
\newcommand{\Fht}{{\cl{F}_{h,\theta}}}

\newcommand{\om}{\varOmega}
\renewcommand{\Omega}{\varOmega}
\renewcommand{\forall}{\text{ for all }} 
\newcommand{\Hdiv}{H(\div)}
\newcommand{\Hdivoh}{H(\div, \oh)}
\newcommand{\Ho}{\mathring{H}^1}
\newcommand{\Hoh}{\mathring{H}^{1/2}(\d\oh)}

\newcommand{\ksa}{k_{s,\alpha}}
\newcommand{\ka}[1]{k_{{#1},\alpha}}

\newcommand{\Lbdo}{\cl{L}^{\fosls}}
\newcommand{\Lfosls}{\cl{L}^{\fosls}}
\newcommand{\Ldpg}{\cl{L}^{\dpg}}
\newcommand{\hLsdpg}{\hat{\cl{L}}_{s}^{\dpg}}
\newcommand{\Lsdpg}{{\cl{L}}_{s}^{\dpg}}
\newcommand{\Lss}[2]{{{\cl{L}}}_{{#1},{#2}}^{\dpg}}
\newcommand{\hLss}[2]{\hat{{\cl{L}}}_{{#1},{#2}}^{\dpg}}
\newcommand{\hLssL}[2]{\hat{{\cl{L}}}_{{#1},{#2}}^{\dpg, L}}
\newcommand{\hLssR}[2]{\hat{{\cl{L}}}_{{#1},{#2}}^{\dpg, R}}
\newcommand{\hLs}[1]{{\hat{\cl{L}}}_{{#1}}^{\dpg}}
\newcommand{\Lsa}[1]{{{\cl{L}}}_{{#1}}^{\dpg}}

\newcommand{\hU}{\hat{\cl{U}}}
\newcommand{\hRT}{\hat{\cl{R\!T}}_{\!0}}
\newcommand{\RT}{{\cl{R\!T}_{\!0}}}

\let\div\relax
\DeclareMathOperator{\tr}{{\ensuremath{\text{tr}}}}
\newcommand{\trn}{\mathop{\mathrm{tr}_n}}
\DeclareMathOperator{\osc}{{\ensuremath{\text{osc}}}}
\DeclareMathOperator{\div}{{\ensuremath{\text{div}}}}

\newcommand{\vepzh}{\varepsilon^{\w}_h}
\newcommand{\vepzs}{\varepsilon^{\w}_{s}}
\newcommand{\vepzhs}{\varepsilon^{\w}_{h, s}}
\newcommand{\vpi}{\varPi_h}
\newcommand{\vpig}{\varPi_h^g}
\newcommand{\vpid}{\varPi_h^d}
\newcommand{\ip}[1]{\langle{#1}\rangle}
\def\d{\partial}

\DeclareMathOperator*{\argmin}{arg\,min}
\DeclareMathOperator*{\cmax}{cmax}

\def\argmin{\mathop{\rm argmin}}
\newcommand{\pp}{\alpha}
\newcommand{\pdom}{\mathcal{Q}}
\newcommand{\cR}{\mathcal{R}}
\newcommand{\cM}{\mathcal{M}}
\newcommand{\cF}{\mathcal{F}}
\newcommand{\R}{\mathbb{R}}
\newcommand{\LL}{\mathbb{L}}

\newcommand{\x}{\mathfrak{u}}
\newcommand{\y}{\mathfrak{v}}
\newcommand{\z}{\mathfrak{z}}
\newcommand{\w}{\mathfrak{w}}
\newcommand{\uu}{\mathfrak{u}}

\newcommand{\V}{\cl Y}
\newcommand{\X}{\mathbb{X}}
\newcommand{\Y}{\mathbb{Y}}
\newcommand{\cX}{\mathcal{X}}
\newcommand{\cY}{\mathcal{Y}}
\renewcommand{\xfosls}{\x^{\rm fos}_h}
\renewcommand{\xdpg}{\x^{\rm dpg}_h}

\begin{document}
\begin{abstract}
  {We develop, analyze, and experimentally explore residual-based loss functions for machine learning of parameter-to-solution maps in the context of parameter-dependent families of partial differential equations (PDEs). Our primary concern is on rigorous accuracy certification to enhance the prediction capability of the resulting deep neural network reduced models. This is achieved by the use of variationally correct loss functions. Through one specific example of an elliptic PDE, details for establishing the variational correctness of a loss function from an ultraweak  Discontinuous Petrov Galerkin (DPG) discretization are worked out. Despite the  focus on the example,  the proposed concepts apply to a much wider scope of problems, namely problems for which stable DPG formulations are available. The issue of high-contrast diffusion fields and ensuing difficulties with degrading ellipticity are discussed. Both numerical results and theoretical arguments illustrate that for high-contrast diffusion parameters the proposed DPG loss functions deliver much more robust performance than simpler least-squares losses.}
\end{abstract}
\maketitle

\section{Introduction}
\label{sec:introduction}
\subsection{Background and Motivation}
The need to recover information on a physical state of interest from a
limited amount of observational data is ubiquitous in scientific and
technological applications.  To make such tasks feasible, it is
crucial to leverage the physical laws that the states obey. However,
these are often only partially known. A common way to deal with such
uncertainties is to formulate the mathematical model as a system of
partial differential equations (PDEs) depending on unknown model data
(such as coefficient fields, initial or boundary conditions, source
terms, or constitutive laws), collected into {\em parameters} $\pp$
from finite- or infinite-dimensional spaces. We have in mind PDE solutions $u(x, \alpha)$ depending on a spatial variable $x$ in a low-dimensional domain $\om \subset \R^d$ and a parameter variable $\alpha$ in a possibly high-dimensional domain $\cl Q$. 
For each parameter instance $\pp\in \pdom$, the solution $u(x, \alpha)$, also written as $u(\alpha)$ or simply $u$, is 
the unique solution of 
\begin{equation}
\label{fiber}
\cR(u,\pp)=0,
\end{equation}
for some (generally nonlinear) PDE $\cR$ written in residual form.
The structure and range of $\pdom$ determine the ``design space'' or
``solution manifold'' $\cM$ comprised of the states $u(\pp)$ that are
obtained as solutions of \eqref{fiber} when $\pp$ traverses $\pdom$.
Equation~\eqref{fiber} implicitly defines a {\em
  parameter-to-solution} map $\cF: \pdom \to \cM$.
At each parameter instance $\pp \in \pdom$, the problem~\eqref{fiber}
for $u$ at the corresponding point on the manifold is referred to as a
``fiber problem.''

Solving tasks like state-estimation or parameter-estimation, when
using $\cM$ as a prior, typically requires (approximately) evaluating
the map $\alpha \mapsto \cF(\alpha)\in \cM$ for a large number of
parameter queries $\alpha \in \mathcal Q$.  A correspondingly large
number of high-fidelity solves of \eqref{fiber} during the inversion
process entails enormous, often prohibitive, computational cost. This
calls for generating {\em reduced models} and efficiently computable
surrogates for the {parameter-to-solution} map $\cF$.  The underlying
rationale is that evaluating such a surrogate at any given parameter
instance is more efficient than computing a corresponding
high-fidelity approximate solution~\cite{BHKS2021,DDP,Stu1}.  In
``Scientific Machine Learning'' (SML) one employs machine learning
concepts to generate a surrogate for $\cF$, a central instance of
``Operator Learning''~\cite{BLST2023,KLLABSA2023,Stu2}.
Results~\cite{OpschSchwaZech22} showing that neural networks, under
certain regularity assumptions, can approximate $\cF$ well despite the
high dimension, are encouraging in this context. The offline cost of
learning $\cF$ amortizes quickly when many parameter queries are
needed. Nevertheless, offline resources are not unlimited, an aspect
that leads one to distinguish between the data-driven and data-free
paradigms in SML.

\subsection{Purely Data-Driven vs.~Residual-Based Training}
Most operator learning approaches are purely {\em data-driven}: 
one first computes a sufficiently large number
of high-fidelity solutions that are then used as training samples for performing regression over suitable hypothesis classes for operators, see,
e.g.,~\cite{BLST2023,GeistPeterRasla21,CGE:23,Stu2,KLLABSA2023}.
The accuracy can then be assessed via estimating the  ``generalization error'' incurred.
This can be done by evaluating the loss function on a test set,
an independent set with a sufficiently large
number of test samples. Since one expects, at best, slow Monte-Carlo convergence rates of empirical losses to corresponding expectations, the number of high-fidelity solutions to be generated in an offline phase 
may  become so high that  alternate
strategies may become preferable.

In fact, employing {\em residual-based loss functions} allows one to avoid the need    to solve many large systems of equations during training. The best-known version is PINN (Physics-Informed Neural Network)~\cite{Kar2,Kar3,Kar1,PINNDeepO}. Here, letting $u_\theta(x, \alpha)$ denote the
solution output of a network (whose weights and biases are collected into $\theta$), 
one  minimizes  a least-squares loss 
\begin{equation}
    \label{PINN}
\theta^*\in \argmin_{\theta\in \Theta}\frac{1}{\#\widehat\pdom \;\#\widehat\Omega}\sum_{\pp\in \widehat\pdom, x\in \widehat\Omega}\big|\cR(u_\theta(x,\pp),\pp)\big|^2
\end{equation}
to learn an approximation $u_{\theta^*}$ from a given hypothesis class
determined by a finite budget $\Theta$ of trainable values of $\theta$.
Here $\widehat\Omega$ is a finite set of spatial samples (such as
quadrature points in $\Omega$ as well as on $\partial\Omega$) and
$\widehat\pdom\subset\pdom$ is a finite collection of model
parameters. The convenient applicability in the above form to a wide range of problems explains why this ansatz 
is very attractive for both, learning parameter-to-solution maps as well as solving
a single PDE.  It is not possible to summarize the burgeoning PINN
literature here but it is worth noting that  most of such literature is   concerned with
parameter-independent problems  aiming to solve just a single fixed PDE, while
our focus is on the parameter-to-solution map.  

In this paper, we entirely focus on the second residual-based learning
approach. PINN~\eqref{PINN}  does, in general, not provide error control. Our central theme is to guarantee rigorous accuracy control using a mathematical
theory of correct loss functions.  While~\eqref{PINN} offers great
computational convenience, problems with its accuracy, especially when
$\cR$ is not an $L_2(\om)$ function, are well-known. Although they
have been noted in several works~\cite{BDO,Pardo,ZeinhMasriMarda24},
we have provided a few further experimental results
in~\S\ref{ssec:compare-pinn}. However, the main thrust of this paper, as described  next, is on theoretical tools for mathematically rigorous accuracy control.

\subsection{Accuracy Control:  Variational Correctness}

In SML, applications are typically error-averse. For an estimation
method to have {\em prediction capability} it is necessary to be able
to assess the accuracy achieved by the training process. From a
classical numerical analysis perspective, it is concerning that the
most commonly used validation strategy is comparison with reference
solutions, often treated as ground truths.  Such reference solutions are
necessarily manufactured solutions or high-fidelity numerical
solutions from particular cases, so  they typically fail to expose all 
intrinsic obstructions posed by a given PDE model, such as boundary
layers in convection-dominated regimes or low spatial regularity
arising from rough or degenerating coefficients. Consequently, they
are not grounds for reliable ``accuracy extrapolation''.  This paper
is motivated by the goal to mitigate such limitations  by
developing  rigorous accuracy control for residual-based approaches.

\subsubsection{Residual-Error Relations}

In order to assess the accuracy of a residual-based training outcome, the generalization of the training loss is essentially the only accessible quantity from which accuracy can be inferred.  The mere fact that the residual vanishes at the exact solution does not, by itself, quantify the accuracy of an approximation. What is crucial, instead, is the metric in which the residual is measured: only in appropriate norms do residuals provide reliable control of the approximation error.

\subsubsection{What are Appropriate Error Metrics?}

Convenience cannot be the only criterion for choosing a loss function.
For instance, when \eqref{fiber} represents a second-order homogeneous
elliptic Dirichlet problem over some domain $\Omega$, even if
convenience may dictate to measure the residual in an $L_2$-type norm,
the wellposedness of the underlying variational formulation implies
that $\cF$ should map into $H^1_0(\Omega)$ and accuracy should be
measured in $H^1_0(\Omega)$. Often flux approximations are even more
important than approximating just the state variable $u$ in a weaker
norm.  Moreover, the method should not require the states $u$ to
belong to a higher regularity space $H^s(\Omega)$, $s>1$: this would
unnecessarily narrow the range of an inversion map and aggravate the
ill-posedness of inversion problems. A method should not be contingent
solely on {\em a priori} excess regularity assumptions. Indeed, when
applying~\eqref{PINN} to a second order elliptic problem, the sampling
limit approximates the bulk residual in $L_2(\Omega)$, which is too
strong a norm: the residual may not even belong to
$L_2(\Omega)$ (a problem
noted in several works~\cite{BDO,Pardo,ZeinhMasriMarda24}).  Furthermore, typical discrete $L_2$ boundary penalties
are too weak, since the trace belongs to
$H^{1/2}(\partial\Omega)$. These considerations indicate that the loss
in \eqref{PINN} need not provide a reliable error bound for $u$ in the
relevant norm.

\subsubsection{Variational Correctness}

In this paper, we insist on using 
residual-based loss functions built using a correct error metric.
 We follow \cite{BDO} 
in adopting the following terminology:\medskip
\begin{quote}
  A loss function is called {\em variationally correct} if it is
  uniformly proportional to the square of the error of its input with
  respect to a model-compliant norm.
\end{quote}
\medskip
\noindent
(see also Definition~\ref{def:vc} later).  The key message conveyed in
\cite{BDO} is that the correct spaces (and norms/metrics) for
solutions to~\eqref{fiber} are determined by weak formulations
of~\eqref{fiber} that, in the linear case, are  {\em well-posed} in the
sense of the Babu\v{s}ka (or BNB) Theorem. 
We refer to corresponding norms as ``model-compliant'' as they
determine the minimum necessary regularity a solution to \eqref{fiber}
should have. 
Wellposedness of  a variational formulation  is equivalent to errors in a trial space $\cl X$ being  uniformly proportional to residuals in a test space dual  $\cl Y'$, i.e., $\|
\mathcal R(u_\theta(\alpha))\|^2_{\cl Y'}$ is a variationally correct loss function.
The obvious caveat is that,  unless $\cl Y$ is self-dual $\cl Y=\cl Y'$ (such as 
an $L_2$-space), residuals in a dual norm are hard to evaluate. So in order to make 
practical use of these concepts one has to design computable expressions which, when
applied to a given hypothesis class, remain uniformly proportional to the ideal loss.
It  is shown in~\cite{BDO} that this is possible for a wide range of PDE models where, 
as a common first crucial step, the PDE is transformed (if not yet in this form) into a {\em first order system} of PDEs. In this form one can show that for many models (steady diffusion, Stokes, linear elasticity, parabolic problems, acoustic wave equations) a variational formulation with the test space being an $L_2$-space, is stable.
This gives rise to least-squares formulations (known as FOSLS, see e.g. \cite{FOSLS_coercivity,bochev,opschoor2024,KF1,KF2,GS}) where
the ideal loss is already computable (and PINN would be variationally correct, given correct boundary treatment). The second scenario addresses problems where an {\em ultraweak Discontinuous Petrov-Galerkin} (DPG) formulation is uniformly stable. We refer to~\cite{DPGacta} for the wide spectrum of PDE models for which this has been shown to hold.

Exploiting the  DPG-methodology requires, however, employing so-called {\em hybrid hypothesis 
classes} for the construction of surrogates. More precisely,
we use  finite elements to discretize the low-dimensional $\om$
but use deep neural networks (NN) to discretize $\cl F$ as a mapping from the high-dimensional parameter domain $\cl Q$ to the trial space. This means that our solution predictions take the form
\begin{equation}
  \label{eq:1}
  u_{\theta}(x, \alpha)  = \sum_k v_{k, \theta}(\alpha)\; \psi_k(x),
\end{equation}
where $\psi_k(x)$ are standard finite element basis functions while the
coefficients $v_{k,\theta}(\alpha)$ in the basis expansion are
provided by   NNs,   determined by a collection of trainable weights $\theta$.
Here the spatial resolution, represented by the domain partition, can be viewed as
a hyper-parameter that allows one to steer accuracy in the same way as spatial sampling sites do when both spatial and parametric variables are input-parameters of a single global NN (as, e.g., in~\eqref{PINN}).
Of course, variational correctness is then to be established
with mesh-independent constants.

Although not necessary in the case of $L_2$-type losses, we opt for predictions 
in the form~\eqref{eq:1} in all scenarios, because they offer the following 
advantages (see \cite{BDO} for a detailed discussion). They allow one to compute 
spatial derivatives using finite elements rather than through neural networks; 
they greatly facilitate incorporating boundary conditions; they support ``nested 
refinement'' concepts; and as demonstrated in \cite{BDO}, they improve  
optimization robustness.

\subsection{Related Works}
We refer to \cite{BDO,CGE:23,GZW2022,NeuFE,GeistPeterRasla21} for prior studies where such hybrid hypothesis classes have also been
used. 
Moreover, error control and error analysis have received the attention
of others:
see~\cite{AinswDong25,Bonito,Mishra1,Pardo,ZeinhMasriMarda24,ZBYZ} and
references therein. Depending on the type of PDE, the residual $\cl R$
need not be a function in $L_2(\Omega)$, a problem noted in several
works~\cite{Pardo,ZeinhMasriMarda24}.  These works use ideas related
to wellposed variational formulations confined to special PDE types
(mostly elliptic problems) and attempt to solve a single PDE in low
spatial dimensions.  The wellposedness also features essentially in
the construction of the respective residual loss functions
~\cite{Canuto,Bertol,Friedrichs,Urban,MR:21}.  As already mentioned,
we directly build on the concept of variational correctness introduced
in~\cite{BDO}.  A more recent work~\cite{QCD} also exploits
variational correctness for parameter dependent steady diffusion and
linear ellipticity models, paying special attention to a variationally
correct incorporation of mixed (inhomogeneous) boundary conditions.
FOSLS type loss functions are used in~\cite{QCD} and the hybrid
systems involve variationally correct POD reduced basis functions.
The present work explores issues not addressed in~\cite{BDO,QCD}, and
goes in directions which, to the best of our knowledge, have not been
addressed in earlier studies, as explained next.

\subsection{Our Contribution}

We first provide results on parameter robustness of loss
functions (in addition to proving variational correctness of practical
DPG loss functions).
To put this in
context, the very recent results on variationally correct losses
(of~\cite{BDO} or~\cite{QCD}) hold provided the
underlying variational formulations are stable {\em uniformly} with
respect to the parameter range $\mathcal Q$. In many important
applications, and specifically in the context of inverse problems
involving complex material parameters, an {\it a priori} restriction
of the parameter range may greatly degrade inversion results.  For
instance, in Darcy's model for porous media flows log-normal diffusion
fields can have degenerating ellipticity, albeit with low
probability. Of course, degeneracies may occur for a wide scenario of
PDE models but steady diffusion already marks a challenging and
important application area. Moreover, in view of the results
in~\cite{KF1,KF2,GS} and the scope of wellposed variational
formulations, concepts developed for degenerating ellipticity would
carry over to other models like time-dependent dissipative or wave
propagation models. Therefore, we focus in this work on an in-depth
study of a classical diffusion model allowing ellipticity to
degenerate.

Specifically, in our numerical illustrations we consider {\em
  piecewise constant high-contrast parametric diffusion
  fields}. High-contrast diffusion fields have high ratios of adjacent
values of the diffusion field. The notorious challenges they pose have
been examined to some extent in the reduced order modeling community:
see e.g.~\cite{CDDS} and the literature cited there.  It is not hard
to see that for piecewise constant diffusion (and constant right hand
side) the solution tends to zero in regions where the diffusion tends
to infinity while the gradient tends to infinity where the diffusion
tends to zero (see e.g.~\cite{CDDS}).  The approach in~\cite{CDDS}
developed particular reduced bases incorporating these limit
features. However, these reduced bases suffer severely from the Curse
of Dimensionality.  Neural networks provide interesting alternatives
in this case, which is one reason for our choice of such high-contrast
diffusion fields.  The reason for selecting piecewise constant
diffusion fields is to examine how challenges posed by singularly
perturbed regimes compound with very low spatial regularity of
solutions, which at high-contrast cross-points can go down arbitrarily
close to $H^1$.  Also, such piecewise constant fields allow us to
conveniently vary and monitor the quantitative effect of
degeneracies. We should emphasize that (unlike what seems to be usual
practice in SML) our numerical studies are not intended to validate
our method (since we {\em prove} its correctness by showing the equivalence
of our loss functions with model-compliant error metrics). The
numerical experiments are only intended to shed light on learning
success in relation to sample complexity and to show practical effects of
degenerating ellipticity.

When diffusion parameters vary strongly (while remaining positive and
finite), although FOSLS losses are in principle applicable, their
proportionality to the error in a reference norm degrades as the ratio
of continuity and coercivity constant (or the conditioning of the
elliptic operator) increases.  In later sections, after showing
practical manifestations of this problem, we study variationally
correct training losses based on DPG
concepts~\cite{DemkoGopal11,DPGacta}.  They lead to a more parameter
robust loss functional, as explained next. The so-called conforming
ultraweak variational formulation is wellposed as long as diffusion
coefficients remain positive and finite, regardless of their size.
Its trial space is just the $(d+1)$-fold Cartesian product of
$L_2(\om)$, while the test space $\V$ is now a boundary
condition-induced closed subspace of the graph space of the {\em dual}
operator. Now, changing the standard metric of $\V$ for the one 
induced by the so-called {\em optimal test norms}~\cite{DPGacta}, makes 
both the inf-sup as well as the continuity constant equal to one. 
So the operator, induced by such a weak formulation is an {\em isometry} 
from $L_2$ to the dual of the test space, yielding ``ultimate parameter 
robustness''. Moreover, due to the first order reformulation, we have two 
unknowns, namely the state variable $u$ and its flux. An error metric 
measuring both in $L_2$ (which is what the ultraweak formulation gives) is 
equivalent to $H^1$-control of $u$ solving the original second order 
formulation, i.e., we obtain error control in the right model-compliant 
metric and perfect parameter robustness (as worked out in detail in
Example~\ref{eg:uw}).  Unfortunately, the non-trivial dual norm in
this conforming ultraweak formulation cannot be easily evaluated, so
we cannot exploit its ultimate parameter robustness in this form.  The
idea is to instead use an ultraweak DPG formulation, which mimics, to
some extent, the conforming ultraweak formulation. We study to what
extent the parameter robustness of the conforming variant carries over
to the ultraweak DPG formulation. Our main findings revolve around
this issue.

They are precisely given in Section~\ref{sec:param-robust},
but let us  describe the thinking in lay terms here.
A close study of the DPG formulation reveals that it cannot perfectly emulate
the optimal test norm of its conforming counterpart for two
roughly sketchable reasons.  First, one must endow DPG's interelement
trace variables with a proper norm which, in essence, is the graph
norm of lifted traces. Therefore, it re-introduces to some extent the
FOSLS norm and hence inherits its problems when ellipticity
degenerates.  Second, the test norm is a broken version of the optimal
test norm for its conforming variant, augmented with an $L_2$ norm
term to avoid potential element-wise kernels of the former. Both of
these obstructions can be significantly mitigated, as we shall show,
by two measures. First, we introduce a scaling parameter to weight the
zero-order terms in the broken test norm. The smaller the zero-order
contribution the closer we get to the optimal (conforming) test norm in
which we expect increased robustness. Second we present modified DPG
losses that allows us to ``fade out'' the effect of the FOSLS graph
norm while retaining robust $L_2$-error bounds for $u$ and its flux
(which corresponds to $H^1$ error control of the original second-order
model solution). This suggests a number of interesting new algorithmic
variants, in particular, making the scaling factor a {\em trainable
  parameter}, thereby adapting the loss in favor of parameter
robustness.

\subsection{Layout of the Paper}
Section~\ref{sec:param-solution-map} is devoted to the problem
description. Starting with three simple examples in
\S\ref{ssec:problem}, we precisely outline some of the issues
(mentioned previously in lay terms) that motivate the remainder of the
developments. In \S\ref{ssec:vc}, we define concisely the central
notion of variationally correct loss functions based on a functional
setting that respects the mapping properties of parameter-to-solution
maps. In Section~\ref{sec:fosls-loss-function} we introduce the FOSLS
loss function and establish its variational correctness (with
constants depending on the parameter range). In essence, this
follows~\cite{BDO}, except that we highlight the dependence of the
constants involved on the parameter $\pp$.  An alternate DPG
formulation is then introduced and analyzed in
Section~\ref{sec:dpg-loss-function}. Its variational correctness is
also rigorously proved. The performance of both formulations with
regard to prediction accuracy is studied in
Section~\ref{sec:initial-results}.  These findings indicate some
advantages of the DPG formulation when the contrast in the diffusion
coefficient increases and the conditioning of the FOSLS operator
exhibits a stronger degradation. This motivates an in-depth
robustness study in Section~\ref{sec:param-robust}.

\section{A parametric PDE model}
\label{sec:param-solution-map}

\subsection{The model problem and first-order reformulations}
\label{ssec:problem}

Let $\om$ be a bounded Lipschitz domain in $\Real^d$, let
$\alpha : \om \to \Real$ be a positive function, and let
$a(x) = \alpha(x)^{-1}$. Consider the stationary diffusion equation
\begin{equation}
  \label{laplaceDP}
    \begin{aligned}
      -\div(a(x) \nabla u)
      & = f && \text{ in } \quad \Omega,
      \\
      u & = 0 &&  \text{ on } \quad \partial\Omega.
    \end{aligned}
\end{equation}
Here we are not interested in solving \eqref{laplaceDP} for a fixed
$a(x)$, but for a parameterized range $\pdom$ of diffusion
coefficients which will be specified later below. When $a$
varies, or equivalently $\alpha$ varies, so does the solution $u$,
which we shall write as $u(\alpha)$ or $u(x, \alpha)$ when we need
to highlight the $\alpha$-dependence. As mentioned in
Section~\ref{sec:introduction}, a central objective in operator
learning is to construct surrogates for the mapping $\alpha \mapsto u(\alpha)$ so
as to facilitate an efficient exploration of the corresponding design
space $\cM$ of viable states $u(\alpha)$, $\alpha\in \pdom$.

The boundary value problem~\eqref{laplaceDP} admits many well-posed variational
formulations (three of which are given below).
Suppose $\x$ is a group variable which contains $u$ as
one of the components and suppose that there is a well-posed variational
formulation that uniquely fixes $\x \in \cl X$ as the solution of a
weak formulation
\begin{equation}
  \label{eq:2}
  b_\alpha(\x, \y) = \ell(\y) \qquad \text{ for all } \y \in \V,
\end{equation}
where $b_\alpha: \cl X \times \cl Y \to \bb R$ is a continuous
bilinear form on the product of some Hilbert spaces $\cl X$ and
$\cl Y$ and $\ell \in \cl Y'$ is the functional capturing the
source~$f$. Let the continuous operator generated by $b_\alpha$ be
denoted by $B_\alpha: \cl X \to \cl Y'$, i.e.,
$(B_\alpha \w)(\y) = b_\alpha(\w, \y)$ for all $\w \in \cl X$ and
$\y \in \cl Y$. The wellposedness of~\eqref{eq:2}, by Babu\v{s}ka's
theorem, implies that
\begin{equation}
  \label{eq:3}
  \| \w \|_{b_\alpha} = \sup_{0 \ne \y \in \cl Y}
  \frac{ b_\alpha( \w, \y) }{ \| \y\|_{\cl Y} } 
\end{equation}
is an equivalent norm on $\cl X$. 
Suppose we are given a candidate
approximation $\w \in \cl X$ to the exact solution $\x$
of~\eqref{eq:2}. The error in such an approximation satisfies
\begin{equation}
  \label{eq:4}
  \| \x - \w \|_{b_\alpha} =
  \sup_{0 \ne \y \in \cl Y}
  \frac{ b_\alpha( \x - \w, \y) }{ \| \y\|_{\cl Y} }
  =
  \sup_{0 \ne \y \in \cl Y}
  \frac{ \ell(\y) - b_\alpha( \w, \y) }{ \| \y\|_{\cl Y} }
  = \| \ell - B_\alpha \w \|_{\cl Y'},
\end{equation}
which we refer to as the {\em error-residual equation}.
We may now set $\cl R$ in \eqref{fiber} by
$\cl R( \w, \alpha) = \| \ell - B_\alpha \w \|_{\cl Y'}$ and observe that
the solution $\x$ indeed satisfies~\eqref{fiber} due to~\eqref{eq:4}.

The central issues that occupy us in this paper can be described in
simple terms using~\eqref{eq:4}. We would like to control the error
$\x - \w$ in some desired $\alpha$-independent norm as robustly as
possible while $\alpha $ varies in $\cl Q$.  To do so, we may use the
residual $\| \ell - B_\alpha \w \|_{\cl Y'}$ in~\eqref{eq:4}
containing only known quantities $\ell$ and $\w$, provided we know
{\em (i)}~how to compute (the residual in) the dual norm
$\| \cdot \|_{\cl Y'}$, and {\em (ii)}~how to obtain $\alpha$-robust
equivalences between the desired $\alpha$-independent norm and
$\| \cdot \|_{b_\alpha}$ (which involves dealing with degenerating
ellipticity when $\alpha$ attains very large or very small values
while ranging over $\pdom$).  The following examples show to what
extent these {\em two issues can be influenced by the choice of the
variational formulation.}

\begin{example}[Primal formulation]
  \label{eg:standard-weak-form}
  The canonical weak formulation of~\eqref{laplaceDP} seeks the
  solution $\x = u$ in the Sobolev space
  ${\Ho} = \{ v \in L_2(\om): \d_i v \in L_2(\om), \text{ for every }
    i=1,\dots, d$ and  $v|_{\d\om} =0\}$, setting
  \[
    \cl X = \cl Y = \Ho,\qquad
    b_\alpha(u, v) = \int_\om \alpha^{-1} \nabla u \cdot \nabla v \; dx, \qquad
    \ell (v) = \int_\om f \, v \; dx
  \]
  and solving~\eqref{eq:2}.  Issue~{\em (i)} manifests as the
  difficulty to compute the $H^{-1}:=(\Ho)'$ norm. Issue~{\em
    (ii)} is evident from the strong $\alpha$-dependence of the
  $\| \cdot \|_{b_\alpha}$-norm in this case, e.g., 
  \[
    \| \nabla v \|_{L_2}^2
    \;\min_{x \in \om} \alpha^{-1}  
    \;\le\;  \| v \|_{b_\alpha}^2
    \;\le\;
    \| \nabla v \|_{L_2}^2\; \max_{x \in \om} \alpha^{-1}  
  \]
  shows that attempts to control the error
  in an $\alpha$-independent norm, such as  
  $ \| u \|_{\cl X} = \| u \|_{\cl Y} = \| \nabla u \|_{L_2}$,
  or the standard $H^1$-norm, are plagued by difficulties. Here and throughout,
  $\| \cdot \|_{L_2}$ denotes the norm
  of $L_2$ or its Cartesian products.
\end{example}

\begin{example}[FOSLS formulation]
  \label{eg:fosls}
  Given a square-integrable $f\in L_2(\om)$, 
  \eqref{laplaceDP} can equivalently be written in the first-order form using $\alpha$ as 
  \begin{subequations}
    \label{laplaceDP-first-order}
    \begin{align}
      \label{laplaceDP-first-order-1}
      \alpha(x)  \,q + \nabla u
      & =  0
      &&\text{ in } \quad \Omega,
      \\
      \div  q
      & = f
      &&\text{ in } \quad \Omega,
      \\
      u & = 0
      && \text{ on } \quad \partial\Omega.
    \end{align}
  \end{subequations}
  It is useful to introduce a notation for the entire first-order
  operator in~\eqref{laplaceDP-first-order}. 
  \begin{equation}\label{LaplaceDP_OPform}
    A_\alpha (q, u) =
    \big( \alpha  q + \nabla u, \; \div  q \big).
  \end{equation}
  Then~\eqref{laplaceDP-first-order} takes the form
  $ A_\alpha \x = (0, f) $ in $\om$ for the group variable $\x = ( q, u)$.

  Since the weak formulation of \eqref{laplaceDP} allows for $f\in H^{-1}(\om)$ the above assumption requiring $f$ in $L_2(\om)$ seems to entail a restriction. However, recall that every $f\in H^{-1}(\om)$ can be written as $f= f_\circ + \div\,g$ where $f_\circ\in L_2(\Omega)$ and $g\in L_2(\om)^d$. Then the first-order system takes the form $A_\alpha ( q, u) = (g, f_\circ)$, thus covering the same range of source terms as \eqref{laplaceDP}. It is only  for the sake of convenience that we continue to work with $f_\circ=f, g=0$.

The choice of a pair of trial and test spaces $\cX,\cY$, that gives rise to a stable weak formulation of \eqref{LaplaceDP_OPform}, is now less canonical.
Equation~\eqref{laplaceDP-first-order-1} defines the ``flux variable''
$q$ in terms of derivatives of the ``primal variable'' $u$. The former
lies in $\Hdiv$, the space of vector fields whose components
and divergence are in $L_2(\om)$. The latter lies in
$\Ho$ in accordance with the weak formulation of \eqref{laplaceDP}.
Thus, we view $A_\alpha$ as a mapping
\begin{equation}\label{Adef}
  A_\alpha : \Hdiv  \times \Ho \longrightarrow [L_2(\Omega)]^d \times L_2(\Omega),
\end{equation}
and construct a weak formulation of $ A_\alpha \x = (0, f) $ by
multiplying both sides by $\y = (r, w)$ in
$[L_2(\Omega)]^d \times L_2(\Omega)$ and integrating over $\om$ to
arrive at the setting~\eqref{eq:2} with
\begin{equation}
  \label{eq:5}
  \begin{gathered}
    \cX = \Hdiv \times \Ho,
    \qquad
    \cY = L_2(\om) \times [L_2(\om)]^d,
    \\ 
    b_\alpha(\x, \y) = \int_\om A_\alpha \x \cdot \y\; dx,
    \qquad
    \ell(\y) = \int_\om F \cdot \y\; dx,
  \end{gathered}
\end{equation}
where $F = (0, f)$. Equation~\eqref{eq:2} in this case can equivalently be written
in component form, with  $\x = (q, u)$ and $\y = (r, w)$ as 
\begin{equation}
    \label{weakA1}
    \int_\om (\alpha q+\nabla u) \cdot r + ({\rm div}\,q -f) w\,dx = 0,
    \quad w\in L_2(\om),\, r\in [L_2(\om)]^d.
\end{equation}
This is the weak formulation underlying the FOSLS method.  It is well known that~\eqref{Adef} is a continuous bijection and that~\eqref{weakA1}  is uniquely solvable for every
$\alpha$ that is uniformly positive and bounded in
$\bar\om$. Moreover, it can be shown to yield the same solution as the
canonical weak formulation of \eqref{laplaceDP} in the sense that $u$
is the weak solution of \eqref{laplaceDP} if and only if
$q = -a(x)\nabla u$ and $u$ solves \eqref{weakA1}---see
\cite{BDO,FOSLS_coercivity}.

The  error-residual equation~\eqref{eq:4} now takes the form
\begin{equation}
\label{eq:6}
\| \x - \w \|_{b_\alpha} = \| F - A_\alpha \w \|_{L_2},\qquad \w \in \Hdiv \times \Ho.
\end{equation}
The  interest in
the FOSLS formulation~\eqref{eq:5} is now evident: all difficulties
with the issue~{\em (i)} disappears since the norm of 
$\cl Y = \cl Y'$ is just the $L_2$-norm which is simple to evaluate.
In regards to issue {\em
  (ii)}, note that~\eqref{eq:5} implies
$  \| \w \|_{b_\alpha} = \| A_\alpha \w\|_{L_2}.    
$
Unfortunately, this
is an $\alpha$-dependent norm, which cannot be easily controlled by
$\alpha$-independent norms with $\alpha$-robust equivalence
constants---see \eqref{eq:ellipticity_condition} below and
Section~\ref{sec:fosls-loss-function}. We will see practical
manifestations of this difficulty in our experimental studies in
Section~\ref{sec:dpg-loss-function}.
\end{example}

\begin{example}[Ultraweak formulation]
  \label{eg:uw}
  This formulation is also based on the first-order
  reformulation~\eqref{laplaceDP-first-order}.  Integrating
  \eqref{weakA1} by parts,
  $$
  \int_\om \alpha q \cdot r - u {\rm div}\,r -q\nabla w- f w\,dx
  + \int_{\partial\om} (n\cdot r)w + (n\cdot q)w\,ds =0.
  $$  
  Taking $w\in \Ho$ and $r\in \Hdiv$ as test functions, we arrive at
  the problem of finding $(q,u)\in [L_2(\om)]^d \times L_2(\om)$ such
  that
  \begin{equation}
    \label{weakA2}
    \int_\om q \cdot (\alpha r -\nabla w) - u\, {\rm div}\,r - f w\,dx=0.
  \end{equation}
  What were previously essential boundary conditions are now natural ones.
  Introducing the adjoint operator 
  \begin{equation}
    \label{eq:Aadjoint}
    A_\alpha^* (q, u) =
    \big(
    \alpha q - \nabla u,  \; -\div q
    \big),
  \end{equation}
  we may rewrite~\eqref{weakA2} as~\eqref{eq:2} with the following settings: 
  \begin{equation}
    \label{eq:7}
    \begin{gathered}
      \cX = [L_2(\om)]^d \times L_2(\om), 
      \qquad
      \cY =  \Hdiv \times \Ho,
      \\ 
      b_\alpha(\x, \y) = \int_\om  \x \cdot A_\alpha^* \y\; dx,
      \qquad
      \ell(\y) = \int_\om F \cdot \y\; dx,
    \end{gathered}
  \end{equation}
  with $\x = (q, u) $ and $\y = (r, w)$.  This is often called an
  ``ultraweak formulation'' since all the solution components need
  only be regular enough to be in $L_2(\om)$. It is well
  posed~\cite{CarstDemkoGopal16}. If the  ultraweak solution
  belongs to $\Hdiv\times \Ho$  then it agrees with the
  solution to \eqref{weakA1}.

  Consider the error-residual equation~\eqref{eq:4} for the
  formulation~\eqref{eq:7}. It is well known that
  $A_\alpha^* : \Hdiv \times \Ho \to [L_2(\om)]^d \times L_2(\om)  $
  is  a continuous bijection (just as \eqref{Adef} is).  Hence $\cl Y$ may
  equivalently be normed by
  \begin{equation}
    \label{eq:opti-test-norm}
    \|\y\|_{\cl Y_\alpha} : = \| A_\alpha^* \y\|_{L_2}, \qquad
    \y \in \Hdiv \times \Ho.    
  \end{equation}
  Such norms have been used under the name ``optimal test norms''
  previously~\cite{DHSW,DPGacta,ZitelMugaDemko11}.  Let $\cl Y_\alpha$
  denote $\Hdiv \times \Ho$ endowed with the above
  norm~$\|\cdot\|_{\cl Y_\alpha}.$ Revising $\cl Y$ in~\eqref{eq:7} to
  $\cl Y_\alpha,$ the norm in~\eqref{eq:3} becomes
  \[
    \| \w \|_{b_\alpha} = \sup_{0 \ne \y \in \cl Y_\alpha}
    \frac{ b_\alpha( \w, \y) }{ \| A_\alpha^*\y\|_{L_2} } = \| \w \|_{L_2},
  \]
  so~\eqref{eq:4} implies 
  \begin{equation}
    \label{eq:uw-error-residual}
    \| \x - \y \|_{L_2} = \| \ell - B_\alpha \y \|_{\cl Y_\alpha'}.
  \end{equation}
  An attractive feature of~\eqref{eq:uw-error-residual} is that it has
  no $\alpha$-dependent constants and the norm on the left hand
  side is the $L_2$-norm, which is completely independent of $\alpha$.
  Thus the previously
  mentioned issue~{\em (ii)} disappears. However, the norm on the
  right hand side of~\eqref{eq:uw-error-residual} is a non-trivial
  dual norm whose numerical evaluation is an issue, i.e., issue {\em
    (i)} remains. This can be fixed to some extent using the so-called
  ultraweak DPG method, as we shall see in
  Section~\ref{sec:param-robust}.
\end{example}

\subsection{Parameter-to-solution map}\label{ssec:ptsm}

Since handling
arbitrary variations in the function parameter $\alpha$ is beyond the scope of this paper,
we first restrict ourselves to a finite-dimensional parameter set for
$\alpha$ by subdividing $\om$ into finitely many subdomains $\om_i$,
$i=1, \dots, m_\alpha$, and requiring $\alpha$ to be a constant in each
$\om_i$. Namely, using $\chi_i(x)$ to denote the indicator function of
$\om_i$, we define the parameter domain to be
\begin{equation}
  \label{eq:Qdefn}
  \mathcal{Q} = \left \{\underline{\alpha} + \sum_{i=1}^{m_\alpha} \alpha_i \chi_i(x) \; : \; \alpha_i \geq 0, \underline{\alpha} > 0 \right \}.
\end{equation}
We are interested in approximating the
parameter-to-solution map~$\cl F$ defined by
\begin{equation}
  \label{eq:exact-F}
  \alpha
  \overset{{\cl{F}}}{\longmapsto}
  \text{ the } (q, u) \text{ solving \eqref{laplaceDP-first-order},} 
\end{equation}
i.e., $\cl F(\alpha) =  (q(x,\pp),u(x,\pp))$ is a  function of $x$ and $\pp$ defined on the spatio-parametric domain $\Omega\times\pdom$. In a learning context such approximations result from minimizing appropriate loss functions.
In constructing such approximations, even though $\cl Q$ is a subset of a vector space of dimension $\dim \cl Q = m_\alpha$, the vector space structure is not used. In fact, in many practical problems, the
parameter domain is a bounded subset of $\cl Q$. We keep $\cl Q$ unbounded in our studies here to make the learning task more challenging.

The approximation to $\cl F$ depends on the particular weak formulation of~\eqref{laplaceDP-first-order} we choose to use. Without specifying the formulation at the moment, suppose 
that we are given a basis $\psi_k$, $k=1, \dots, m_h$, for a finite
element space, built using a mesh of maximal element diameter~$h$ and
that an approximation of
$\cl F(\alpha) =  (q(x,\pp),u(x,\pp))$
is found, for each given $\alpha$ in
$\cl Q$, in their span
\begin{equation}
  \label{eq:Xh-defn}
  \Xh = \left\{
    \sum_{k=1}^{m_h} v_k\, \psi_k(x) : v_k \in \bb R \right\}
\end{equation}
by some method.  Then a computable approximation of the exact $\cl F$
in~\eqref{eq:exact-F}, which we denote by $\Fh$, 
can be implemented as a 
map between Euclidean  vector spaces,
$\mathbb{F}: \mathbb{R}^{m_\alpha} \to \mathbb{R}^{m_h}$. The mapping
of functions $\Fh$ is generated from the mapping $\bb F$ of vectors
by
\begin{equation}
  \label{eq:fun-map}
  \Fh\Big( \sum_{i=1}^{m_\alpha} \alpha_i \chi_i(x)\Big)
  = \sum_{k=1}^{m_h} \,[\bb F(\bbalpha)]_k\, \psi_k(x),
  \qquad x \in \om.  
\end{equation}
where $\bbalpha$ is the vector of coefficients of $\alpha \in \cl Q$, i.e.,
$\bbalpha = (\alpha_1, \alpha_2, \dots)$.
We shall construct such approximate mappings $\Fh$ using neural
networks. Two different choices of the finite element space $\Xh$
(and accompanying loss functions for  training the neural networks)
are discussed in Sections~\ref{sec:fosls-loss-function} and~\ref{sec:dpg-loss-function}.

\begin{figure}
  \centering
  \subcaptionbox
  {A parameter function $\alpha \in \cl Q$  with
    $\bbalpha = (1, 2, 3, 4)$,
    made by subdividing a square $\om$ into  4 congruent subdomains ($h=0.1$).
    \label{fig:param-example}
  }[0.45\textwidth]
  {
    \includegraphics[scale=0.26,
    trim={100  90  50  80}, clip]{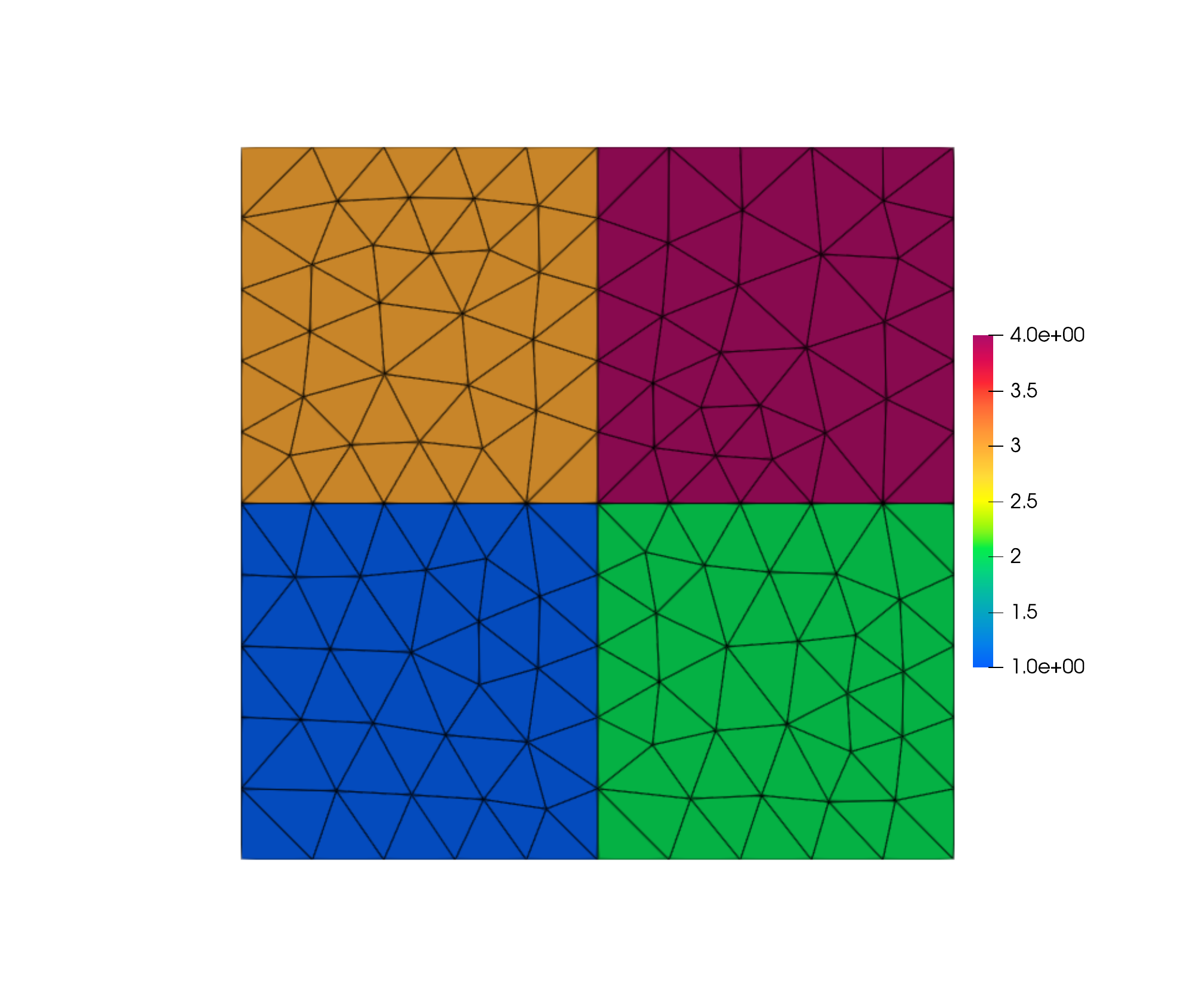}
  }
  \qquad
  \subcaptionbox  
  { \label{fig:param-16}
    A parameter function $\alpha \in \cl Q$,  generated by 
    $\bbalpha = (1, 2, \dots, 16)$,
    made with sixteen congruent subdomains $(h = 0.06)$.
  }[0.45\textwidth]
  {
    \includegraphics[scale=0.26,
    trim={110  90  50  90}, clip]{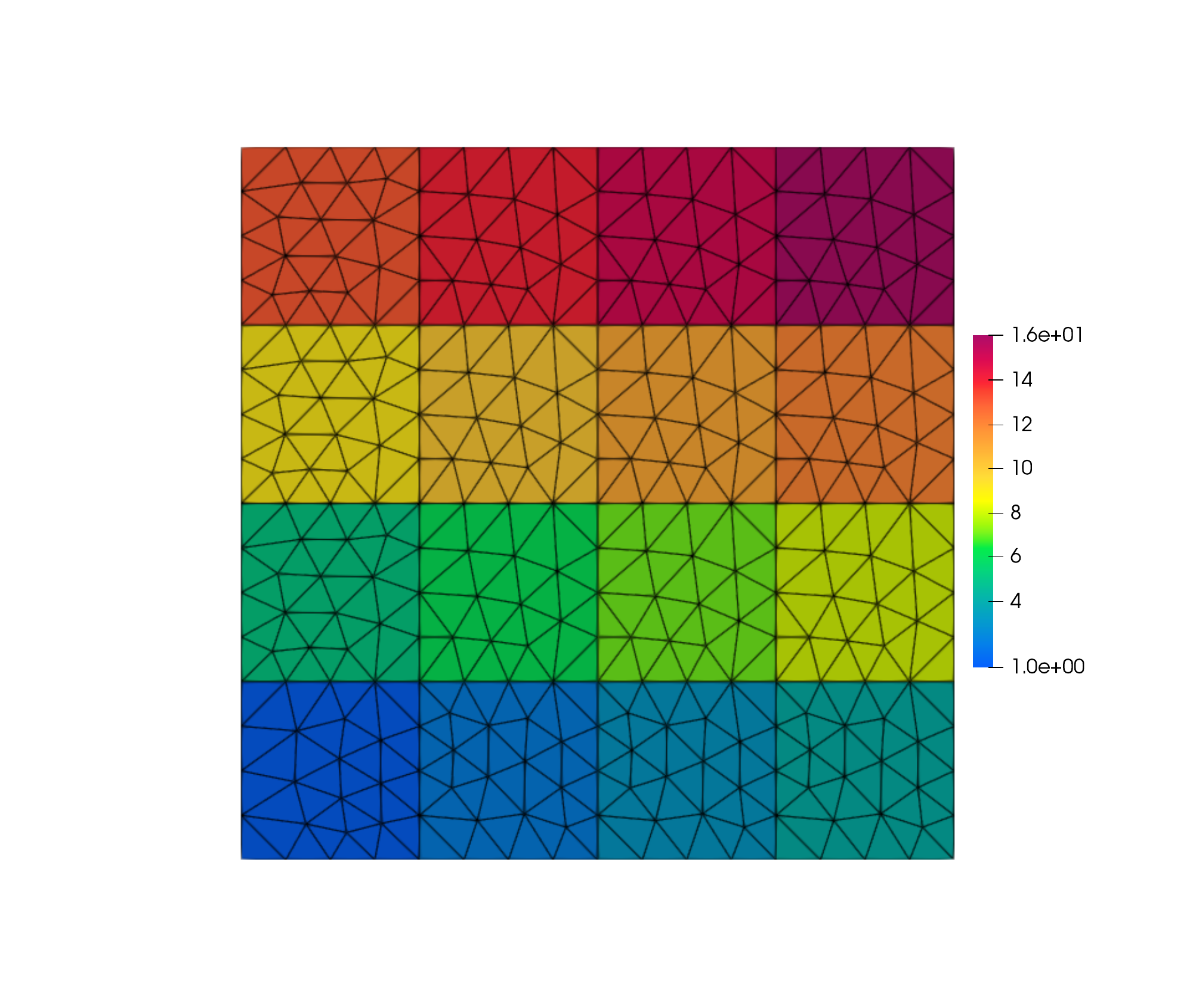}    
  }
  \caption{Example functions in parameter space $\cl Q$ which also show
    the subdomain ordering used.
  }
\end{figure}

\subsection{Neural network (NN) architecture}
\label{ssec:neural-network}
  
To construct the map $\mathbb{F}$ using an NN, we
use a ResNet architecture proposed in
\cite{BDO},
described next. Let $\varsigma$ denote a leaky ReLU function
defined component wise by 
\[
  \varsigma: \mathbb R^N \to \mathbb R^N,\qquad
  \varsigma(y)_i =\max( y_i, 10^{-3} y_i), \qquad y \in \bb R^N,
\]
for some integer $N$.
Next, we select  another integer $r \ll N$, to promote sparsity 
as suggested in \cite{BDO} and
consider 
matrices
$A_{\text{out}} \in \bb R^{m_h \times N},
A_m \in \bb R^{N \times r},
W_m \in \bb R^{r \times N},
A_{\text{in}} \in \bb{R}^{N \times {m_\alpha}},$
and vectors
$b_m \in \bb R^r,
b_{\text{out}} \in \bb{R}^{{m_h}},
b_{\text{in}} \in \bb{R}^{N}. 
$
For $m=1, \dots, l$, 
let
\begin{align*}
  & \Phi_m: \mathbb R^N \to \mathbb R^N, 
  && \Phi_m (z) =
     z 
     + A_m 
     \varsigma\big( W_m
     z 
     +
     b_m
     \big),
  \\
  & L_{\text{out}}: \mathbb R^N \to \mathbb R^{m_h}, 
  && L_{\text{out}}(z) =
     A_{\text{out}}
     z 
     +
     b_{\text{out}},
  \\
  & L_{\text{in}} : \mathbb R^{m_\alpha}\to \mathbb R^N,
  && L_{\text{in}}(p) =
     A_{\text{in}}
     p
     +
     b_{\text{in}},
\end{align*}
for any $z \in \bb R^N $ and $p \in \bb R^{{m_\alpha}}$. Then we approximate
  $\mathbb{F}$ by expressions of the form
  \begin{equation}
    \label{eq:NN-F}
  \mathbb{F}_\theta : \mathbb{R}^{m_\alpha} \to  \mathbb{R}^{m_h}, \qquad
  \mathbb{F}_\theta = L_{\text{out}} \circ \Phi_l \circ \dots \circ \Phi_1 \circ L_{\text{in}}
\end{equation}
(as illustrated in Figure~\ref{fig:NN})
where the subscript $\theta$ represents all the ``trainable'' degrees
of freedom (weights and biases) above, i.e.,
$  \theta = \{A_{\text{in}}, b_{\text{in}}, A_m, W_m, b_m,  b_{\text{out}}, A_{\text{out}}: \; m \in 1, \dots, l\}.$
Let  $\Theta$ denote the set of all such values that  $\theta$ can take.
An appropriate  value of  $\theta$ in $\Theta$
is
found by a training process, as commonly done in the context of
all neural networks.
Let $\Fht : \cl Q \to \Xh$  denote the mapping of functions 
obtained by replacing $\bb F$ by $\bb F_\theta$ in~\eqref{eq:fun-map}.

 \begin{figure}
   \centering
   \tikzset{%
     hidden neuron/.style={
       circle, brown!50!blue,
       draw, thick,
       minimum size=1cm
     },
     input neuron/.style={
       circle, red, 
       draw, thick,
       minimum size=1cm
     },
     output neuron/.style={
       circle, green!70!blue,
       draw,  thick,
       minimum size=1cm
     },
     neuron missing/.style={
       draw=none, 
       scale=4,
       text height=0.333cm,
       execute at begin node=\color{gray}$\vdots$
     },
     neuron missinginput/.style={
       draw=none, 
       scale=4,
       text height=0.333cm,
       execute at begin node=\color{red}$\vdots$
     },
     neuron missinghidden/.style={
       draw=none, 
       scale=4,
       text height=0.333cm,
       execute at begin node=\color{brown!50!blue}$\vdots$
     },
     neuron missingoutput/.style={
       draw=none, 
       scale=4,
       text height=0.333cm,
       execute at begin node=\color{green!70!blue}$\vdots$
     },
     neuron missinghoriz/.style={
       draw=none, 
       scale=4,
       execute at begin node=\color{brown!50!blue}$\dots$
     },
     neuron missinghblank/.style={
       draw=none, 
       scale=4,
       execute at begin node=\color{white}$\dots$
     },
   }

   \begin{tikzpicture}[x=1.5cm, y=1.5cm, >=stealth]

     \foreach \m/\l [count=\y] in {1,2,3,missinginput,4}
     \node [input neuron/.try, neuron \m/.try] (input-\m) at (0,2.5-\y) {};

     \foreach \m [count=\y] in {1,2,missinghidden,3}
     \node [hidden neuron/.try, neuron \m/.try] (hidden-\m) at (2,3.5-\y*1.6) {};

     \foreach \m [count=\y] in {missinghoriz,missinghblank,missinghblank,missinghoriz}
     \node [hidden neuron/.try, neuron \m/.try] (hiddenh-\m) at (3.2,3.5-\y*1.6) {};

     \foreach \m [count=\y] in {1,2,missinghidden,3}
     \node [hidden neuron/.try, neuron \m/.try] (hidden2-\m) at (4.4,3.5-\y*1.6) {};

     \foreach \m [count=\y] in {1,2,missingoutput,missingoutput,3}
     \node [output neuron/.try, neuron \m/.try ] (output-\m) at (6,4.5-\y*1.65) {};

     \foreach \l [count=\i] in {1,2,3,{{\dim\mathcal{Q}}}}
     \draw [<-, red] (input-\i) -- ++(-1,0)
     node [red, above, midway] {$\alpha_{\l}$};

     \foreach \l [count=\i] in {1,2,{\dim\mathcal{X}_h}}
     \draw [->, green!70!blue] (output-\i) -- ++(1,0)
     node [above, midway] {\textcolor{green!70!blue}{$v_{\l}$}};

     \foreach \i in {1,...,4} 
     \foreach \j in {1,...,3}
     \draw [->, red] (input-\i) -- (hidden-\j);

     \foreach \i in {1,...,3}
     \foreach \j in {1,...,3}
     \draw [->,brown!50!blue] (hidden2-\i) -- (output-\j);

     \foreach \i in {1,...,3}
     \foreach \j in {1,...,3} {
       \draw [->,brown!50!blue] (hidden-\i)--($(hidden-\i)!0.8cm!(hidden2-\j)$);
       \draw [<-,brown!50!blue] (hidden2-\i)--($(hidden2-\i)!0.8cm!(hidden-\j)$);
     }
     
     \node [align=center, above] at (0,3.5)
     {\textcolor{red}{Input}
       \\ \textcolor{red}{layer}
     };
     \node [align=center, above] at (3,3.5)
     {\textcolor{brown!50!blue}{Hidden $l$ layers}
       \\
       \textcolor{brown!50!blue}{of $N$ nodes each}
     };
     \node [align=center, above] at (6,3.5)
     {\textcolor{green!70!blue}{Output}
       \\
       \textcolor{green!70!blue}{layer}
     };

   \end{tikzpicture}
   \caption{Schematic of the neural network architecture of \eqref{eq:NN-F}.
     The input layer ($L_{\text{in}}$) receives parameter values  $\alpha_i$,
     the output layer ($L_{\text{out}}$)
     produces coefficients $v_k$ of finite element basis expansion of a
     solution prediction, and the hidden layers involve the
     transformations~$\Phi_m$.}
   \label{fig:NN}
 \end{figure}

The training process requires a loss function $\cl L \equiv \cl L_\alpha$ that estimates
the error in any approximation generated by the NN at a given parameter value $\pp$, 
\begin{equation*}
   \mathcal{L}: \Xh \to \mathbb{R}.   
\end{equation*}
The calculation of $\cl L(\w)$, for any $\w \in \Xh$, at a given
parameter value $\pp$, can depend on $\alpha$ even when $\w$ is
independent of $\pp$. Although a proper notation indicating so would
be $\cl L_\alpha$, we dispense with the subscript to lighten the
notation.
The training  is carried out through a stochastic minimization process for
\begin{equation}
\label{cL}
\theta^* = \argmin_{\theta \in \Theta} \left ( \frac{1}{M} \sum_{\alpha \in \mathcal{Q}_M} \mathcal{L}(\Fht(\alpha)) \right )
\end{equation}
where $\mathcal{Q}_M$ is a 
collection of $M$ training samples, drawn from a random distribution of functions in 
$\mathcal{Q}$ after endowing it with a probability measure~$\mu$. Here, in terms of computational vectors $\bbalpha$,
the summand above is 
\[
  \mathcal{L}(\Fht(\alpha))
  =
  \mathcal{L}\left(
  \sum_{k=1}^{m_h} \,[\bb F_\theta(\bbalpha)]_{k}\, \psi_k(x)
  \right)
\]
where $\bb F_\theta$ is as defined by the NN in \eqref{eq:NN-F}.

In machine learning jargon, the minimization in~\eqref{cL} is referred
to as {\em empirical risk minimization}.  When a scheme produces
candidate solution approximations $\w(\alpha) \in \Xh$ for each
$\alpha$ in the probability space $\cl Q$ with measure $\mu$, its
``ideal risk'' is defined as the expectation of the random variable
$\cl L \circ \w$, i.e.,
\begin{equation}
  \label{idealloss}
  \LL(\w;\,\pdom):= \int_\pdom \cl L(\w(\pp))\;d\mu(\pp)
  = \mathbb{E}_\mu\big[\cl L(\w(\pp))\big].
\end{equation}
The empirical risk is the analogue computed from the $M$ random samples in $\cl Q_M$
drawn from $\cl Q$, defined by 
\begin{equation}
  \label{eq:9}
  \LL(\w, \cl Q_M):=  \frac{1}{M} \sum_{\alpha \in \mathcal{Q}_M} \mathcal{L}(\w(\alpha)).
\end{equation}
It can be viewed as a Monte Carlo approximation to the ideal risk.
Viewing the NN parameter-to-solution map $\Fht$ as a function on the
probability space $\cl Q$, the ideal risk of the NN model is
$\LL(\Fht;\,\pdom)$.  The objective function being minimized
in~\eqref{cL}, equaling $\LL(\Fht, \cl Q_M)$, is then the empirical
risk, a Monte Carlo approximation to the ideal risk.
Desirable criteria while choosing $\mathcal{L}$ are described in \S\ref{ssec:vc}.

\subsection{The ``lifted'' problem}\label{ssec:lifted}

To quantify the accuracy of outcomes of the trained model
${\cl{F}_{h,\theta^*}}$ from~\eqref{cL} in relation to the exact
$\cF$, as a function of $x\in \om$ and $\pp\in \pdom$, one needs an
appropriate model-compliant ``continuous reference metric''.  As
suggested in \cite{BDO} and in agreement with \eqref{idealloss}, we
consider the ``lifted'' variational problem for a group variable
$\x (x, \alpha)$ containing the flux and solution components
as functions on the spatio-parametric domain $\Omega\times\pdom$.
For each $\alpha \in \cl Q$ the solution $\x(\cdot, \alpha)$ is in some trial Hilbert
space~$\cX$ and satisfies a well-posed formulation~\eqref{eq:2} with some test space $\cl Y$. Generalizing over all parameters  using
$\X:= L_2(\pdom;\cX)$, the lifted formulation seeks
 $\x \in \X$ such that for all
$ \mathfrak{v} \in \Y:= L_2(\pdom;(\cY_\pp)_{\pp\in\pdom})$,
\begin{equation}
    \label{lifted}
    b(\x, \y):= \int_\pdom b_\pp (\x, \y)\; d\mu(\pp)
    = \int_\pdom \ell (\y)\; d\mu(\pp)
\end{equation}
where 
$\mu$ is a given probability measure on $\pdom$.
When the test-spaces $\cY_\pp$  depend on $\pp$, the notation $L_2(\pdom;(\cY_\pp)_{\pp\in\pdom})$ indicates that,  under mild measurability assumptions on the bilinear forms $b_\pp(\cdot,\cdot)$, the space  $\Y$ is indeed a well-defined Hilbert space, a so-called direct integral: see~\cite{BDO}. Moreover, it is shown there that \eqref{lifted} is well-posed if and only if the individual fiber problems~\eqref{eq:2} are
well-posed (in the sense of the Babu\v{s}ka's theorem)
{\em uniformly} in $\pp\in\pdom$.   The reference norm $\|\cdot\|_{\X}$ for the continuous problem is reasonable because it measures
errors
\begin{equation}
    \label{expectation}
\|\x-\w\|^2_{\X}=\int_\pdom \|\x(\cdot,\pp)-\w(\cdot,\pp)\|^2_\cX\;d\mu(\pp)= \mathbb{E}_\mu\big[\|\x(\cdot,\pp)-\w(\cdot,\pp)\|^2_\cX\big]
\end{equation}
as an {\em expectation} over $\pdom$.
This underlines the stochastic nature of the learning problem.

A key consideration  in our approach is the requirement that
\eqref{expectation} be {uniformly proportional} to
the loss \eqref{idealloss} when $\w =\Fht$ ranges over the hypothesis class.
Since we will not be able to assess exact values of \eqref{idealloss}
over the entire parameter range, we use a large (random) test set
$\widehat\pdom_{\rm test}\subset\pdom$ to gauge accuracy.  The errors
in the ideal reference metric $\|\cdot\|_\X$ of \eqref{expectation} have two sources.  The
first is the deviation between the ideal loss and its empirical
counterpart
$ \LL( \Fht, \widehat\pdom_{\rm test})$.
The
second is the error in the {\em discrete Bochner norm}  $\|\x- \Fht\|_{\X(\widehat\pdom_{\rm test})}$ where
$\X(\widehat\pdom_{\rm test}):= \ell_2(\widehat\pdom_{\rm test};\cl X)$.  The former is
often referred to as the {\em generalization error} and depends (among
other things) on the size of sufficiently large test set
$\widehat\pdom_{\rm test}\subset\pdom$.  We do not address the associated 
stochastic estimation concepts (see e.g.~\cite{BCDD})
but concentrate on estimating the  empirical risk
$\|\x- \cl F_{h,\theta}\|_{\X(\widehat\pdom_{\rm test})}$
from corresponding
computable losses
$ \LL( \Fht, \widehat\pdom_{\rm test})$.

\subsection{Variational correctness, reliability, and efficiency}\label{ssec:vc}

As we have seen, classical boundary value problems like~\eqref{laplaceDP-first-order}
admit many different variational formulations of the form~\eqref{eq:2}.  The effectiveness of an approximation $\w \in \cl X$
to the exact solution $\x$ in $\cl X$
can be measured by the norm $\| \x - \w\|_{\cl X}$.  Conforming
finite element methods are based on a variational formulation and use
a finite element subspace $\Xh \subset \cl X$.
Following~\cite{BDO}, we
pursue loss functions that track the error, i.e., $\cl L (\w)$ is
bounded above and below by the true error $\| \x - \w\|_{\cl X},$ scaled
by some $\w$-independent constants, for all candidate
approximations~$\w$.  In the finite element literature, such estimates
are well known for {\it a posteriori} ``error estimators''
which are typically based on {residuals}. Then they do not require knowing the exact solution $\x$ but involve problem data such as source terms whose evaluation may also require some care. We now
borrow the standard terminology for finite element error estimators
into NN loss functions and specify what we mean by their variational
correctness.

\begin{definition}
  \label{def:vc}
  For some class of source functions~$f$ representable in computations, 
  the loss function $\cl L: \Xh \to \bb R$ is said to
  be {\em reliable} in $\Xh$ if there is a $C_1>0$ such that 
  \[
    C_1 \| \x - \w \|^2_{\cl X}  \le \cl L(\w)
  \]
  for all $\w \in \Xh$. Here $\x$ is the exact solution of \eqref{eq:2}.
  If there is a $C_2>0$ such that
  \[
    \cl L(\w) \le C_2\| \x - \w\|^2_{\cl X}
  \]
  for all $\w \in \Xh$, then the loss function $\cl L$ is said to be
  {\em efficient}. A loss function that is both reliable and efficient
  is said to be {\em variationally correct} in $\Xh$.
\end{definition}

When a prediction $\w\in \Xh$ is of the form \eqref{eq:fun-map}, i.e., the expansion coefficients are NNs, the notorious uncertainty of optimization success renders {\it a priori} estimates based on the expressive power of NNs useless. More useful  information on the fidelity of an NN-based prediction of the exact solution $\x$ is drawn from the value of
the loss $\cl L (\w)$.   Then the two-sided bounds of a variationally correct loss
function,
\begin{equation}
  \label{eq:VC}
  C_1 \| \x - \w \|_{\cl X}^2 \le \cl L(\w) \le C_2 \| \x - \w \|_{\cl X}^2, \qquad
  \w \in \Xh,
\end{equation}
show that $\cl L(\w)$ provides an ``error certificate,'' up to the
values of the constants $C_1$ and $C_2$. 

In Definition~\ref{def:vc}, the solution $\x \equiv \x(x, \alpha)$,
the candidate approximation $\w \equiv \w(x, \alpha),$ as well as the
constants $C_1\equiv C_1(\alpha), C_2\equiv C_2(\alpha)$, all depend
on $\alpha$. To indicate this, we often call $\cl L$ the ``fiber loss'' (of the fiber problem) at $\alpha$. Let $\alpha$ vary  in some selected parameter domain $\Qt$ (which is typically a strict compact subset of the $\cl Q$ in~\eqref{eq:Qdefn}). We say that a variationally correct loss function $\cl L$ is {\em robust over $\Qt$} if the constants $C_1$ and $C_2$ do not depend on the parameter  $\alpha$ as it  ranges over  $ \Qt$. Arguably a useful concept of parameter robustness should also
constrain $\| \cdot \|_{\cl X}$ to include a parameter-independent part, but this is hard to prescribe abstractly.

When $\cl L$ is robust over $\Qt$,  it is shown in~\cite{BDO} that \eqref{eq:VC} holds
uniformly for $\pp\in\Qt$ if and only if
\begin{equation}
  \label{eq:VCX}
 C_1 \| \x - \w \|_{\X}^2 \le \LL(\w;\Qt) \le C_2 \| \x - \w \|_{\X}^2, \qquad
  \w \in \X,
\end{equation}  
where
$\LL$ is as in \eqref{idealloss}
and $\|\cdot \|_{\X}$ denotes the Bochner norm of  $\X = L_2(\Qt;\Xh)$, i.e., if
and only if the ``lifted loss''  $\LL$ is variationally correct in $\X$.
Moreover, uniform validity of \eqref{eq:VC} over $\Qt$  implies the two-sided bounds
\begin{equation}
  \label{eq:VCX}
 C_1 \| \x - \w \|_{\X(\Qt_M)}^2 \le \LL(\w;\Qt_M) \le C_2 \| \x - \w \|_{\X(\Qt_M)}^2, \qquad
  \w \in \X(\Qt_M) ,
\end{equation}  
for any
$\Qt_M\subset \Qt$ of $M$ samples, and
for the  corresponding computable quantities, involving the discrete
Bochner norm of $\X(\Qt_M)= \ell_2(\Qt_M;\Xh)$ and the empirical risk
defined in \eqref{eq:9}. Thus, variational correctness of $\cl L$ over
$\Xh$ and robustness of $\cl L$ over $\Qt$ ensures variational
correctness of the lifted loss $\LL$ over $\X = L_2(\Qt, \Xh)$ and
ensures that the computable $ \LL(\w;\Qt_M)$ provides rigorous error
certificates over {\em any} test set of random samples
$\Qt_M\subset \Qt$. Inference on the continuous counterpart
$ \LL(\w;\Qt)$ is then a matter of statistical estimation when $M$
increases.

Two-sided inequalities like~\eqref{eq:VC} have been developed extensively
for error estimators in the finite element literature~\cite{AisnwOden00}.  When the
source $f$ is not represented exactly in the finite element
calculations, one can only expect a modified version of~\eqref{eq:VC}
to hold with additional data representation errors (often also called
data oscillation terms). In our numerical studies  we fix $f$ to a simple function
(without the need for such terms) to focus exclusively on $\alpha$
dependence. While some finite element error estimators admit 
bounds like~\eqref{eq:VC} only when $\w$ is a Galerkin approximation, note that we are requiring~\eqref{eq:VC} to hold for any $\w \in \Xh$.

In the next two sections, we proceed to discuss two variationally
correct loss functions $\cl L$ compared in this paper. They are based on the variational formulations \eqref{weakA1} and \eqref{weakA2}. In both cases,
we will need the following inequality to establish variational
correctness: There exists (an $\alpha$-dependent) $c_\alpha>0$ such
that
\begin{equation}
\label{eq:A-bdd-below}
\| (r, w)\|_{L_2} \le c_\alpha \| A_\alpha (r, w) \|_{L_2},
\qquad
(r, w) \in \Hdiv \times \Ho.
\end{equation}
Clearly, this inequality is independent of the finite element
discretization and the specifics (including the norms and the spaces)
of different variational formulations used to
treat~\eqref{laplaceDP-first-order}.  The
inequality~\eqref{eq:A-bdd-below} can be proved using the theory of
mixed systems~\cite[Ch.~II, Prop.~1.3]{BrezzForti91}. A proof can be
found in many places, such as~\cite{BocheGunzb09a},
\cite[Theorem~3.1]{FOSLS_coercivity}, or
\cite[Lemma~4.4]{DemkoGopal11a}, and it holds provided that $\alpha$ 
satisfies the so called ellipticity condition, i.e., there
exists two finite positive constants
$\alpha_{\text{min}}, \alpha_{\text{max}} > 0$ such that
\begin{equation}
    \label{eq:ellipticity_condition}
    \alpha_{\text{min}} \le \alpha(x) \le \alpha_{\text{max}} 
    \quad \text{ a.e. } x \in \Omega.
\end{equation}
Then one can make explicit the $\alpha$-dependence 
of $c_\alpha$, namely, 
there exists~\cite[Lemma~4.4]{DemkoGopal11a}
some $\alpha$-independent constant 
$c_0 >0$ such that
\begin{equation}
  \label{eq:calpha}  
  c_\alpha = c_0 \left(1 +
    \frac{\max(1, \alpha_{\text{max}}) }{\alpha_{\text{min}}}
  \right).
\end{equation}
One often also uses the graph norm generated by the operator
$A_\alpha$, namely 
$( \| \w \|_{L_2}^2 + \| A_\alpha \w \|_{L_2}^2)^{1/2}, $ for any
$\w = (r, w) \in \Hdiv \times H^1.$ By triangle inequality,
\begin{align*}
  \| \nabla w \|_{L_2} + \|\div r \|_{L_2} 
  & \le
  \| \alpha r + \nabla w \|_{L_2}
  + \| \alpha r \|_{L_2}
  + \| \div r \|_{L_2} 
  \\
  & \le \| A_\alpha (r, w) \|_{L_2}
  +{\overline\alpha} \|  r \|_{L_2},
\end{align*}
showing that
\begin{equation}
  \label{eq:graph-bound}
  \| \w \|_{\Hdiv \times H^1}^2 \lesssim
  \| \w \|_{L_2}^2 + \| A_\alpha \w \|_{L_2}^2.  
\end{equation}
Here and throughout, when two quantities $A$ and $B$ admit a bound
$A \le C B$ for some constant $C>0$ which may depend on $\alpha$ but
is independent of~$h$, we write $A \lesssim B$, and we write
$A \sim B$ to mean that both $A \lesssim B$ and $B \lesssim A$
hold. It is  easy to see that
\begin{equation}
  \label{eq:Ydpg-equiv}
  \| \w \|_{L_2}^2 + \| A_\alpha \w \|_{L_2}^2
  \;\sim\; \| \w \|_{\Hdiv \times H^1}^2,
  \qquad \w \in \Hdiv \times H^1,
\end{equation}
since the reverse bound of~\eqref{eq:graph-bound} can also be
established using the triangle inequality.

\section{The FOSLS loss function}
\label{sec:fosls-loss-function}

\subsection{The FOSLS solution}

We consider the FOSLS method~\cite{FOSLS_coercivity}, based on \eqref{eq:5} or 
\eqref{weakA1}, in its lowest
order form. Since $A_\alpha$ as an operator in~\eqref{Adef} is onto, we can symmetrize the variational formulation in \eqref{eq:5} by writing every test function $\y$  there as $A_\alpha \w$ for some trial function $\w$ to get the  following equivalent FOSLS formulation: Find $\x \in \Xfosls$ satisfying
\begin{subequations}
  \label{eq:fosls-formulation}
  \begin{equation}
    \label{eq:fosls-formulation-eq}
      b_\alpha^\fosls(\x, \w) = \ell^\fosls(\w), \qquad \forall \w \in \Xfosls,
  \end{equation}
where
\begin{gather}
  \Xfosls = \Hdiv \times \Ho, \qquad
  \| (r, w) \|_{\Xfosls}^2 = \| r \|_{\Hdiv}^2 + \| w \|_{\Ho}^2.
  \\
  b_\alpha^\fosls(\x, \w) := (A_\alpha \x, A_\alpha \w)_{L_2}, \qquad
        \ell^\fosls(\w) := (F, A_\alpha \w)_{L_2}.
\end{gather}
\end{subequations}
Here  $F = (0, f)$ and $(\cdot, \cdot)_{L_2}$ denotes the inner
product in $L_2$ or its Cartesian products.

The lowest order FOSLS loss  sets the finite element space $\Xh$ (in \eqref{eq:Xh-defn}) to a subspace of
$\Xfosls$ described now.  Let $\oh$ be a conforming shape-regular simplicial
finite element mesh of $\om$, where $h$ denotes the maximal diameter
of elements in the mesh. Let $\cl{P}_k$ denote the space of
piecewise polynomials of degree at most $k$ on each element of
$\Omega_h$ (without any inter-element
continuity conditions). Then the lowest order Raviart-Thomas
subspace~\cite{RaviaThoma77} of $\Hdiv$ is
\begin{subequations}
\label{eq:RT0-Lag1}  
\begin{equation}
  \label{eq:RT0}
  \RT
= \left [ \mathcal{P}_0^d + x
  \mathcal{P}_0 \right ] \cap \Hdiv.  
\end{equation}
Using also the lowest order Lagrange space on the same mesh,
\begin{equation}
  \label{eq:Lag1}
  \cl U_1 = \cl{P}_1 \cap {\Ho}  
\end{equation}
\end{subequations}
the FOSLS method sets the finite element space $\Xh$ to
\[
  {\Xhfosls} = \RT \times \cl U_1.
\]
Given a diffusion parameter $\alpha$, the FOSLS approximation to the
exact solution $\x=\x(\cdot,\pp) = (q(\cdot,\pp), u(\cdot,\pp))$ in $\Hdiv \times {\Ho}$ satisfying \eqref{laplaceDP-first-order}
is the unique $\x_h \in {\Xhfosls}$ satisfying
\begin{equation}
  \label{eq:FOSLS-method}
  (A_\alpha  \xfosls, A_\alpha  \w)_{L_2} = ( F, A_{\alpha}  \w)_{L_2},
  \qquad  \text{ for all }  \w \in   {\Xhfosls},
\end{equation}
where $F = (0, f)$ and $A_\alpha$ is as in~\eqref{Adef}. The inner
product in $L_2$ or its Cartesian products is denoted above and
throughout by $(\cdot, \cdot)_{L_2}$ and the corresponding norm by
$\| \cdot \|_{L_2}$.  The evaluation of the FOSLS
parameter-to-solution map for each given $\pp$
\[
  \alpha \overset{\ffosls}{\longmapsto} \xfosls(\pp), \qquad \ffosls(\alpha) = \xfosls
\]
can be obtained by solving~\eqref{eq:FOSLS-method}.
Note that the set of equations contained in~\eqref{eq:FOSLS-method} for
each test function is the same as the set of normal equations for the 
minimization problem
\[
  \min_{ \w \in {\Xhfosls}} \|A_\alpha  \w -  F\|^2_{L_2}.
\]
Thus, for each $\pp\in \pdom$, the FOSLS solution can be characterized as the following unique minimizer:
\[
  \ffosls(\alpha) = \xfosls = \argmin_{ \w \in {\Xhfosls}} \|A_\alpha  \w -  F\|^2_{L_2}.
\]

\subsection{Variational correctness of the FOSLS loss function}

For the fiber problem at any $\alpha \in \cl Q$, 
the ``FOSLS loss function'' is defined by
\begin{equation}
  \label{eq:fosls-loss}
  \Lfosls(\w)
  := \|A_\alpha  \w -  F\|^2_{L_2}, \qquad \w \in {\Xhfosls}.
\end{equation}
The FOSLS parameter-to-solution map is found by performing the
minimization in~\eqref{cL} with $\cl L$ set to $\Lfosls$
and $\Xh$ set to $\Xhfosls$.
As previously mentioned, 
we refer to this  residual in \eqref{eq:fosls-loss}
as ``fiber loss.’’ Note that it explicitly depends on  $\alpha$ through the
operator $A_\alpha$ (even though to avoid notational clutter we have chosen to write $\Lfosls$ instead of $\Lfosls_\alpha$).
The next proposition is essentially contained in
\cite{FOSLS_coercivity} but we include a short proof for completeness.

\begin{proposition}
  \label{prop:Lfosls_VC}
  For each $\pp\in\pdom$,  $\Lbdo$ is variationally correct in
  $\Xhfosls$, i.e., there are mesh-independent (but $\alpha$-dependent)
  constants $C_1, C_2>0$
  such that
  \begin{equation}
    \label{eq:fosls-2-sided}
    C_1\| \x(\alpha) - \w \|_{\Xfosls}^2
    \le \Lbdo(\w) \le C_2 \| \x(\alpha) - \w \|_{\Xfosls}^2,
  \end{equation}
  for all $\w \in \Xfosls.$ Here
  $\x(\pp) = (q(\pp), u(\pp))\in \Xfosls $ is the exact solution
  of~\eqref{laplaceDP-first-order} found using \eqref{weakA1} or
  equivalently \eqref{eq:fosls-formulation}.
\end{proposition}
\begin{proof}
  First note that 
    Since $\x=\x(\pp)$ is the exact solution,  $F = A_\alpha \x$ and
    $\Lbdo( \w) = \| A_\alpha (\w - \x) \|^2_{L_2}$ for any $\w$ in
    $\Xhfosls$. Then~\eqref{eq:A-bdd-below} implies that
    \[
      \| \x - \w\|_{L_2}^2 + \| A_\alpha( \x-\w) \|_{L_2}^2
      \le
      (c^2_\alpha+1) \, \| A_\alpha(\x - \w) \|_{L_2}^2
      =
      (c^2_\alpha+1) \,  \Lbdo(\w).
    \]
    Hence by \eqref{eq:Ydpg-equiv}, 
    \[
      \| \x - \w\|^2_{\Xfosls} \lesssim  \,\Lbdo( \w)
    \]
    for all $\w \in \Xfosls$, showing reliability of the FOSLS loss
    function (the lower inequality of \eqref{eq:fosls-2-sided}).
    Furthermore, a simple application  of triangle inequality
    shows that
    \begin{equation}
    \label{eq:A-bdd-above}
        \|A_\pp\w\|_{L_2} \le \,\max \{ 1, \alpha_{\text{max}} \} \, \|\w\|_{\Xfosls},
      \end{equation}
      for all $\w \in \Xfosls.$
      Thus, $\Lbdo$ is also efficient.
\end{proof}

Although FOSLS loss is variationally correct, it is not robust over
the unbounded parameter domain $\cl Q$ in \eqref{eq:Qdefn}.  Let $a$
and $b$ be two positive constants and consider the bounded parameter
domain
\begin{equation}
  \label{eq:Qab}
  \cl Q^{[a,b]} = \{ \alpha \in \cl Q: a \le  \min_{x \in \om} \alpha(x), \;
  \max_{x \in \om } \alpha(x) \le  b \}.  
\end{equation}
From the proof of Proposition~\ref{prop:Lfosls_VC} it is clear that
the FOSLS loss is robust over the bounded parameter domain
$\cl Q^{[a, b]}$. Indeed the constants $C_1$ and $C_2$ produced in the
proof depend only on $\alpha_{\text{max}}$ and $c_\alpha$ (which can 
depend also on $\alpha_{\text{min}}$ per \eqref{eq:calpha}),  so $C_1$ and $C_2$
can also be expressed
using only $a$ and $b$, making them independent of
$\alpha \in \cl Q^{[a, b]}$.  As a consequence, per the discussion in
Section~\ref{ssec:vc}, the lifted loss function
\begin{equation}
    \label{FOSLSLoss}
    \LL^\fosls(\w;\pdom^{[a,b]}_M):=
    \frac 1M \sum_{\pp\in \pdom^{[a,b]}_M}
\Lbdo(\w(\alpha))
\end{equation}
is variationally correct for any set of $M$ samples $\pdom^{[a,b]}_M\subset\pdom^{[a,b]}$
and admits two-sided bounds as in \eqref{eq:VCX} with efficiency and
reliability constants depending solely on $a$ and $b$.

\section{The DPG  loss function}
\label{sec:dpg-loss-function}

In this section we introduce the ultraweak DPG method
for~\eqref{laplaceDP-first-order}. This method was first presented in
\cite{DemkoGopal11a} and put in a larger context in the recent
review~\cite{DPGacta}. After presenting the discrete method, we present the DPG
loss function and discuss its variational correctness.

\subsection{The DPG solution.}

The ultraweak DPG formulation can be described for very general PDEs,
in discretized and undiscretized forms, and in arbitrary order spaces 
(see~\cite[Section~7]{DPGacta}).
It is a mesh-dependent variational formulation
in so-called ``broken graph spaces'' whose unbroken  counterpart was already given in Example~\ref{eg:uw}. For the  discretization,
we restrict ourselves here to the simplest
lowest order case for transparent presentation of the ideas.  Let
$\tr(w)$ and $\trn(r)$ denote element-wise trace operators defined by
$\tr(w)|_{\d K} = w|_{\d K}$ and
$\trn(r)|_{\d K} = (r \cdot n )|_{\d K}$.  Here and throughout $n$
denotes the outward unit normals on element boundaries.  Using the
spaces defined in \eqref{eq:RT0-Lag1},
define the following  spaces on element boundaries by
\begin{align*}
  \hU_1
  & = \{ \hat w: \; \hat w|_{\d K} = w|_{\d K}
    \text{ for some } w \in \cl{U}_1 \text{ for all } K \in \oh
    \}
  \\
  {\hRT}
  & = \{ \hat q_n: \; \hat q_n|_{\d K} = (r\cdot n)|_{\d K}
    \text{ for some }  r \in \RT \text{ for all } K \in \oh
    \}.
\end{align*}
Since the lowest order spaces $\cl U_1$ and $\RT$ are completely
defined by their element-wise traces, the spaces $\hU_1$ and ${\hRT}$
are isomorphic to $\cl U_1$ and $\RT$, respectively. We write a
generic element of ${\hRT}$ as $\hat r \cdot n$ or $\hat r_n$
to indicate its dependence on $n$.

The lowest order ultraweak
DPG method uses distinct trial  and test
spaces,
\[
\Xhdpg = \cl P_0^d \times \cl P_0 \times \hU_1 \times {\hRT},
\qquad
\Yhdpg = \cl P_2^d \times \cl P_{d+1}.
\]
Recall the formal adjoint of $A_\alpha$, given by
\eqref{eq:Aadjoint}.
As in Example~\ref{eg:uw}, 
it is used to  norm the test space by 
\begin{equation}
  \label{eq:Y-dpg-norm}
  \| (\tau, \nu) \|_{\cl Y}^2  =
  \sum_{K \in \oh}
  \Big( \| A_\alpha^*(\tau,\nu) \|_{L_2(K)}^2 + \| (\tau, \nu) \|_{L_2(K)}^2
  \Big)
\end{equation}
for all $(\tau, \nu) \in \Yhdpg$.
Note that the differential operators within $A_\alpha^*$ are applied
element by element when computing  the norm~\eqref{eq:Y-dpg-norm}, which is
referred to as a ``broken graph norm''.
The corresponding inner
product is denoted by $(\cdot,\cdot )_{\cl Y}$ and it is an essential
ingredient in computations with the DPG method.  Different flavors of
the method are obtained by changing the $\cl Y$ inner product, as we
shall see later. 

Define the DPG bilinear form by 
\begin{align}
  \label{eq:dpg-b-discrete}
  b_\alpha^\dpg((q,u,\hat u, \hat q\cdot n),
    (\tau,\nu))
    & =
      \sum_{K \in \oh} \left( 
      \int_K (q, u) \cdot  A_\alpha^*(\tau, \nu)
      + \int_{\d K}
      \hat u  \,\tau \cdot n
      + \int_{\d K}
      \hat q \cdot n \,\nu \right),
\end{align}
for any $(q, u, \hat u, \hat q\cdot n ) \in \Xhdpg$ and
$(\tau, \nu) \in \Yhdpg.$  We omit the standard
Lebesgue measure in integrals such as the above
when there can be no confusion.
The DPG method can be presented as
a Petrov-Galerkin formulation or equivalently as a mixed Galerkin
method. We only present the latter (and refer the reader to
\cite{DPGacta} for other equivalent forms). Given an $F$ in
$L_2(\om)^d \times L_2(\om)$ (which, as before, will be set to $(0, f)$
with $f$ as in~\eqref{laplaceDP-first-order}), the DPG method produces
an approximate solution $ \xdpg = (q_h^\dpg, u_h^\dpg, \hat u_h^\dpg, \hat q_h^\dpg)$ in $\Xhdpg$
as well as a built-in approximate error representation $e_h$ in
$\Yhdpg$, that together satisfy the equations 
\begin{subequations}
  \label{eq:DPG-method}
  \begin{alignat}{4}
    \label{eq:DPG-method-1}
  (e_h, y)_{\cl Y}
  \;+\;
  &
    b_\alpha^\dpg(\xdpg , \y)
  & \,=\,
  & (F,y)_{L_2},
    \qquad 
  & y \in \Yhdpg, 
  \\ \label{eq:DPG-method-2}
  & b_\alpha^\dpg(\w, e_h)
  & \,=\,
  & 0, \qquad 
  & \w \in \Xhdpg.
\end{alignat}
\end{subequations}
This generates  the DPG parameter-to-solution map
\[
  \alpha \overset{\fdpg}{\longmapsto} \xdpg,
  \qquad \fdpg(\alpha) = \xdpg,\quad
\]
for any $\pp\in\pdom$, i.e., 
the application of $\fdpg$ to any $\alpha$ is obtained by
solving~\eqref{eq:DPG-method}. We want to approximate $\fdpg$ by a NN.

\subsection{The DPG loss function}

The DPG loss function is computed in two steps. First, given any
$\w \in \Xhdpg$, we compute a residual representation $\vepzh \in \Yhdpg$ by solving 
\begin{subequations}
  \label{eq:DPG-loss}
  \begin{equation}
    \label{eq:DPG-loss-prep}
    (\vepzh, \y)_{\cl Y} = (F, \y)_{L_2} - b_\alpha^\dpg (\w, \y), \qquad
    \text{ for all } \y \in \Yhdpg.
  \end{equation}
Note that this is the same equation as~\eqref{eq:DPG-method-1} with
$e_h$ and $\xdpg$ replaced by $\vepzh$ and $\w$, respectively. 
As before we focus first on  parameter-dependent DPG fiber loss of $\w$ at $\pp$, which,  using $\vepzh$, we define by
\begin{equation}
  \label{eq:DPG-loss-fun}
  \Ldpg(\w) := \| \vepzh\|_{\Ydpg}^2, \qquad \w \in \Xhdpg,
\end{equation}
\end{subequations}
where the $\alpha$-dependent $\Ydpg$ norm is as defined in~\eqref{eq:Y-dpg-norm}.  
We can also write this loss in operator form. One way to do so quickly
is to introduce the mappings
$R_{\alpha, h}: \Yhdpg \to (\Yhdpg)'$ and $B_{\alpha,h}^\dpg : \Xhdpg \to (\Yhdpg)'$,  
defined by $(R_{\alpha, h} y)(z)  = (y, z)_{\cl Y}$ and
$(B_{\alpha,h}^\dpg \w)( y)  = b_{\alpha,h}^\dpg(\w, y)$ for
all $\w \in \Xhdpg$ and $y, z \in \Yhdpg$.  Then~\eqref{eq:DPG-loss-prep}
can be expressed as 
\[
  R_{\alpha, h} \vepzh + B_{\alpha,h}^\dpg \w  = F
\]
and~\eqref{eq:DPG-loss-fun} becomes
\begin{equation}
  \label{eq:DPG-loss-operator-form}
  \Ldpg(\w) = \| R_{\alpha, h}^{-1} (F - B_{\alpha,h}^\dpg \w) \|_{\cl{Y}}^2,
  \qquad \w \in \Xhdpg.
\end{equation}
The key feature that makes this fiber loss  function quickly computable is
the block diagonal nature of the matrix of $R_{\alpha, h}$ (representing the
Riesz map), with one block per element (so its inverse appearing above
can be computed locally element by element).  This is a consequence of
the fact that functions in $\Yhdpg$ have no inter-element continuity
constraints.

\subsection{Variational correctness of the DPG loss function}

The variational setting   which the ultraweak DPG method is based upon is
more complicated than that of FOSLS. First  consider the trial space. 
To describe the norm in an infinite-dimensional Hilbert space 
$\Xdpg$ enclosing the discrete computational trial  space $\Xhdpg$, let
\[
  \text{tr} : \Ho \longrightarrow \prod_{K \in \Omega_h} H^{1/2}(\partial K),\quad
  (\text{tr } u) \big\rvert_{\partial K} = u \big\rvert_{\partial K},
\]
and
\[
  \text{tr}_n : \Hdiv \longrightarrow  \prod_{K \in \Omega_h} H^{-1/2}(\partial K),
  \quad
  (\text{tr}_n r)\big\rvert_{\partial K} = r \cdot n \big\rvert_{\partial K}, 
\]
denote element-wise boundary trace operators extended to the indicated 
global Sobolev spaces.  We define norms of trace and flux functions on 
element boundaries using their preimages of the appropriate trace map by
\[
  \|\hat u\|_{{\Hoh}}
  := \inf_{u \in \text{tr}^{-1}\{\hat u\}} \|u\|_{\Ho},
  \qquad
  \|\hat r_n\|_{H^{-1/2}(\partial \Omega_h)}
  := \inf_{ q \in \text{tr}_n^{-1}\{\hat r_n\}} \| q\|_{H(\div)}
\]
and set $\Hoh$ and $H^{-1/2}(\d\oh)$ to be the ranges of
$\text{tr}(\cdot)$ and $\text{tr}_n(\cdot)$, respectively.
The well-posedness of
ultraweak DPG formulation can be proven in the trial norm
defined by 
\begin{equation}
  \label{eq:Xdpg-norm}
  \| (r, w, \hat w, \hat r) \|_{\Xdpg}^2
  =
  \| r \|_{L_2}^2 + \| w \|_{L_2}^2
  + \|\hat w \|_{\Hoh}^2
  + \| n \cdot \hat r\|_{H^{-1/2}(\d\oh)}^2.  
\end{equation}
The DPG  test space is set to the product of the following spaces
\[
  \Hdivoh = \prod_{K \in \oh} H(\div, K), \qquad
  H^1(\oh) = \prod_{K \in \oh} H^1(K),
\]
representing piecewise $H(\div)$ and $H^1$ spaces (containing functions
without any interelement continuity constraints). With these definitions,
we can now give the undiscretized DPG formulation, representing the fiber problem at some given parameter $\alpha$, for
the exact solution with  four components
\begin{subequations}
  \label{eq:dpg-undiscretized-formulation}
\begin{equation}
  \label{eq:exact-dpg}
\x^\dpg = (q, u, \hat u, \hat q_n).
\end{equation}
Here  $(q, u)$ is the exact solution of~\eqref{laplaceDP-first-order}, $\hat u = \tr(u)$ and $\hat q_n = \trn(q)$. It is well known that $\x^\dpg$ is the unique function in $\Xdpg$ satisfying 
\begin{equation}
  \label{eq:12}
  b_\alpha^\dpg(\x^\dpg, \y) = \ell^\dpg(\y), \qquad \forall \y \in \cl Y^\dpg,
\end{equation}
here, 
\begin{align}
    & \Xdpg
    = L_2(\om)^d \times L_2(\om) \times \Hoh\times  H^{-1/2}(\d\oh),
    \quad \text{ normed by \eqref{eq:Xdpg-norm},}
    \\
    & \cl Y^\dpg
    = \Hdivoh \times H^1(\oh),
    \quad \text{ normed by broken graph norm \eqref{eq:Y-dpg-norm},}
    \\
    & b_\alpha^\dpg((q,u,\hat u, \hat q\cdot n),
    (\tau,\nu))
    =
      \sum_{K \in \oh} \bigg( 
      \ip{ \tau\cdot n, \hat u}_{\d K} + \ip{ \hat q \cdot n, \nu}_{\d K} +
      \int_K (q, u) \cdot  A_\alpha^*(\tau, \nu)
      \bigg)
    \\
    & \ell^\dpg((\tau,\nu)) = (f, \nu)_{L_2},
  \end{align}
\end{subequations}
where $\ip{\cdot,\cdot}_{\d K}$ denotes the duality pairing between $H^{-1/2}(\d K)$ and $H^{1/2}(\d K)$.
Note that $\Yhdpg \subset \Ydpg$ and  $\Xhdpg \subset \Xdpg.$ Since the exact problem has zero error, with a zero error representation $e=0_{\Ydpg}$, we can expand
\eqref{eq:12} into a mixed method of the form \eqref{eq:DPG-method} to  see that 
\eqref{eq:DPG-method} is indeed a conforming discretization of the exact  formulation~\eqref{eq:dpg-undiscretized-formulation}.
The following result is
obtained by putting together known results from DPG
analysis~\cite{CarstDemkoGopal14, DPGacta}.

\begin{proposition}
  \label{prop:VC-DPG}
  The DPG loss function $\Ldpg$ is variationally correct in $\Xhdpg$
  when $f$ is in $\cl P_0$, i.e., there are mesh-independent
  constants $C_1, C_2>0$ (that depend on $c_\alpha$ in
  \eqref{eq:A-bdd-below}) such that
  \[
    C_1 \| \x^{\dpg} - \w \|_{\Xdpg}^2
    \le \Ldpg(z) \le C_2 \| \x^{\dpg} - \w \|_{\Xdpg}^2,    
  \]
  for all $\w \in \Xhdpg.$ Here $\x^{\dpg} = (q, u, \hat u, \hat q_n)$ is the
  exact solution of~\eqref{eq:dpg-undiscretized-formulation}.
\end{proposition}
\begin{proof}
  We apply~\cite[Theorem~6.4]{DPGacta} (see also
  \cite{GopalQiu14,CarstDemkoGopal14}) which  yields
  \begin{equation}
    \label{eq:reliability-efficiency}
    \begin{aligned}
      \gamma \| \x^{\dpg} - \w\|_{\Xdpg}
      & \le \| \vpi \|
    \| \vepzh \|_{\Ydpg}
    + \osc(F),
    \\
      \| \vepzh \|_{\Ydpg}
      & \le \| b_\alpha^\dpg \| \| \x^{\dpg} - \w \|_{\Xdpg},      
    \end{aligned}
  \end{equation}
  provided there is a continuous operator $\vpi: \Ydpg \to \Yhdpg$, 
  called a Fortin operator, satisfying
  \begin{equation}
    \label{eq:Fortin}
    b_\alpha^\dpg(\w_h, y - \vpi y) =0
    \quad \text{ for all }
    \w_h \in \Xhdpg, y \in \Ydpg,
  \end{equation}
  and provided the norm $\| b_\alpha^\dpg \|$ and the
  inf-sup constant $\gamma$ of the DPG form over $\Xdpg \times \Ydpg$
  are finite and positive. In~\eqref{eq:reliability-efficiency},
  $\| \vpi\|$ is the operator norm of $\vpi$ and
  \[
    \osc(F )= \| F \circ (I - \vpi) \|_{(\Ydpg)'}.
  \]
  Since the DPG loss function is defined by~\eqref{eq:DPG-loss-fun},
  the estimates of~\eqref{eq:reliability-efficiency} prove the result
  once we verify these conditions and show that $\osc(F)$ vanishes.

  To obtain the required Fortin operator $\vpi$, recall that for
  any $r \in \Hdiv$ and $w \in \Ho$, by~\cite[Theorems~4.1
  and~5.8]{DPGacta}, there are functions $\vpig w\in \cl P_{1+d}$ and
  $\vpid r\in \cl P_2^d$ such that
  \begin{equation}
    \label{eq:Fortin}    
    \begin{aligned}
      \int_K \vpig w - w  & = 0,
      & \int_F \vpig w - w  & = 0, 
      \\
      \int_K \vpid r - r  & = 0, \qquad 
      & \int_{\d K}  (\vpid r - r) \cdot n & = 0,
    \end{aligned}
  \end{equation}
  for all mesh elements $K$ and all facets $F$ of $\d K$. In the last
  equality the integral of $r \cdot n$ over $\d K$ must in general be
  interpreted as a functional action of the $H^{-1/2}(\d K)$~trace
  $(r \cdot n)|_{\d K} $ on $\d K$.  Setting $\vpi : \Ydpg \to \Yhdpg$
  by $\vpi(r, w) = (\vpid r, \vpig w)$ the same theorems yield the
  bound
  \[
    \| \vpi (r, w) \|_{\Hdiv \times \Ho} \lesssim
    \| (r, w) \|_{\Hdiv \times \Ho}.
  \]
  Now exactly the same argument that proved~\eqref{eq:Ydpg-equiv}
  shows that the above inequality implies
  \[
    \| \vpi (r, w) \|_{\Ydpg} \lesssim  \| (r, w) \|_{\Ydpg}, \qquad (r,w) \in \Ydpg,
  \]
  so $\| \vpi\| \lesssim 1$. Also,~\eqref{eq:Fortin} implies
  $b_\alpha^\dpg(\w_h, (r, w) - \vpi (r, w)) =0$ for all
  $\w_h \in \Xhdpg$ and $(r,w) \in \Ydpg$.  Furthermore,
  \[
    \osc(F)
    = \sup_{y \in \Ydpg}  \frac{(F, y - \vpi y)_{L_2}}{ \| y \|_{\Ydpg} }
  \]
  vanishes since \eqref{eq:Fortin} implies 
  \begin{align*}
    (F, (r, w) - \vpi (r,w))_{L_2}
    & = (f, w - \vpig w)_{L_2} = 0
  \end{align*}
  for all $(r, w) \in \Ydpg$ since $f \in \cl P_0$. It only remains to
  verify that
  \begin{equation}
    \label{eq:cty-infsup}
    \| b_\alpha^\dpg \| \lesssim 1, \qquad 1 \lesssim \gamma.
  \end{equation}
  These are shown in~\cite[Theorem~7.6 and Example~7.8]{DPGacta} (see
  also~\cite{DemkoGopal11a}) using \eqref{eq:A-bdd-below}.
\end{proof}

The DPG parameter-to-solution map is found by performing the
minimization in~\eqref{cL} with $\cl L$ set to $\Ldpg$
and $\Xh$ set to $\Xhdpg$.
Precisely as in the preceding section, we employ the lifted  loss function 
\begin{equation}
    \label{mean-squared}
    \LL^{\dpg}(\w; \pdom_M):= \frac 1M \sum_{\pp\in\pdom_M}
    \Ldpg(\w(\alpha)),
\end{equation}
for the empirical risk minimization that trains the NN. As in the
FOSLS case, one can track constants to argue robustness for bounded
parameter domains such as in \eqref{eq:Qab}. But we have instead
chosen to devote an entire section (Section~\ref{sec:param-robust}) to
the robustness of the DPG loss.

\section{Performance of neural network predictions}
\label{sec:initial-results}

In this section, we show numerical results for both types of
learning methods and compare their performance. In these comparisons we do not take the (substantial, offline) training costs into account but focus solely on the achieved accuracy, reflected by the respective loss functions. Since the evaluation of DPG losses is slightly more expensive, it will be important to understand features that may offset somewhat higher cost.
Our numerical   results in this section  show
comparable performance for the parameter-to-solution map approximations
obtained using NN of the form \eqref{eq:NN-F}, trained either using the FOSLS loss or using the DPG
loss.  This then forms the backdrop for further performance gains
described in Section~\ref{sec:param-robust}.

\subsection{Default parameters}
\label{ssec:default-parameters}
Unless otherwise stated, all our subsequent numerical reports
are obtained with the following default settings. 
We set $m_\alpha = 4$, $\om$ to the unit
square, split into four congruent subdomains, partitioned into a
finite element mesh (that aligns with the subdomains) of mesh size
$h \approx 0.1$, and set $f \equiv 1$. (Smaller $h$ and higher $m_\alpha$ 
are considered in \S~\ref{ssec:16-dim}.)

\begin{figure}
  \centering
  \begin{tikzpicture}
    \begin{semilogxaxis}[
      width=15cm,
      height=5cm,
      xlabel={Contrast: $\max(\alpha) / \min(\alpha)$},
      ylabel={Frequency},
      grid=both,
      grid style={dashed, gray!30},
      smooth, 
      thick,
      ymin=0,
      xmax=1e8,
      xtick={1e-1, 1e1, 1e3, 1e7, 1e13},
      cycle list={
        {tabred},
        {tabblue},
        {tabpurple}
      }
      ]
      
      \addplot table [x=x_0.1, y=y_0.1, col sep=comma] {min_max_alpha_curves_smooth.csv};
      \addlegendentry{$\sigma=$0.1}
      
      \addplot table [x=x_0.5, y=y_0.5, col sep=comma] {min_max_alpha_curves_smooth.csv};
      \addlegendentry{$\sigma=$0.5}
      
      \addplot table [x=x_1.0, y=y_1.0, col sep=comma] {min_max_alpha_curves_smooth.csv};
      \addlegendentry{$\sigma=$1.0}

    \end{semilogxaxis}
  \end{tikzpicture}
  \caption{Histogram of contrasts measured by the  ratios
  $\min(\alpha) / \max(\alpha)$
  of $\alpha$-samples (over $10^5$ samples) across varying standard
  deviation settings}
  \label{fig:min_max_alpha_dist}
\end{figure}

We use the previously described architecture (see
\S\ref{ssec:neural-network}) with the number of hidden layers equal to
$l = 13$, width $N = 128$, and rank $r = 32$. The input to NN are the
four numbers of $\bbalpha$ when $m_\alpha=4$. With a small abuse of
notation, from now on we do not distinguish between the function
$\alpha$ and the vector $\bbalpha$ and simply write $\alpha$ as a
4-tuple, i.e., $\alpha = (\alpha_1, \alpha_2, \alpha_3, \alpha_4),$
when reporting experimental results and component ordering is as
illustrated in Figure~\ref{fig:param-example}. We use
NGSolve~\cite{Schobother}, an open-source finite element library, for
shape functions, matrix assembly, and efficient implementation of all
finite element structures needed for computing the loss functions.  We
implement the DPG and FOSLS loss functions within the
PyTorch~\cite{PaszkGrossMassa19} framework, interfacing it with
NGSolve, so as to leverage the backward differentiation and stochastic
gradient descent facilities of PyTorch.

Each NN is trained using a set of $M=1024$ random samples of $\alpha$
found using a {normal distribution} with a specified mean value
$\bar \alpha \in \bb R^4$ and standard deviation $\sigma$, as
follows. We take the square root of the positive numbers given in
$\bar\alpha$, then sample from a normal distribution around that mean
with standard deviation $\sigma$, and then square the result, i.e.,
$\alpha$ is from a scaled non-central chi-squared distribution, which
we denote by $\chi^2_{\bar\alpha, \sigma}$.  This distribution is
chosen to guarantee that the material coefficient samples are positive
in all experiments (while maintaining the ability to sample difficult
extreme values of $\alpha$).  Figure~\ref{fig:min_max_alpha_dist} shows that these
distributions contain more and more samples with high-contrasts (the ratio
of maximal to minimal values of a sampled parameter function $\alpha$)
as $\sigma$ is increased.
The figure is obtained as a histogram of contrasts (with fine binning)
of $100,000$ $\alpha$-samples from $\chi^2_{\bar\alpha, \sigma}$ for
$\bar\alpha=(1,1,1,1)$ and various~$\sigma$. Some samples reached contrast values as high as $10^{7}$.

By default, we set the number of epochs to $5000$, the batch size to
$32$ and the learning rate to $0.0001$ in trainings.  The NN
predictions of solutions are then tested against the actual finite
element solution on random samples of $\alpha$ in multiple ways as
described below. Note that $M=1024$ is a somewhat modest number of
training samples even if we had restricted $\alpha$ to a bounded
four-dimensional cube. We shall see that even for $\alpha$ in the
above-mentioned unbounded four-dimensional domain, good predictions
can be obtained with just $M=1024$ training samples.

\subsection{Visualization of some NN predictions}

We first present a qualitative visual comparison between the neural
network predictions and the exact finite element solution for some
$\alpha $ taken from the distribution described before, with mean and
standard deviation set by 
\begin{equation}
\label{eq:distribution-prm}
  \bar \alpha = (10^{-1}, 1, 1, 10^{-1}),
  \quad 
  \sigma = 0.1.
\end{equation}
We generate a pair of predictions
\begin{equation}
  \label{eq:dpg-fosls-theta}
  \x_\theta^\fosls = (q_\theta^\fosls, u_\theta^\fosls) \quad \text{ and }\quad 
  \x_\theta^\dpg = (q_{0,\theta}^\dpg, u_{0,\theta}^\dpg, \hat u_\theta^\dpg, \hat q_\theta^\dpg)
\end{equation}
from a NN trained with the FOSLS loss and a NN trained with the DPG loss,
respectively.  The last two DPG components, namely
$(\hat q_\theta^\dpg, \hat u_\theta^\dpg)$, can be identified to be in the same
space where FOSLS solution belongs since they have a unique extension
from the element boundaries to the element interior as functions in
$\Xhfosls$.

We compare these predictions with the corresponding finite element
solutions $(q_h^\fosls, u_h^\fosls)$ and
$(\hat q_h^\dpg, \hat u_h^\dpg)$, obtained by
solving~\eqref{eq:FOSLS-method} and ~\eqref{eq:DPG-method},
respectively, for the same $\alpha$. The prediction $u_\theta^\fosls$
and the finite element solution $u_h^\fosls$ are plotted over $\om$ in
Figures~\ref{fig:good_u_FOSLS_NN} and~\ref{fig:good_u_FOSLS_FEM} for a
randomly chosen
\begin{equation}
  \label{eq:alpha-sample-good}
    \alpha_{\text{sample}} = (0.0904, 0.7255, 0.9192, 0.1948).
\end{equation}
The analogous quantities from the DPG method of comparable order are
$\hat u_\theta^\dpg$ and $\hat u_h^\dpg$ (obtained using the same 
$\alpha_{\text{sample}}$).  They are plotted in Figures~\ref{fig:good_uhat_DPG_NN}
and~\ref{fig:good_uhat_DPG_FEM}.  Clearly, by a visual comparison, we
see that both the neural networks perform very well in predicting the
finite element solution.

\begin{figure}
  \centering
  \subcaptionbox
  {$u^\fosls_\theta$
    \label{fig:good_u_FOSLS_NN}
  }[0.45\textwidth]
  {
    \includegraphics
    [scale=0.25,
    trim={380  100  300  100}, clip]
    {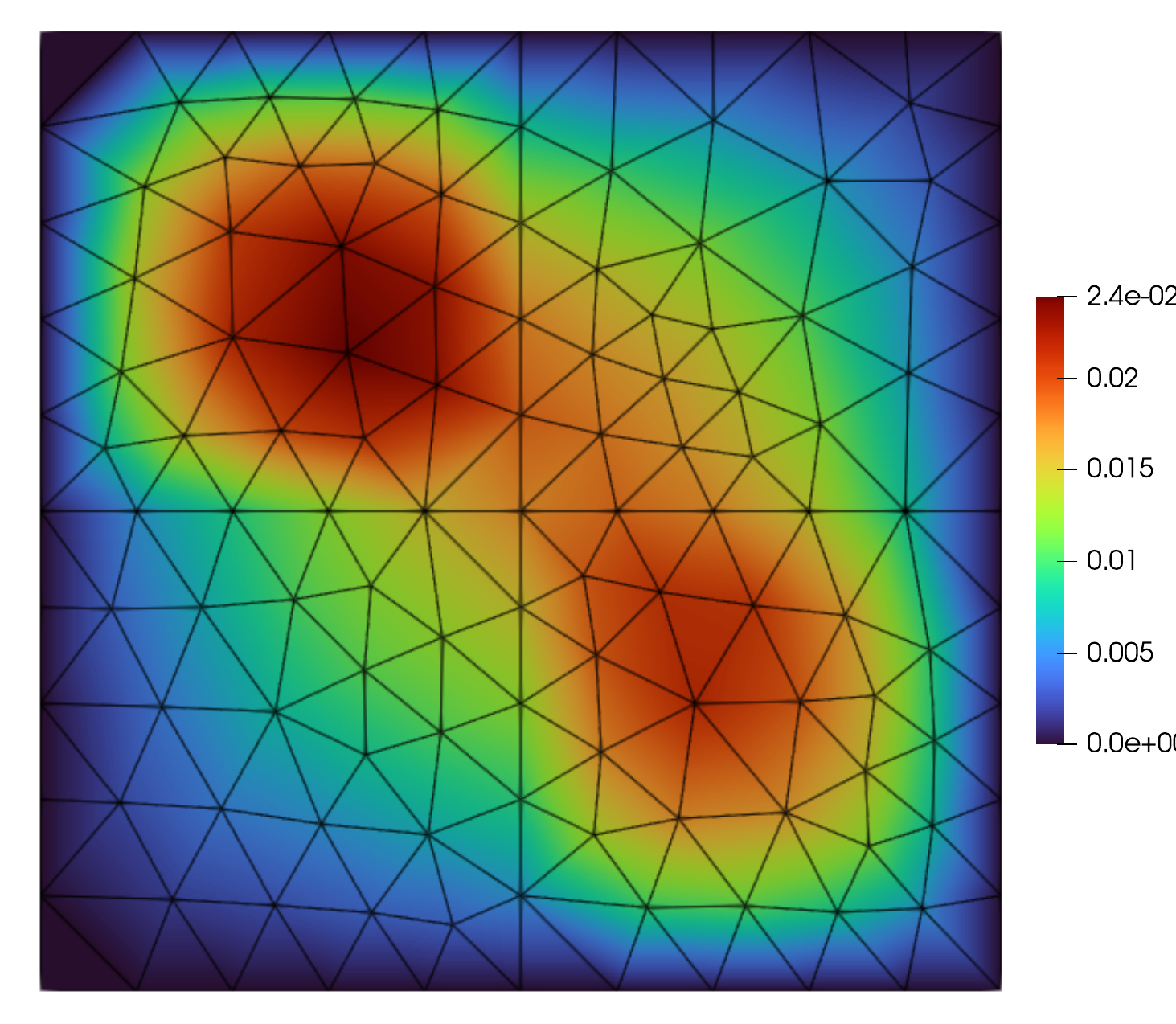}
  }
  \qquad
  \subcaptionbox
  {$u^\fosls_h$
    \label{fig:good_u_FOSLS_FEM}
  }[0.45\textwidth]
  {
    \includegraphics[scale=0.25,
    trim={380  100  300  100}, clip]
    {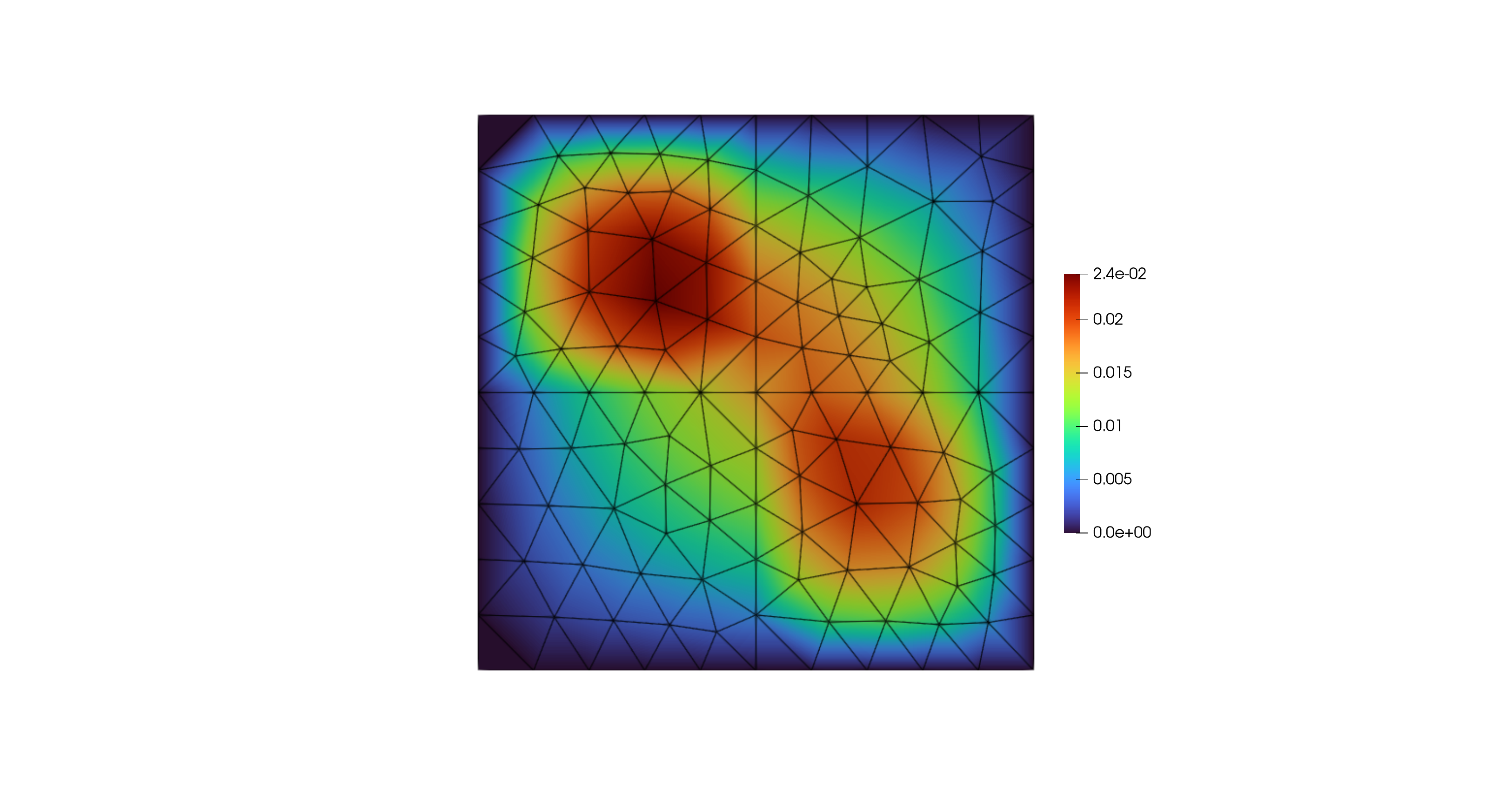}
  }
  \qquad  
  \subcaptionbox
  {$\hat u^\dpg_{\theta}$
    \label{fig:good_uhat_DPG_NN}
  }[0.45\textwidth]
  {
    \includegraphics[scale=0.25,
    trim={380  100  300  100}, clip]
    {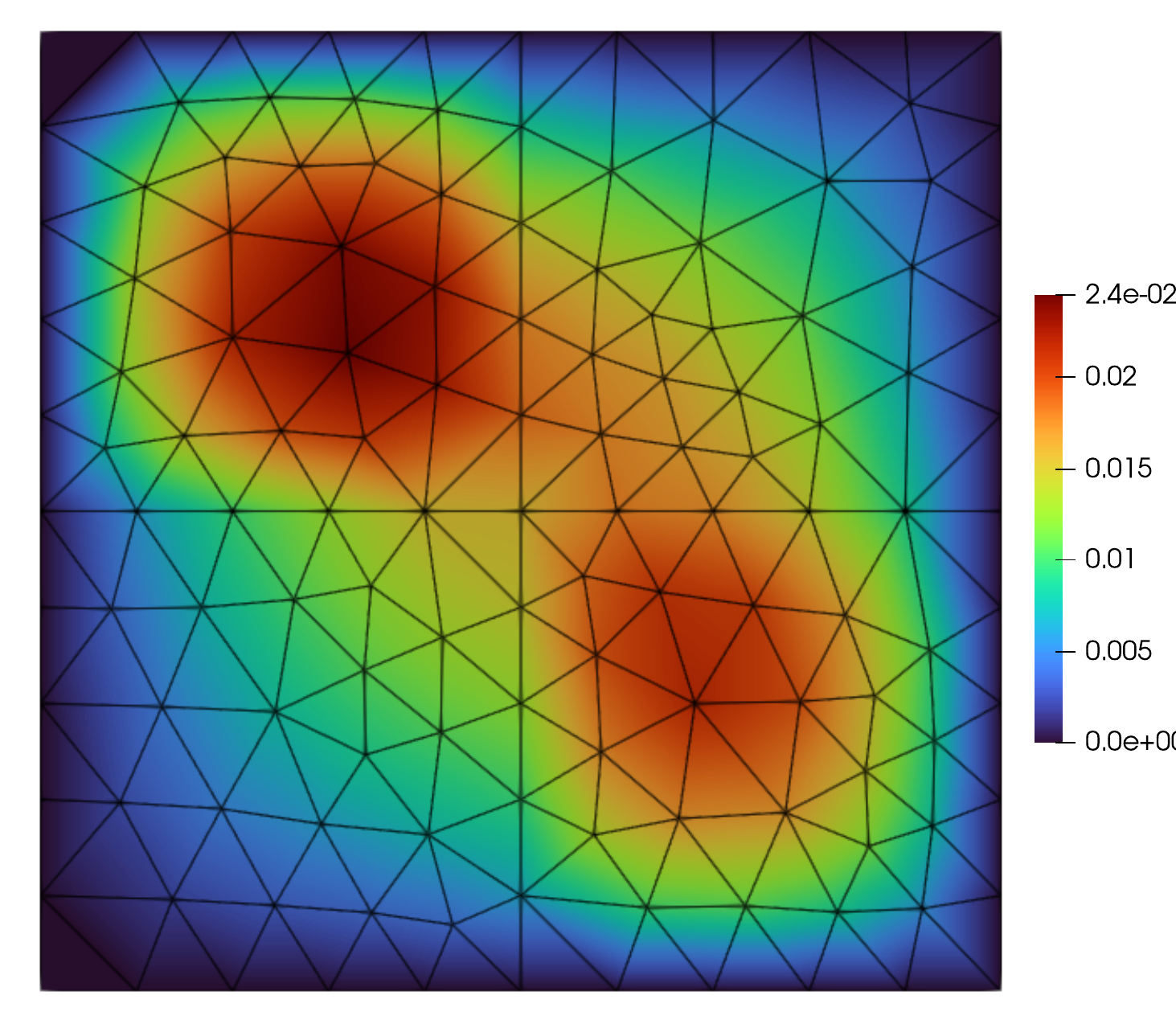}
  } 
  \qquad 
  \subcaptionbox
  {$\hat u^\dpg_h$
    \label{fig:good_uhat_DPG_FEM}
  }[0.45\textwidth]
  {
    \includegraphics[scale=0.25,
    trim={380  100  300  100}, clip]
    {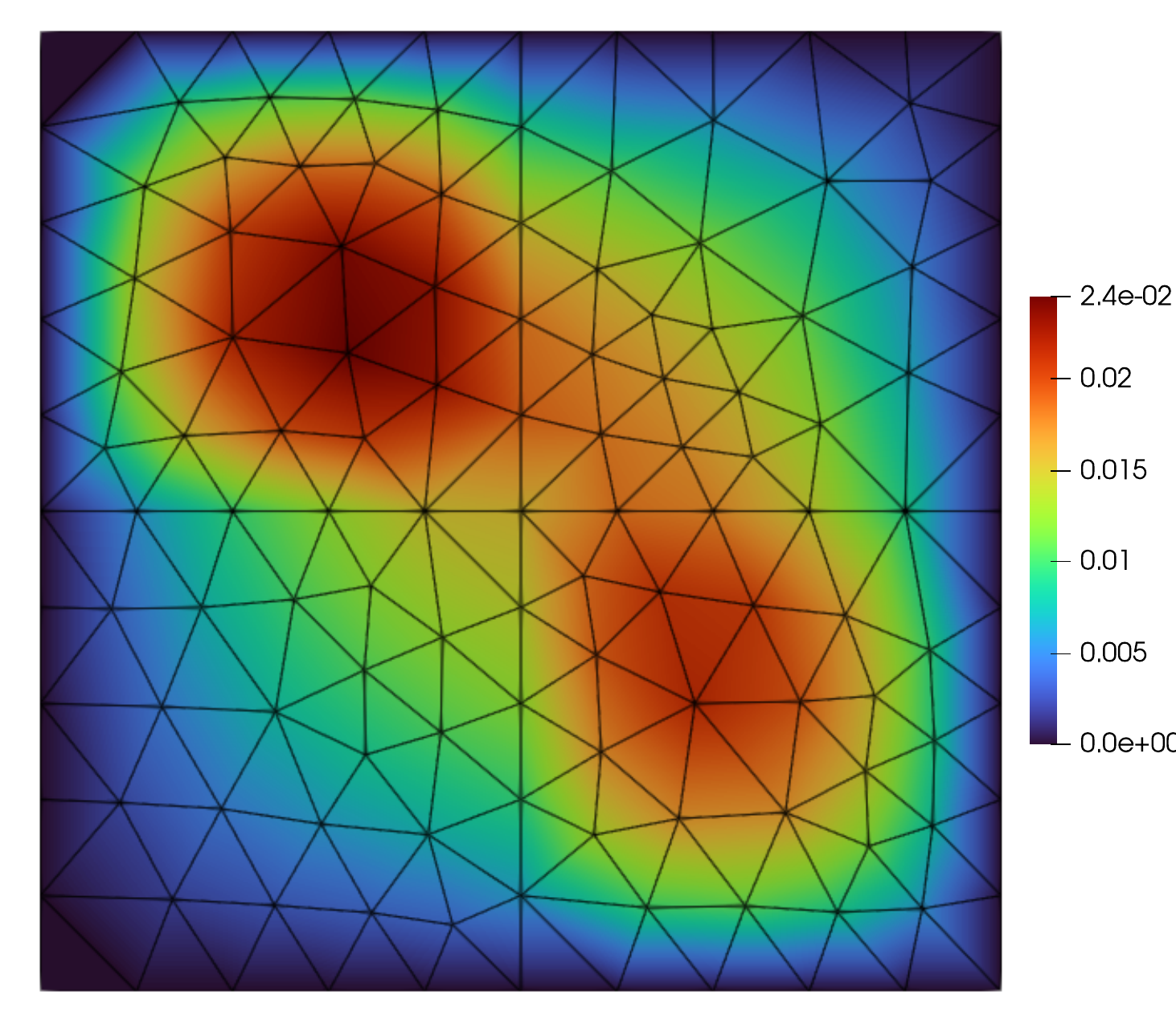}
  }
  \caption{Visualization of actual finite element solutions (right) and predictions (left) from an NN trained around the $\bar\alpha$ in~\eqref{eq:distribution-prm}, for the sampled $\alpha$ in~\eqref{eq:alpha-sample-good}.}
  \label{fig:qualitative_good_alpha}
\end{figure}

Now, we repeat the exercise, this time for a different $\alpha$
deliberately selected to be away from the mean $\bar \alpha$
in~\eqref{eq:distribution-prm}, namely
\begin{equation}
  \label{eq:alpha-sample-bad}
  \alpha_{\text{sample}} = (0.43, 1.0, 1.0, 0.43).
\end{equation}
We do so to observe whether
the predictive capacity of the neural networks decreases when $\alpha$
approaches the tail of the distribution and report this comparison in
Figures~\ref{fig:bad_u_FOSLS_NN}---\ref{fig:bad_u_FOSLS_FEM} for FOSLS
and Figures~\ref{fig:bad_uhat_DPG_NN}---\ref{fig:bad_uhat_DPG_FEM} for
DPG. The NN trained by FOSLS loss produces a prediction which looks
poorer than the one produced by the DPG loss. For this $\alpha$, the
finite element solutions produced by both methods (also displayed) are
still computed stably, even though the NN predictions vary. (Note that
for more difficult $\alpha$, the finite element solutions from both
methods can also visually differ due to stability differences of the methods,
but this is not the case in the figure.)

\begin{figure}
  \centering
   
  \subcaptionbox
  {$u^\fosls_\theta$
    \label{fig:bad_u_FOSLS_NN}
  }[0.45\textwidth]
  {
    \includegraphics[scale=0.25,
    trim={380  100  300  100}, clip]
    {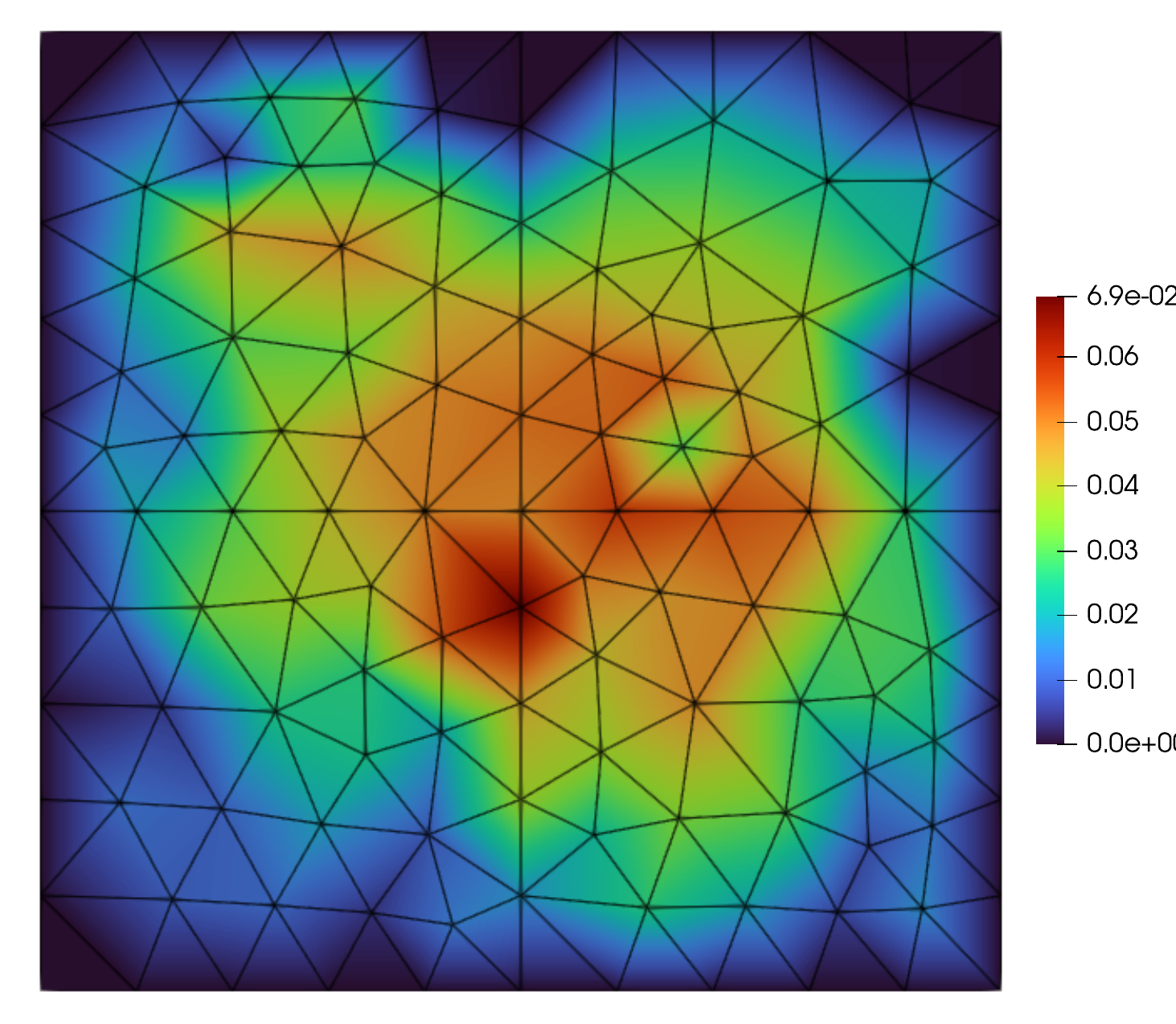}
  }
  \qquad
  \subcaptionbox
  {$u^\fosls_h$
    \label{fig:bad_u_FOSLS_FEM}
  }[0.45\textwidth]
  {
    \includegraphics[scale=0.25,
    trim={380  100  300  100}, clip]
    {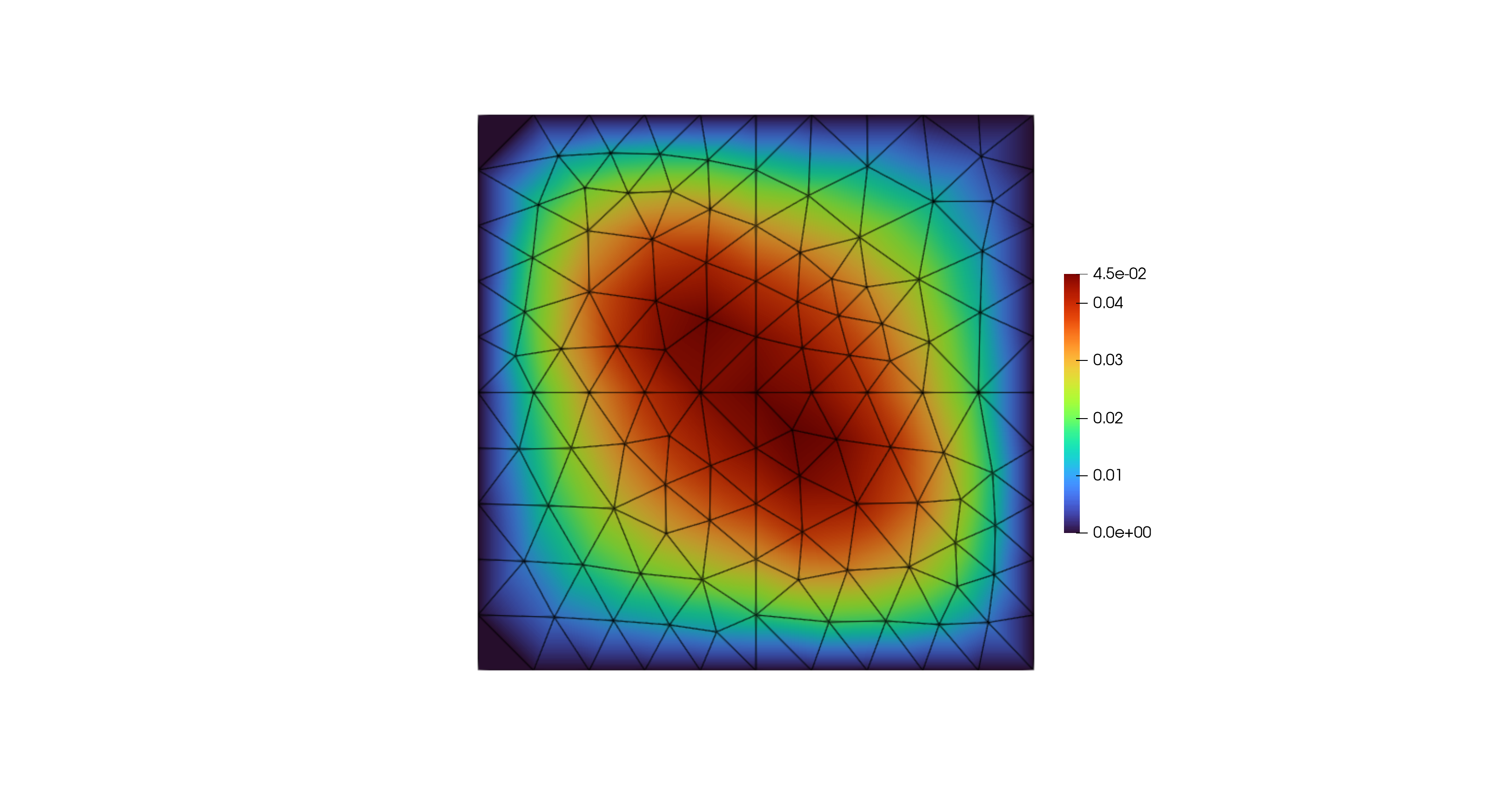}
  }
  \qquad  
  \subcaptionbox
  {$\hat u^\dpg_{\theta}$
    \label{fig:bad_uhat_DPG_NN}
  }[0.45\textwidth]
  {
    \includegraphics[scale=0.25,
    trim={380  100  300  100}, clip]
    {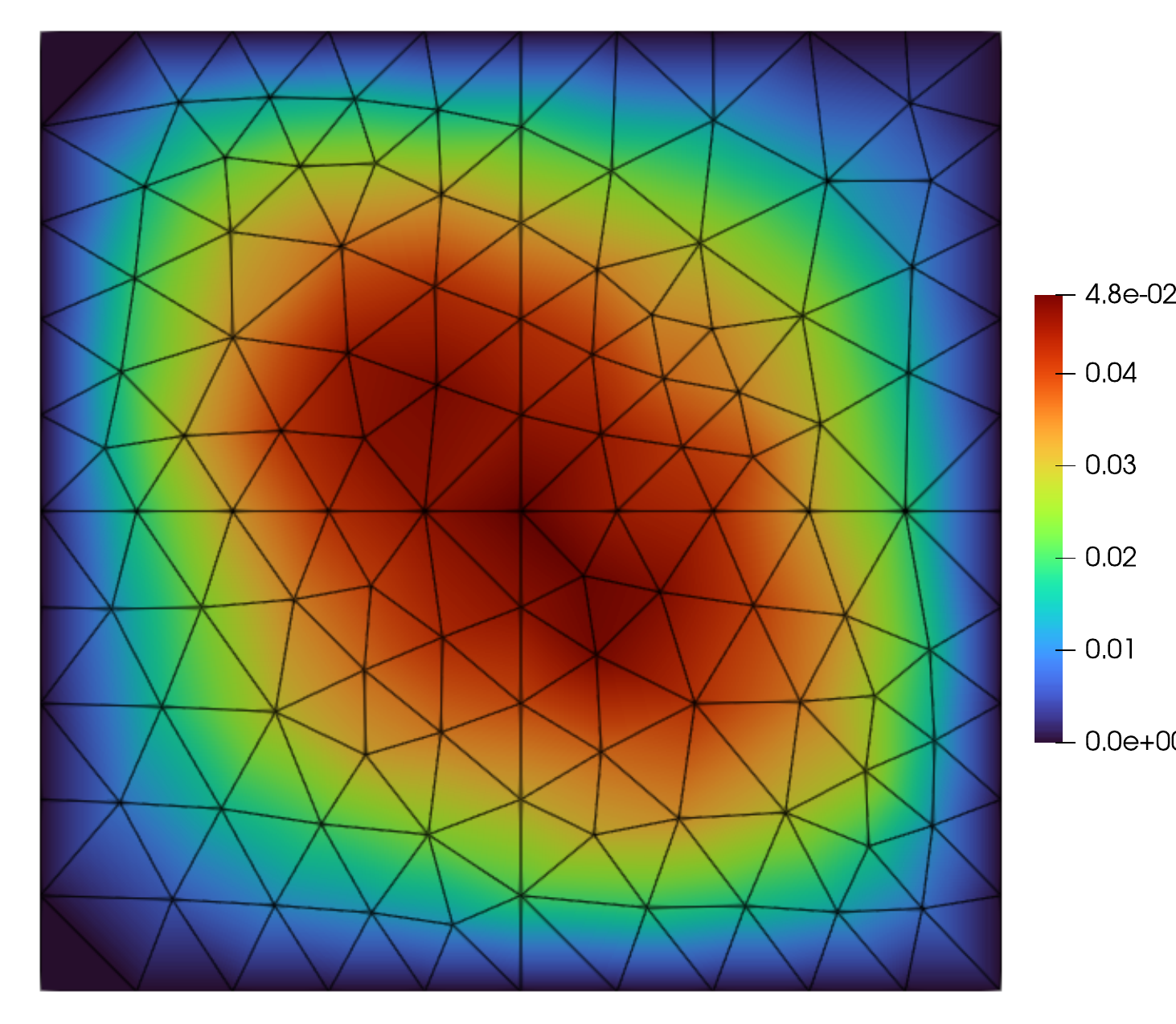}
  } 
  \qquad 
  \subcaptionbox
  {$\hat u^\dpg_h$
    \label{fig:bad_uhat_DPG_FEM}
  }[0.45\textwidth]
  {
    \includegraphics[scale=0.25,
    trim={380  100  300  100}, clip]
    {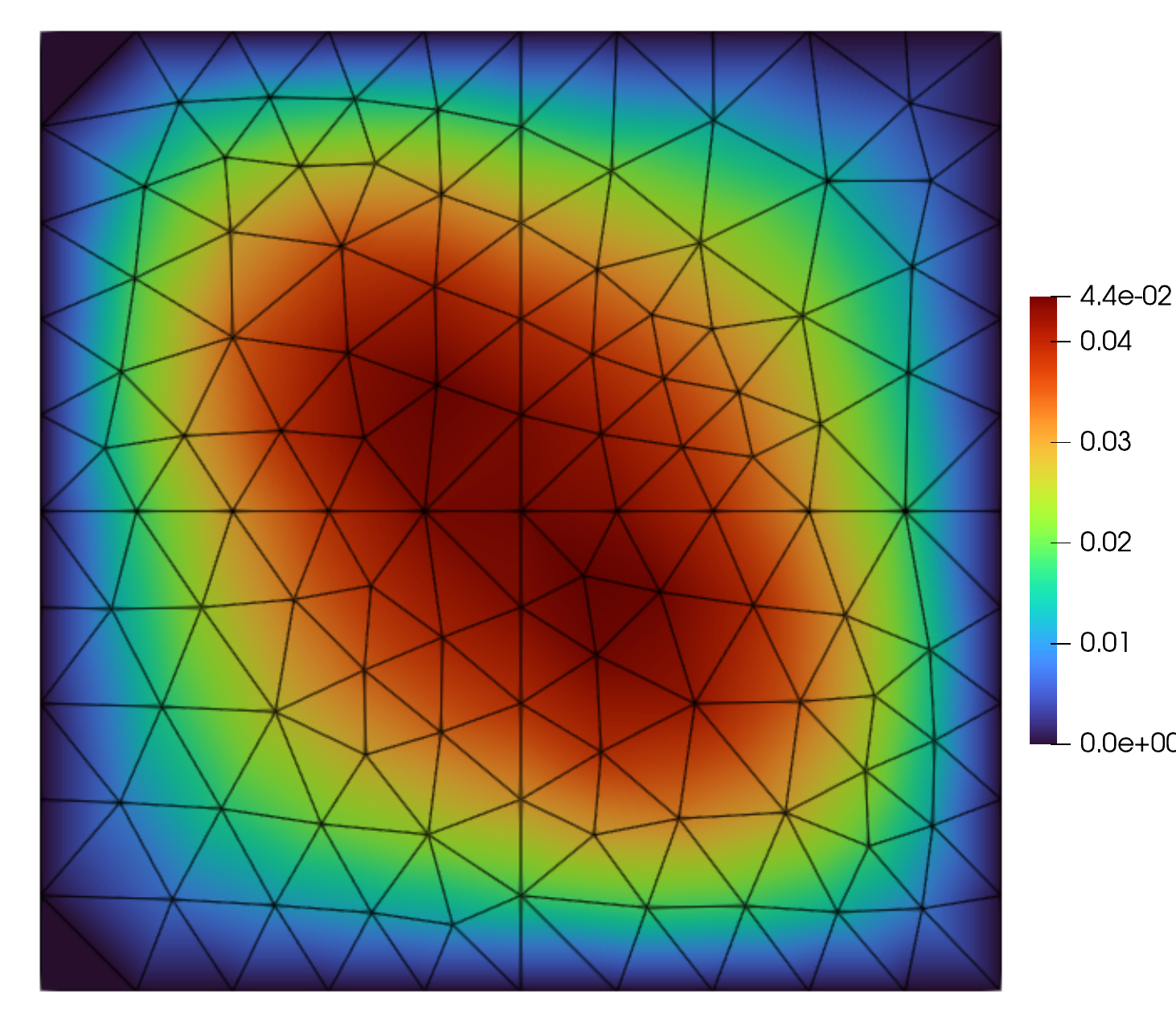}
  }
  \caption{Visualization of actual finite element solutions (right) and predictions (left) from an NN trained around the $\bar\alpha$ in~\eqref{eq:distribution-prm}, for the sampled $\alpha$ in~\eqref{eq:alpha-sample-bad}.}
  \label{fig:qualitative_bad_alpha}
\end{figure}

\subsection{Performance when varying the mean of $\alpha$}

To obtain a more quantitative measure of performance of NN predictions
across the $\alpha$-distribution, we now tabulate the  observed mean of
squares of errors from a (relatively large) sample of 10,000
$\alpha$-values in two experiments.

In our first experiment, we fix $\alpha$ distributions with standard
deviation $\sigma = 0.1$ and vary the mean
$\bar\alpha = ({\alpha_1}, 1, 1, \alpha_1)$ by choosing $\alpha_1$ to 
$100^{-1}, 10^{-1}, 10, $ or $100$. With each such choice of mean, a NN
was trained with FOSLS loss and another with DPG loss.  Predictions
from NN, trained using FOSLS loss, are denoted by
$(q_\theta^\fosls, u_\theta^\fosls)$. They are compared against the
actual FOSLS finite element solution
$\xfosls = (q_h^\fosls, u_h^\fosls) \in \Xhfosls$ obtained by solving
\eqref{eq:FOSLS-method} in the $L_2(\om)$-norm. The results are
tabulated in Table~\ref{tab:vary-mean} together with comparable
quantities for the DPG case. Comparable quantities in the DPG solution
$\xdpg = (q_h^\dpg, u_h^\dpg, \hat u_h^\dpg, \hat q_h^\dpg)$ are the
last two components, since $(\hat q_h^\dpg, \hat u_h^\dpg)$ can be
identified to be in the same space where FOSLS solutions belong since
they have a unique extension from the element boundaries to the
element interior as functions in $\Xhfosls$.  Accordingly, we compare,
in the $L_2(\om)$-norm, the predictions
$(\hat q_\theta^\dpg, \hat u_\theta^\dpg)$ from the NN (trained using
the DPG loss) with the DPG solution $(\hat q_h^\dpg, \hat u_h^\dpg)$
obtained by solving~\eqref{eq:DPG-method} at each sample point.  We
reiterate that the error numbers in Table~\ref{tab:vary-mean} are errors in $\X(\pdom_M)= \ell_2(\pdom_M;\cX^\fosls), \ell_2(\pdom_M;\Xdpg)$, respectively (see Section~\ref{ssec:vc}), i.e., they
are the
mean of $L_2(\om)$-norms of errors of ten thousand random $\alpha$
samples taken from the above-mentioned distribution for
$\alpha$ (which is a large test set size considering that  the training was done only with $1024$ samples). For each sample, the error is computed after solving the
exact finite element problem (DPG or FOSLS) on the mesh. 
We have observed for a few extreme $\alpha$-samples (at most 20 in 10,000) the sparse stiffness matrix inversion required to compute the exact finite element solution appears to have failed (and we could not compute the error). We have discarded such samples while calculating the mean squared error. Interestingly, the loss and NN predictions were computable even when the finite element solution could not be computed.
\begin{table}[htbp]
  \centering
  \begin{tabular}{l|c|c|c|c}
    \hline 
    \multicolumn{5}{c}{From $\alpha$ distributions with mean
    $\bar\alpha = ({\alpha_1}, 1, 1, \alpha_1)$, std.dev.~$\sigma = 0.1$}
    \\ \hline 
    Mean of error squares
    &
      $\alpha_1=1/100$
    &
      $\alpha_1=1/10$
    &
      $\alpha_1=10$
    &
      $\alpha_1=100$
    \\ \hline 
    \footnotesize{$\| \hat u_\theta^{\dpg} - \hat u_h^{\dpg} \|_{L_2}^2$} \hfill DPG:
    &
      7.053e-08   
    &
      2.300e-07   
    &
      5.372e-07   
    &
      2.432e-05   
    \\
    \footnotesize{$\| u_\theta^{\fosls} - u_h^{\fosls} \|_{L_2}^2$}\;\; \hfill FOSLS:
    &
      1.995e-07   
    &
      4.475e-07   
    &
      2.225e-06   
    &
      1.588e-01   
    \\
    \hline 

    \footnotesize{$\| \hat q_\theta^{\dpg} - \hat q_h^{\dpg} \|_{L_2}^2$} \hfill DPG:
    &
      3.712e-02   
    &
      1.467e-03   
    &
      1.120e-04   
    &
      8.123e-03   
    \\
    \footnotesize{$\| q_\theta^{\fosls} - q_h^{\fosls} \|_{L_2}^2$}\;\; \hfill FOSLS:
    &
      3.897e-02   
    &
      1.644e-03   
    &
      2.080e-04   
    &
      1.858e-02   
    \\
    \hline
    \footnotesize{$\Ldpg(\uu^\dpg_\theta)$}\;\; \hfill DPG:
    &
      1.395e-02   
    &
      1.005e-02   
    &
      9.959e-02   
    &
      10.260      
    \\
    \footnotesize{$\Lfosls(\uu^\fosls_\theta)$}\;\; \hfill FOSLS:
    &
      1.536e-02   
    &
      8.063e-03   
    &
      1.058e-01   
    &
      1.656       
    \\
    \hline
  \end{tabular}
  \caption{NN prediction errors and losses as parameter  means are varied}
  \label{tab:vary-mean}
\end{table}

We observe from Table~\ref{tab:vary-mean} that the mean squared error
values  are small and comparable for both methods.  The errors do
increase for more challenging~$\alpha$. The loss values reported  in the last
two rows are expected (by Propositions~\ref{prop:Lfosls_VC} and~\ref{prop:VC-DPG}) to be proportional to the error in a stronger norm 
which also measure fluxes in $H({\rm div};\Omega)$.
Due to the dependence of the proportionality constants between the loss and the squared
errors (in the stronger norm) on the parameters, this error-residual relation may degrade for samples with degenerating ellipticity
(a point taken up systematically later in Section~\ref{sec:param-robust}).
We emphasize that these comparisons between true finite element
solutions and NN predictions are presented not to validate the method
(which is already done by the theory), but rather to illustrate the practical 
learning success.

\subsection{Performance when varying the standard deviation of $\alpha$}

In the next experiment, we fix the mean $\bar \alpha$ and train the NN
with various standard deviation choices of $\sigma$ in
$\{ 0.1, 0.5, 1, 5, 10\}$ using either the DPG loss function or the
FOSLS loss function. We compute errors as in the previous table. They
are reported in Table~\ref{tab:vary-std}. Satisfactory performance is
again observed for both methods as measured by the mean
over ten thousand test samples of errors and losses. 

\begin{table}[htbp]
  \centering
  \begin{tabular}{l|c|c|c|c|c} 
    \hline 
    \multicolumn{6}{c}{From $\alpha$ distributions with mean
    $\bar\alpha = (0.01, 1, 1, 0.01)$ and {std.dev.~$\sigma$}}
    \\ \hline 
    Mean of error squares
    &
      $\sigma = 0.1$    
    &
      $\sigma = 0.5$    
    &
      $\sigma = 1$      
    &
      $\sigma = 5$      
    &
      $\sigma = 10$     
    \\ \hline 
    \footnotesize{$\| \hat u_\theta^{\dpg} - \hat u_h^{\dpg} \|_{L_2}^2$}\hfill DPG:
    & 
      7.058e-08
    &
      3.070e-05
    &
      5.664e-05
    &
      8.614e-03
    &
      4.031e-01
    \\
    \footnotesize{$\| u_\theta^{\fosls} - u_h^{\fosls} \|_{L_2}^2$}\;\; \hfill FOSLS:
    &
      1.324e-07
    &
      3.799e-05
    &
      1.253e-04
    &
      4.135e-01
    & 
      6.185
    \\
    \hline 

    \footnotesize{$\| \hat q_\theta^{\dpg} - \hat q_h^{\dpg} \|_{L_2}^2$}\hfill DPG:
    & 
      4.213e-02
    &
      1.345e-02
    &
      1.684e-02
    &
      5.890e-03
    &
      1.803e-02
    \\
    \footnotesize{$\| q_\theta^{\fosls} - q_h^{\fosls} \|_{L_2}^2$}\;\; \hfill FOSLS:
    &
      3.234e-02
    &
      1.472e-02
    &
      1.138e-02
    &
      4.735e-02
    & 
      4.909e-02
    \\
    \hline
    \footnotesize{$\Ldpg(\uu^\dpg_\theta)$}\;\; \hfill DPG:
    &
      1.175e-02
    &
      6.029e-02
    &
      1.205e-02
    &
      2.998e-01
    &
      1.080
    \\
    \footnotesize{$\Lfosls(\uu^\fosls_\theta)$}\;\; \hfill FOSLS:
    &
      9.573e-03
    &
      8.355e-02
    &
      7.910e-03
    &
      7.871e-01
    & 
      51.586
    \\
    \hline 
  \end{tabular}
  \caption{NN prediction errors and losses as parameter standard deviations are varied for the 4-dimensional parameter set}
  \label{tab:vary-std}
\end{table}

To conclude this section, we highlight an observation from both
Tables~\ref{tab:vary-mean} and~\ref{tab:vary-std}.  We see a slight
improvement in the NN prediction error for the DPG case for
extreme values of $\alpha_1$ in Table~\ref{tab:vary-mean}, namely
$\alpha_1 = 100^{-1}$ and $100$. Similarly, in
Table~\ref{tab:vary-std}, we see that the DPG error squared means are
slightly better for larger standard deviations. This sets the stage
for further investigations in Section~\ref{sec:param-robust}  of parameter robustness of DPG loss functions and more gains that can be
extracted from tailored DPG loss functions.

\subsection{Higher dimensional parameters}
\label{ssec:16-dim}

\begin{figure}
\centering
\begin{tikzpicture}
\begin{semilogyaxis}[
    width=15cm,
    height=7cm,
    xlabel={Epochs},
    ylabel={Batch-averaged value of $\Ldpg$},
    scaled x ticks=false,
    xtick={0, 5000, 10000, 15000},
    xticklabels={0, 5000, 10000, 15000},
    legend pos=north west,
    xmajorgrids = false,
    xminorgrids = false,
    ymajorgrids = true,
    yminorgrids = false,
    legend cell align=left,
]

\addplot+[const plot, thick, color=tabblue, mark=none] table [x=Epoch, 
y=Loss, col sep=comma] {checkboard_training_hist_finest.csv};
\end{semilogyaxis}
\end{tikzpicture}
\caption{Evolution of the batch-averaged $\Ldpg$ loss over 15,000 training 
epochs, corresponding to the training process of a parameter-to-solution 
surrogate with the 16-dimensional parameter domain.}
\label{fig:checkboard-training-hist}
\end{figure}
\begin{figure}
  \centering
  \subcaptionbox
  {$\hat u_\theta$
  }[0.45\textwidth]
  {
    \includegraphics[scale=0.235,
    trim={100  100  40  80}, clip]
    {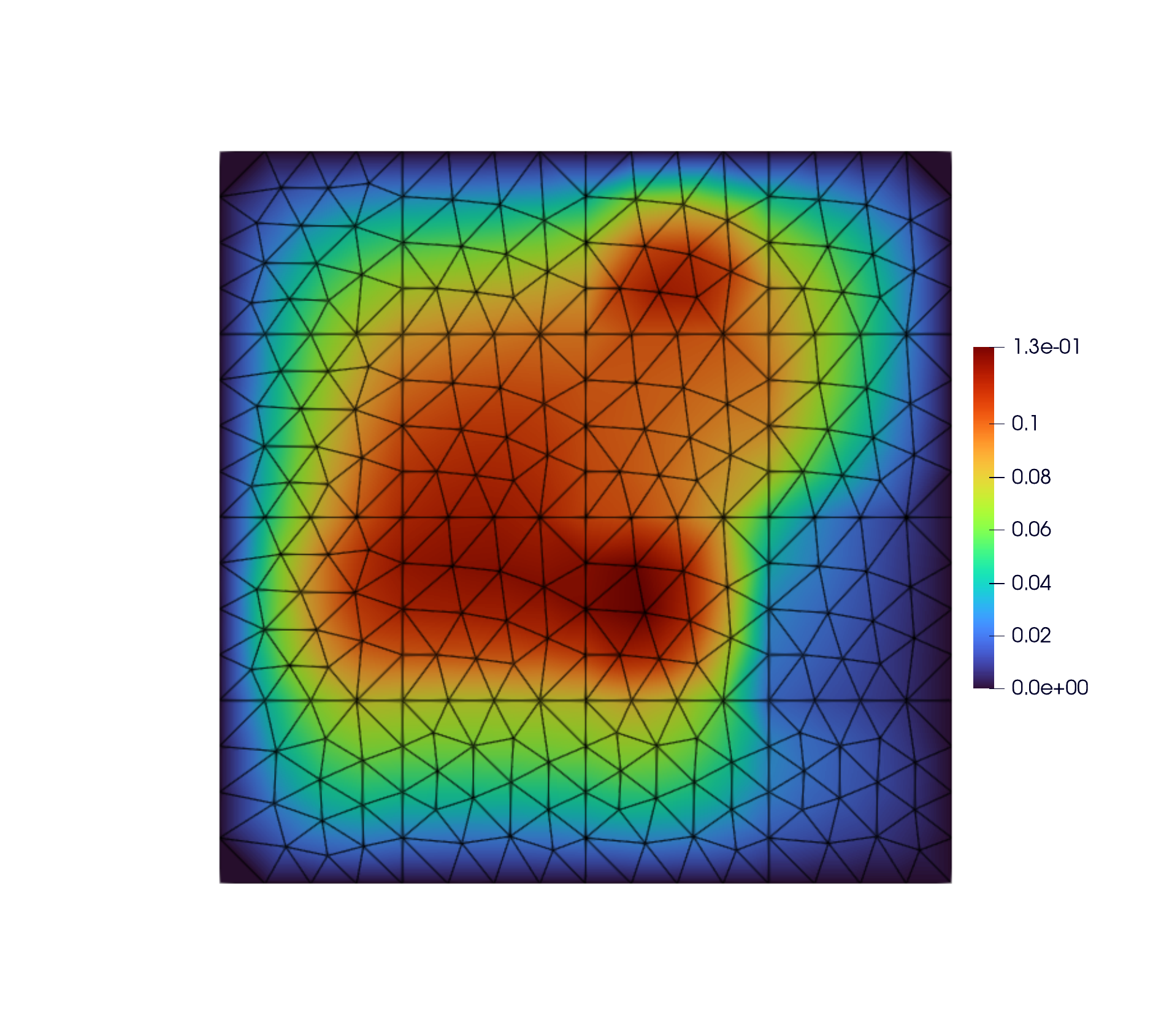}
  }
  \qquad
  \subcaptionbox
  {$\hat u_h$
  }[0.45\textwidth]
  {
\includegraphics[scale=0.235,
    trim={100  100  40  80}, clip]
    {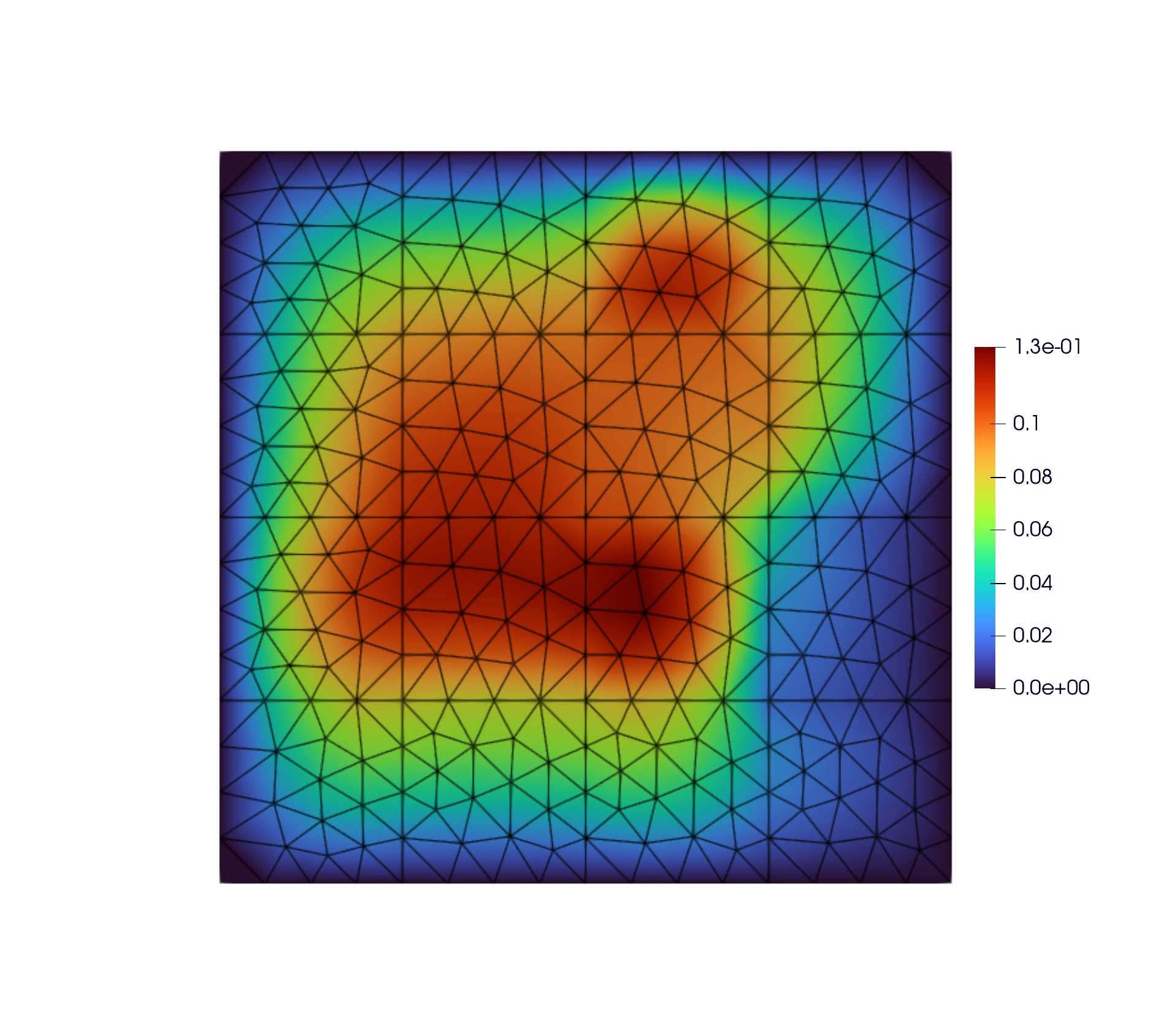}
  }
  \qquad  
  \subcaptionbox
  {$\hat q_{\theta}$
  }[0.45\textwidth]
  {
\includegraphics[scale=0.235,
    trim={100  100  40  80}, clip]
    {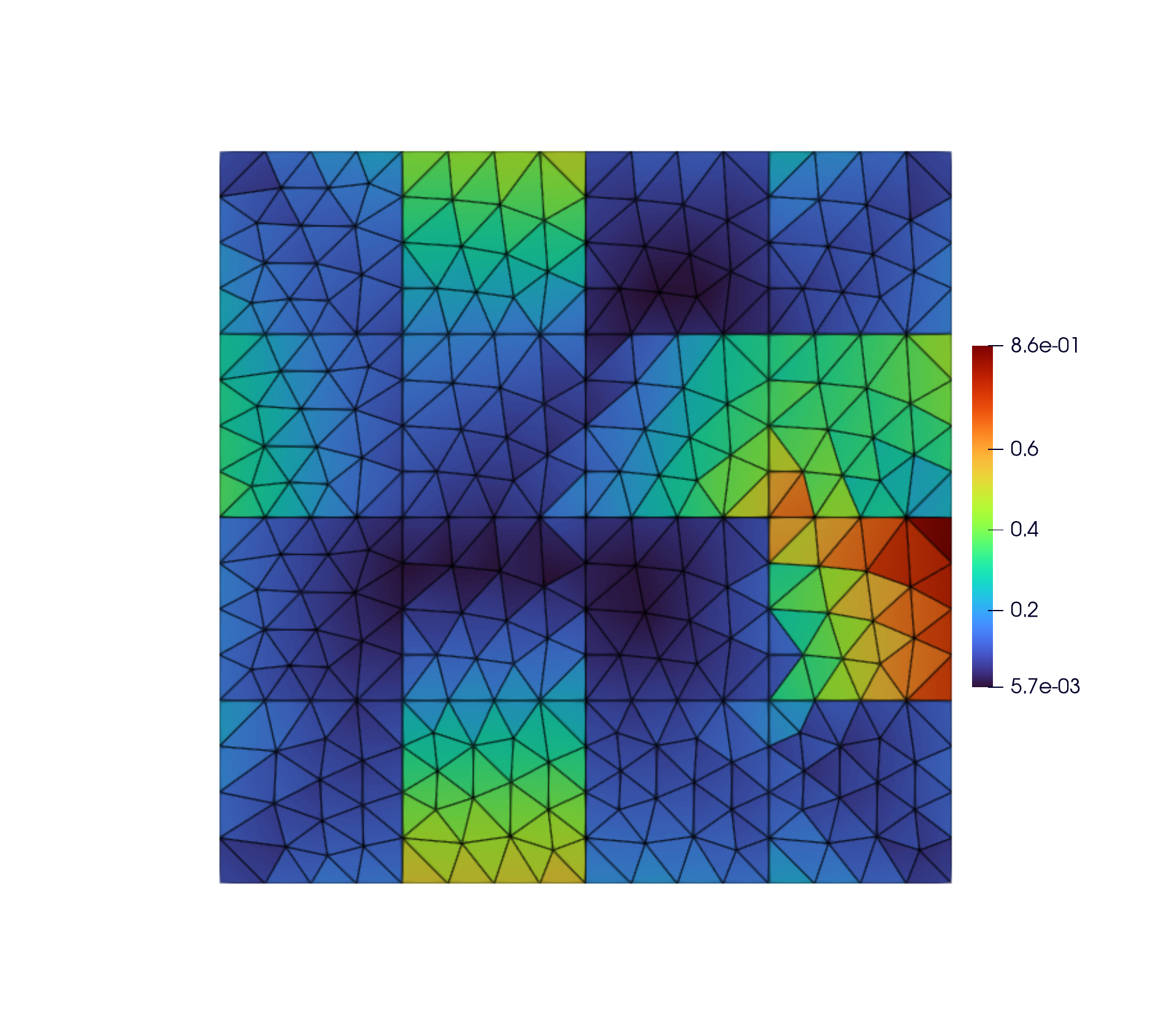}
  }
  \qquad 
  \subcaptionbox
  {$\hat q_h$
  }[0.45\textwidth]
  {
\includegraphics[scale=0.235,
    trim={100  100  40  80}, clip]
    {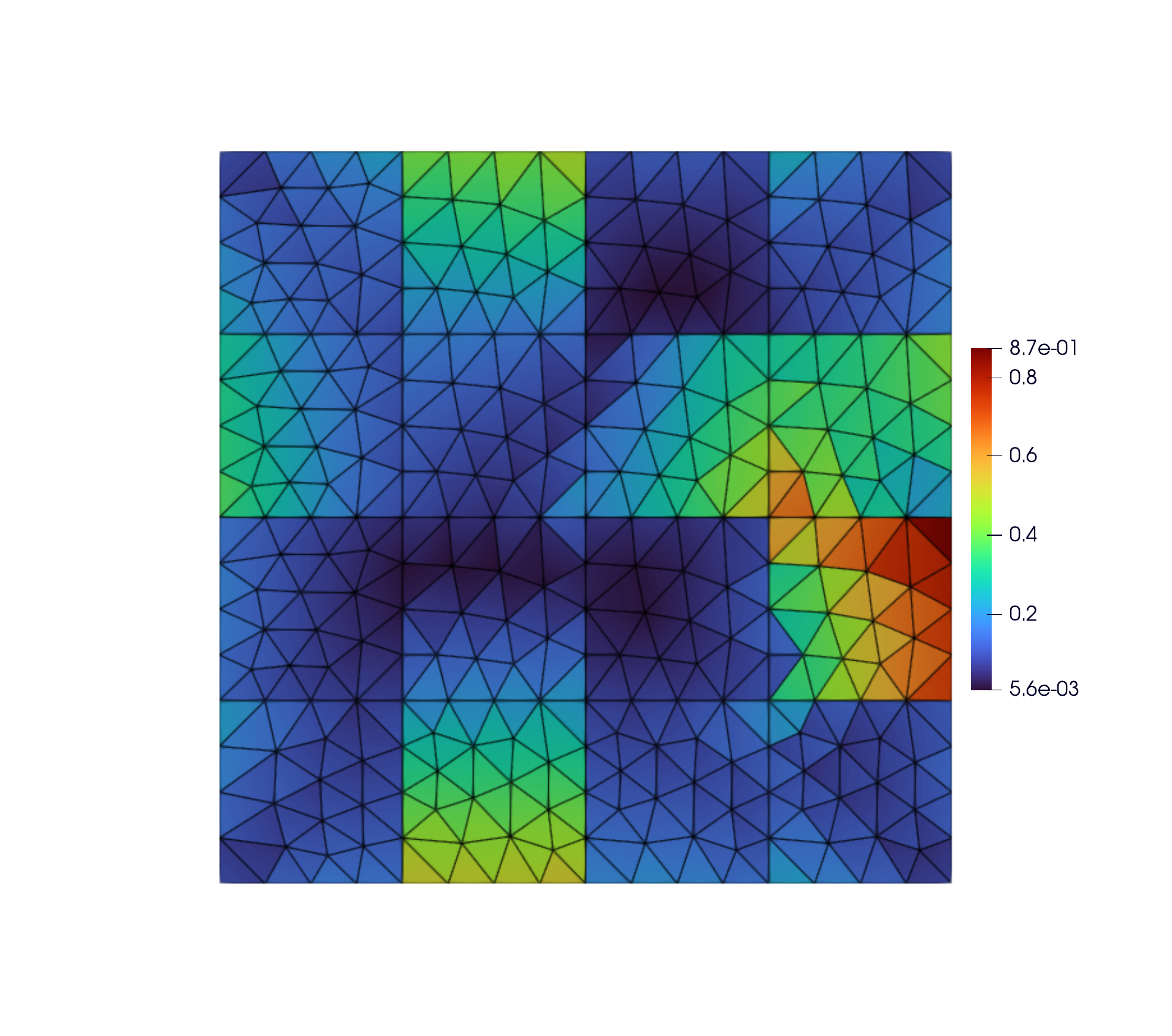}
  }
  \caption{Left: Visualization of the DPG-NN solution predictions for
    an $\alpha$ from the $16$-dimensional $\cl Q$. Right:
    Visualization of the exact finite element solution for the same
    $\alpha$.}
  \label{fig:checkboardparaview_u}
\end{figure}

To show the ease with which our setting generalizes to higher
dimensional parameters, we now report performance from the DPG-NN trained
over a 16-dimensional parameter space.  In this subsection alone, we
set $m_\alpha = 16$ and let
$\om$ be subdivided into 16 congruent subdomains, as shown in
Figure~\ref{fig:param-16}, partitioned into a finite element mesh of
$h \approx 0.06$. We consider  piecewise constant $\alpha$ from $\chi^2_{\bar\alpha,\sigma}$ with
$    \sigma = 0.5,
$
and
$    \bar \alpha = (6, 1, 2, 1, 8, 6, 10, 1, 2, 1, 1, 2, 1, 1, 6, 2)
$ 
under the ordering illustrated by the colors in
Figure~\ref{fig:param-16}.  Using the same neural network architecture
and parameters as before,
we carry our the  training process 
over a  set with $M=8192$ random samples
for a fixed duration
of 15,000 epochs. The training history, reported in
Figure~\ref{fig:checkboard-training-hist}, shows a steady decrease in the batch-averaged $\Ldpg$ loss over the training epochs, indicating effective learning of the surrogate model even in this higher-dimensional parameter setting.

To give a reproducible evaluation of the practical performance of the
resulting surrogate, we compute the average of the squared $L_2$
errors of the predictions relative to the DPG finite element solution,
which is obtained using the same mesh. This evaluation is performed
over a test set of 5,000 samples drawn from the same distribution as
the training set. We also compute the average of the $\Ldpg$ loss over
the test set, evaluated using the full group variable prediction
$\uu^{\dpg}_\theta = (q_{0, \theta}, u_{0, \theta}, \hat u_\theta,
\hat q_\theta)$. The results there, summarized in
Table~\ref{tab:checkboard-errors}, show that the surrogate predictions
are accurate when compared to ground truth $\fdpg(\alpha)$. The loss
values, by variational correctness, track the square of the distance
from the exact solution $\uu^\dpg$ to the prediction. They are also
small as seen from the table. As the test sample size increases, these
numbers converge to a spatio-parametric integral, so they are
reproducible.
\begin{table}[htbp]
  \centering
  \begin{tabular}{l|c}
    \hline 
    \text{Error Metric} & \text{Mean of error squares} \\ 
    \hline \hline
    \footnotesize{$\| u_\theta - u_h \|_{L_2}^2$} & 8.806e-06 \\ 
    \hline
    \footnotesize{$\| q_\theta - q_h \|_{L_2}^2$} & 2.330e-03 \\ 
    \hline
    \footnotesize{$\| \hat u_\theta - \hat u_h \|_{L_2}^2$} & 8.878e-06 \\ 
    \hline
    \footnotesize{$\| \hat q_\theta - \hat q_h \|_{L_2}^2$} & 2.330e-03 \\ 
    \hline
    \footnotesize{$\Ldpg(\uu^{\dpg}_\theta)$} & 1.267e-02 \\
    \hline 
  \end{tabular}
  \caption{Mean squared $L_2$ errors of the DPG-NN predictions from
    the 16-dimensional parameter set sampled over a test set of
    5,000 samples.}
  \label{tab:checkboard-errors}
\end{table}

Figure~\ref{fig:checkboardparaview_u} provides a visual comparison
between the surrogate prediction and the DPG finite element solution
for a randomly drawn sample $\alpha=(3.13, \allowbreak 0.85, \allowbreak 2.73, 
\allowbreak 1.38, \allowbreak 5.65, \allowbreak 2.74, \allowbreak 9.54, 
\allowbreak 0.05, \allowbreak 2.07, \allowbreak 1.10, \allowbreak 0.07, 
\allowbreak 1.17, \allowbreak 2.43, \allowbreak 1.34, \allowbreak 11.13, 
\allowbreak 3.08)$. These plots are generated by considering the unique 
extension of the trace variables $(\hat q_\theta, \hat u_\theta)$ from 
the DPG-NN prediction and $(\hat q_h, \hat u_h)$ for the ground truth finite 
element solution. Visually, the left and right plots are almost identical.

\subsection{Experimental comparison with PINN}
\label{ssec:compare-pinn}

We now compare results from an independent  well-established implementation of PINN
in NeuralPDE~\cite{Zubov2021}, an efficient Julia library package,
with our DPG-NN results. In standard PINN, as already
disussed---see~\eqref{PINN}---the training process of the surrogate is
carried out by minimizing the strong form of the PDE and boundary residual (as in \eqref{laplaceDP}), which
is evaluated at collocation points randomly sampled across   the
spatial domain, the boundary, and the parameter space.
Specifically, for a given spatial
coordinate $(x, y)$ and parameter realization ${a}$, the
point-wise residual loss of PINN implemented in~\cite{Zubov2021} is
\begin{equation}\label{eq:pinnLoss}
    \mathcal{L}^{\rm pinn}(u)(x, y) := 
    |\nabla \cdot ( {a} \nabla u ) + f|^2(x, y)
    +\, \chi_{\partial \Omega}(x,y)\, |u|^2(x, y),    
\end{equation}
where $\chi_{\partial \Omega}(x,y)$ denotes the indicator function 
of $\partial \Omega$.
Each fixed diffusion coefficient ${a}$ is a 4-dimensional
piecewise constant function over $\Omega$, which represents the
pointwise inverse of our resistivity parameter $\alpha$ needed in
the first-order formulation. 
Each sample ${a}_i$ is drawn from a uniform distribution with
a mean of $(1, \bar {a}, \bar {a}, 1)$ and a width of $\pm 0.5$, denoted by $U_a$. As $\bar a$ increases, $U_a$ contains more  high-contrast parameter samples. Our training attempts
with~\cite{Zubov2021} often failed for unbounded parameter
distributions, so we have chosen to make the problem simpler (in this
subsection alone) by restricting to the uniform distribution on the
bounded 4-dimensional box.
We use a training set of $M=1024$
samples for both the parameter space and the 
spatial collocation points over $\om$, an additional independent set of points 
of the same size is also sampled over $\partial \Omega$, in order to feed 
the boundary term of the loss.

\begin{table}
  \centering
  \begin{tabular}{c|c|r|r|r}
    \hline 
    \hline 
    Epochs & Relative Errors
    &
      $\bar {a}=\bar \alpha=1$
    &
      $\bar {a}=\bar \alpha=5$
    &
      $\bar {a}=\bar \alpha=10$
    \\ \hline \hline
   
    \multirow{4}{*}{20000} 
           & \footnotesize{$\varrho^{\rm dpg}_{L_2, \rm rel}$}
    &
      0.12\%  
    &
      0.43\%  
    &
      1.42\%  
    \\
           & \footnotesize{$\varrho^{\rm pinn}_{L_2, \rm rel}$}
    &
      1.67\%  
    &
      10.62\%  
    &
      45.16\%  
    \\
    \cline{2-5} 

           & \footnotesize{$\varrho^{\rm dpg}_{H^1, \rm rel}$}
    &
      1.61\%  
    &
      1.81\%  
    &
      1.96\%  
    \\
           & \footnotesize{$\varrho^{\rm pinn}_{H^1, \rm rel}$}
    &
      4.18\%  
    &
      24.38\%  
    &
      67.93\%  
    \\
    \hline \hline
   
    \multirow{4}{*}{100000} 
           & \footnotesize{$\varrho^{\rm dpg}_{L_2, \rm rel}$}
    &
      0.84\%  
    &
      4.33\%  
    &
      0.43\%  
    \\
           & \footnotesize{$\varrho^{\rm pinn}_{L_2, \rm rel}$}
    &
      1.91\%  
    &
      6.26\%  
    &
      9.53\%  
    \\
    \cline{2-5} 

           & \footnotesize{$\varrho^{\rm dpg}_{H^1, \rm rel}$}
    &
      1.92\%  
    &
      8.85\%  
    &
      1.21\%  
    \\
           & \footnotesize{$\varrho^{\rm pinn}_{H^1, \rm rel}$}
    &
      4.58\%  
    &
      16.11\%  
    &
      26.55\%  
    \\
    \hline 
  \end{tabular}
  \caption{Comparison of relative generalization errors across
    relevant parameter distributions after 20,000 and 100,000 training
    epochs. 
  }
  \label{tab:Vanilla-error-comparison}
\end{table}

We set the code in~\cite{Zubov2021} to work with a 4-layer fully
connected neural network with $64 \times 64 = 4096$ neurons per layer.
We selected this architecture since our training attempts failed often
for PINN with wider or deeper dense nets, as well as with our prior
sparse ResNet architecture of \S\ref{ssec:neural-network}. (In
striking contrast, {\em all} trainings of DPG-NN converged.)  Further
reports of PINN training failures can be found in the
literature~\cite{BDO}.  After much experimentation, we settled on the
above-mentioned fully connected architecture which seems to give
convergent PINN trainings.  The input layer is 6-dimensional, expecting
inputs of the form $ ({a}, x, y) \in
U_{{a}} \times [0,1] \times [0,1] $ and the 1-dimensional output layer
delivers the approximation for $u$ as a function of the all these
inputs. Enforcement of boundary conditions, while easy in our finite element approach even on complicated spatial domains, is not so easy in PINN.
But having selected the simple square $\om$ for this study, we can
enforce the homogeneous Dirichlet boundary conditions by the
so-called hard-constrained PINN technique, i.e., we
multiply the raw output of the neural network by a boundary cap 
function  $ 160[xy(1-x)(1-y)]$ that vanishes on the boundary of our $\om$.
We opted for this approach since our attempts
to enforce boundary conditions through 
Tychonov-regularization alone did not succeed in 
all training instances.

Next, for a fair comparison, we reduce our ResNet
architecture for DPG-NN presented in \S\ref{ssec:neural-network},
scaling it down as follows: fix depth to $l = 4$ layers, width to
$N = 128$, and rank to $r = 16$, yielding 4,112 degrees of freedom per layer
(comparable to the above-mentioned $4,096$ degrees of freedom layer for PINN).  To achieve a
spatial resolution analogous to the 1,024 spatial collocation points
used by the standard PINN, we fix a finite element mesh with 2,120
elements bounded by a maximum diameter of $h \approx 0.0333$. This
provides 1,121 spatial degrees of freedom for the space $\hU_1$, of
which 1001 are the true degrees of freedom unconstrained by boundary
conditions. For training this DPG-NN comparator, we draw $M=1024$
samples from a uniform distribution, denoted by   $U_{\alpha}$, consisting of parameter values of the form $(\bar \alpha, 1, 1, \bar \alpha)$ with distribution width $\pm 0.5$, 
and use the same prior batch size of~32.

To evaluate the performance of both networks, we compare their
predictions against ``reference'' finite element solutions
$\uu_* = (q_*, u_*, \hat u_* , \hat q_*)$ for
$\Mtest=1000$
parameter samples.  The reference solutions are practically
chosen (to stay within the limit of our computational time allocations)
by the DPG discretization on a finer mesh of $h \approx 0.015$ and
higher polynomial degree,
$ \uu_* \in \Xhdpg = \cl P_2^2 \times \cl P_2 \times \hU_3 \times
\mathcal{RT}_2.
$  To understand the performance in the
terms of the components of the model-compliant $H^1$-norm metric, we
choose relative errors in the $L_2$-norm and the
$H^1$-seminorm. Namely, using the DPG-NN predicted components
$\hat u_\theta^\dpg$ and the PINN predition $u_\theta^\pinn$, we
compute
\begin{align}\label{eq:pinn-err-ratios}
    \varrho^{\pinn}_{L_2, \rm rel} 
  & =
    \frac{\sum_{i=1}^{\Mtest} 
    \|u^\pinn_{\theta}({a}_i) - u_*({a}_i)\|^2_{L_2(\Omega)}}
    {\sum_{i=1}^{\Mtest} \|u_*({a}_i)\|^2_{L_2(\Omega)}}, 
  &
    \varrho^{\dpg}_{L_2, \rm rel} 
  & = 
    \frac{\sum_{i=1}^{\Mtest} 
    \|\hat u^\dpg_{\theta}(\alpha_i) - \hat u_*(\alpha_i)\|^2_{L_2(\Omega)}}
    {\sum_{i=1}^{\Mtest} \|\hat u_*(\alpha_i)\|^2_{L_2(\Omega)}}, 
\end{align}
as well as the analogously defined $\varrho^{\pinn}_{H^1, \rm rel}$ and
$\varrho^{\dpg}_{H^1, \rm rel}$ obtained by replacing the $L_2(\om)$-norm
by the $H^1(\om)$-seminorm above. For terms with the interface variable, we compute the norms using their shape function extensions into the elements, and 
furthermore, in order to calculate derivatives of the PINN predictions,
a post-processing that consists of interpolating them into the high fidelity 
finite element space $\mathcal{U}_3$ over the fine mesh $h\approx0.015$ was 
needed, a processing that acts in favor of PINN, since DPG-NN predictions are 
limited by the best approximation error intrinsic to the lowest order 
ultraweak DPG spaces defined over the coarse mesh. 
Observe that in~\eqref{eq:pinn-err-ratios} the left sums are over random 
samples in $U_a$ while the right sums are over $U_\alpha$. Note that the 
numerators and the denominators in these ratios converge to the respective
spatio-parametric Bochner integrals in the limit of large test
sets.  Only such relative error measures make sense for comparison
since the magnitudes of the reference solutions vary significantly
depending on whether we use the resistivity $\alpha_i$ parameter or 
the diffusivity parameter~$a_i$.

These $\varrho$-error ratios are computed and tabulated in
Table~\ref{tab:Vanilla-error-comparison} as percentages. Both NNs were
trained for 20,000 epochs first to extract these ratios in the first
rows. Results after a more extensive training up to 100,000 epochs are also 
offered for comparison. Across all entries of the table,
we find that the variationally correct DPG-NN's $\varrho$-error metrics are
better than those of standard PINN. Moreover, as the parameter
probability distributions are changed to include more challenging
parametric coefficients in either case, the DPG-NN's gains become more
pronounced. It is also evident DPG-NN delivers good quality predictions with 
fewer training epochs: while further training epochs leads to some  
improvements, they appear to be an overkill for DPG-NN.

Before concluding this comparison discussion, let us also discuss 
Figure~\ref{fig:DPG-Vanilla-training-hist} showing the evolution of
the batch-averaged losses $\Ldpg$ and $\mathcal{L}^{\pinn}$ over
100,000 epochs for the models trained with
$\bar \alpha = \bar {a} = 10$.  The reductions in DPG losses over the
epochs are meaningful since the DPG losses are variationally correct
(and loss reductions correlate to error reductions).  It is also
instructive to see the solution predictions at the training
distribution mean of both models. The solution predictions from PINN
and DPG-NN, both trained over 100,000 epochs, denoted by
$u^{\pinn}_{\theta}$ and $u^{\dpg}_{\theta}$ together with the
fine-grid reference solution $\hat u_*$ are displayed in
Figure~\ref{fig:vanillaparaview}. They provide visual examples of
errors that cause the differences in the quantitative error measures
of Table~\ref{tab:Vanilla-error-comparison}.

\begin{figure}
\centering
\begin{tikzpicture}
\begin{semilogyaxis}[
    width=15cm,
    height=7.5cm,
    xlabel={Epochs},
    ylabel={Batch-averaged  $\Ldpg$ and 
    $\mathcal{L}^{\rm pinn}$}, 
    scaled x ticks=false,
    xtick={0, 50000, 100000},
    xticklabels={$0$, $50000$, $100000$},
    legend pos=north east,
    xmajorgrids = false,
    xminorgrids = false,
    ymajorgrids = true,
    yminorgrids = false,
    legend cell align=left,
]

\addplot+[const plot, thick, color=tabblue, mark=none] table [x=epoch, y=train_loss, col sep=comma] {training_history_dpgnn_100kepochs_finer.csv};
\addlegendentry{DPG-NN}

\addplot+[const plot, thick, color=tabred, mark=none] table [x=Iteration, y=Loss, col sep=comma] {training_history_vanilla_pinnT_100kiter_finer.csv};
\addlegendentry{PINN}

\end{semilogyaxis}
\end{tikzpicture}
\caption{The training histories of DPG-NN and PINN models. 
}
\label{fig:DPG-Vanilla-training-hist}
\end{figure}

\begin{figure}
  \centering
  \subcaptionbox
  {$u^{\rm pinn}_\theta$
  }[0.3\textwidth]
  {
    \includegraphics
    [height=4.4cm,
    trim={115  120  45  75}, clip]
    {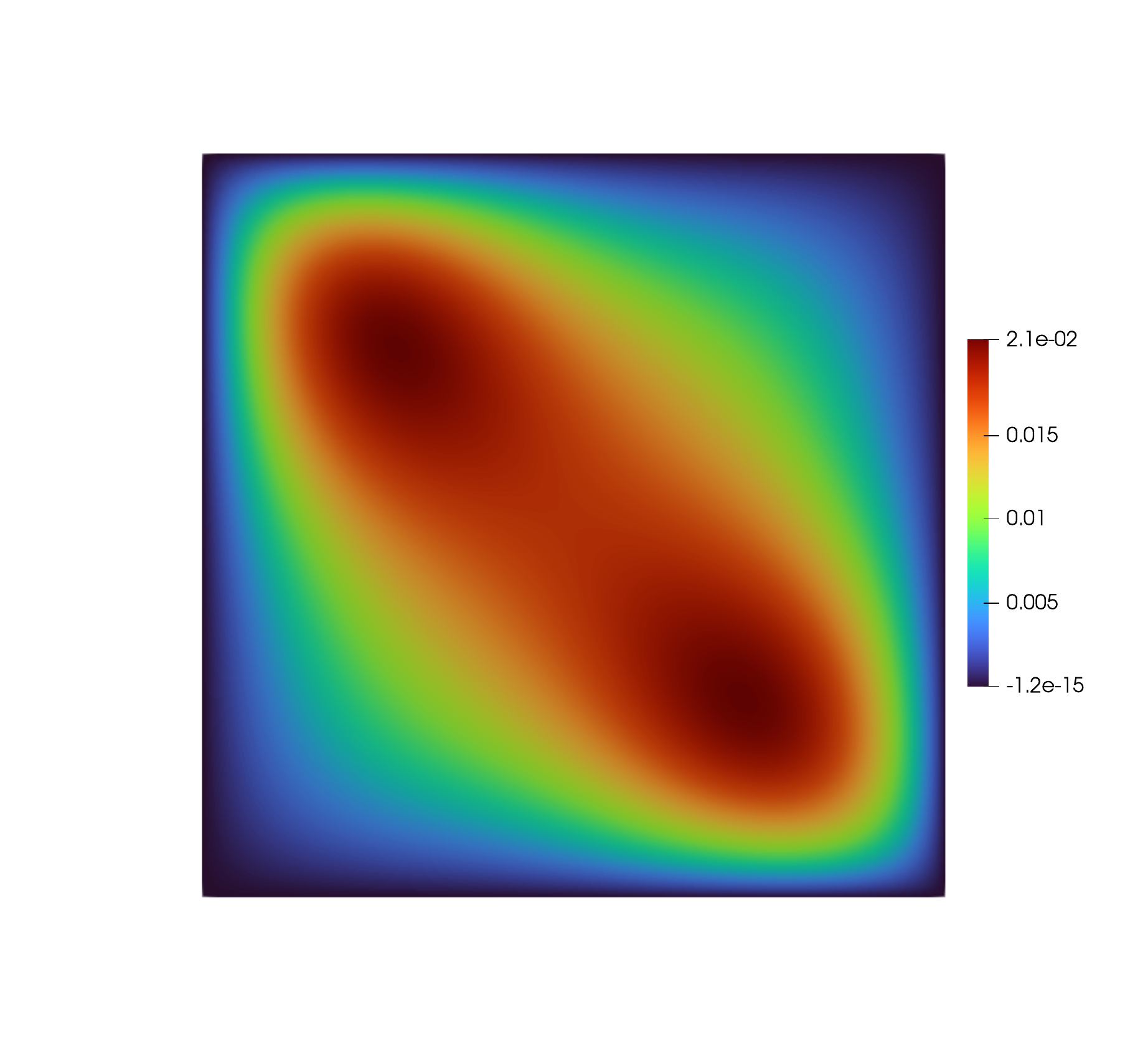}
  }
  \quad  
  \subcaptionbox
  {$\hat u^\dpg_{\theta}$
  }[0.3\textwidth]
  {
    \includegraphics[height=4.3cm,
    trim={100  100  40  80}, clip]
    {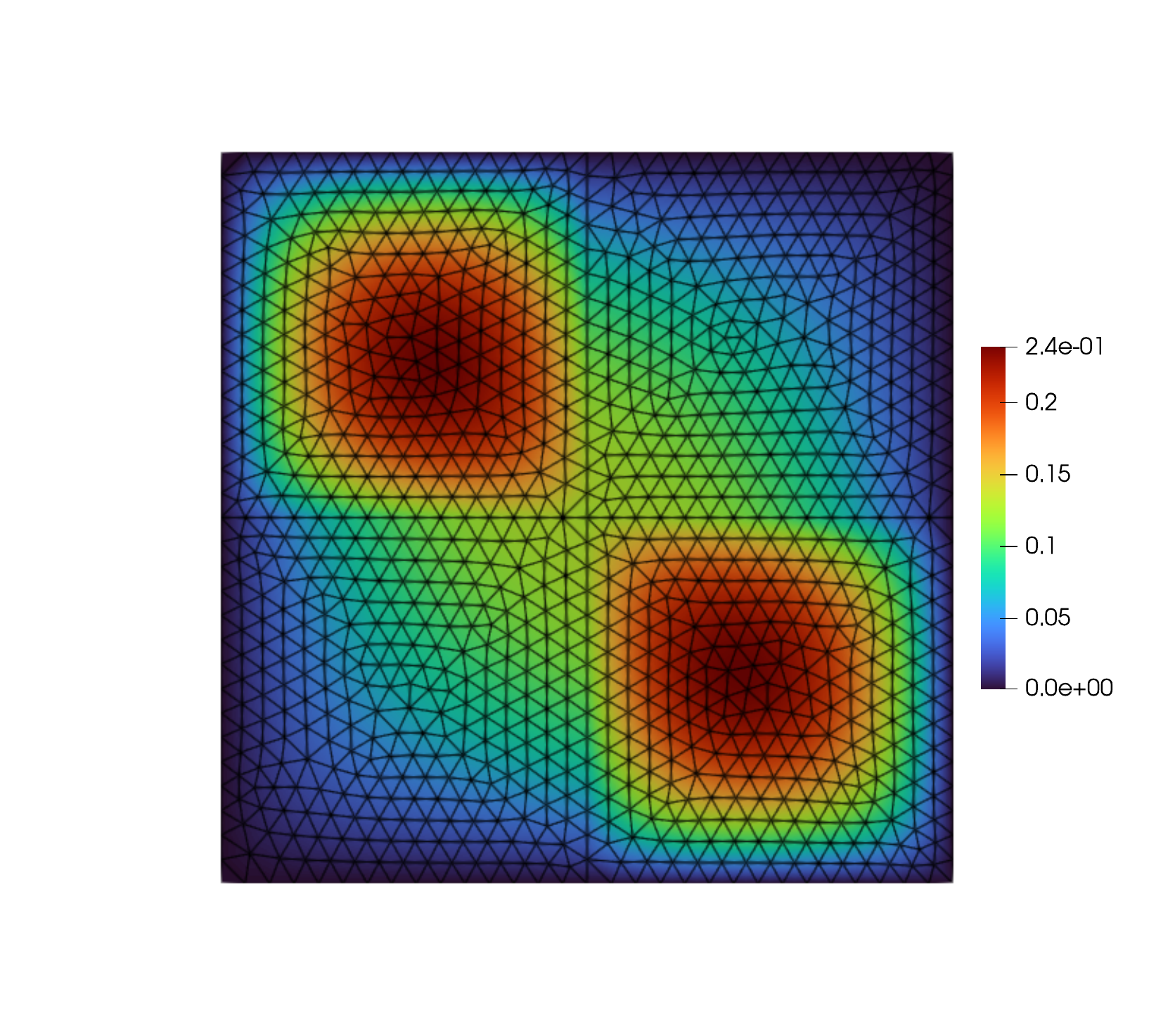}
  } \quad
  \subcaptionbox
  {$\hat u_*$
  }[0.3\textwidth]
  {
    \includegraphics[height=4.3cm,
    trim={100  100  40  80}, clip]
    {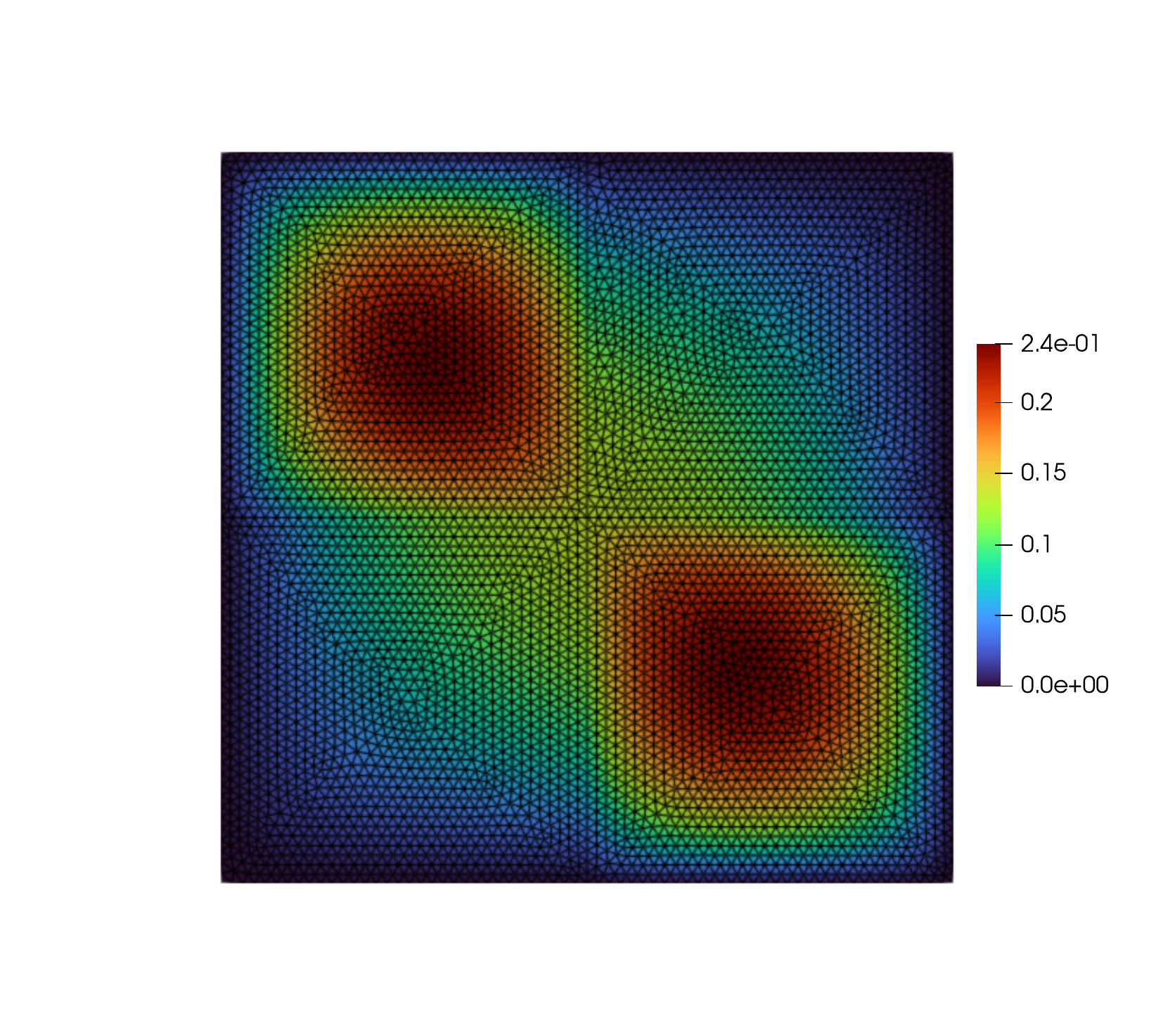}
  }
  \\
  \subcaptionbox
  {$|\nabla u^{\rm pinn}_\theta|$
  }[0.3\textwidth]
  {
    \includegraphics[height=4.4cm,
    trim={115  120  45  75}, clip]
    {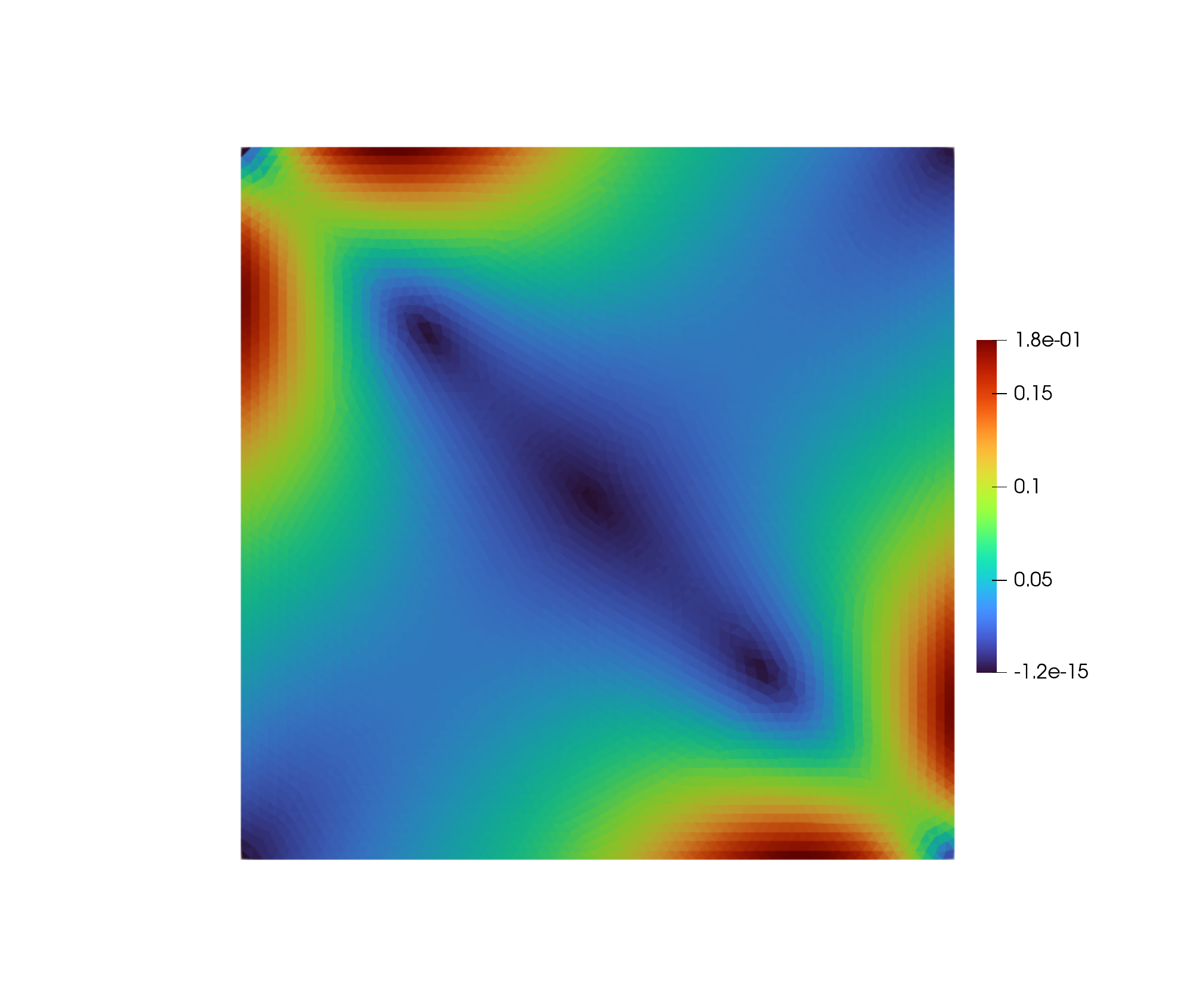}
  } \quad
  \subcaptionbox
  {$|\nabla \hat u^\dpg_{\theta}|$
  }[0.3\textwidth]
  {
    \includegraphics[height=4.35cm,
    trim={115  120  50  75}, clip]
    {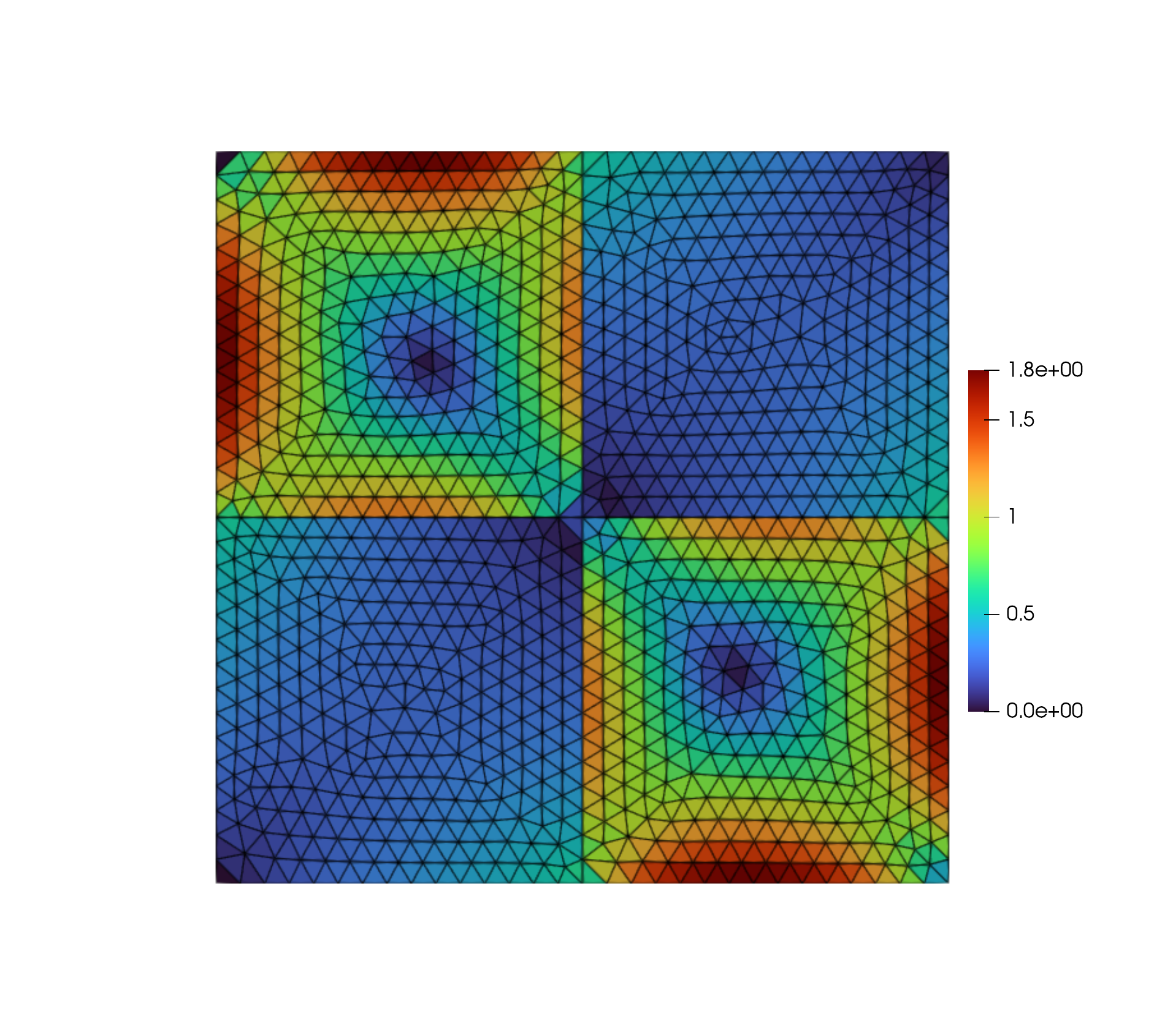}
  } \quad
  \subcaptionbox
  {$|\nabla \hat u_*|$
  }[0.3\textwidth]
  {
    \includegraphics[height=4.35cm,
    trim={115  120  45  75}, clip]
    {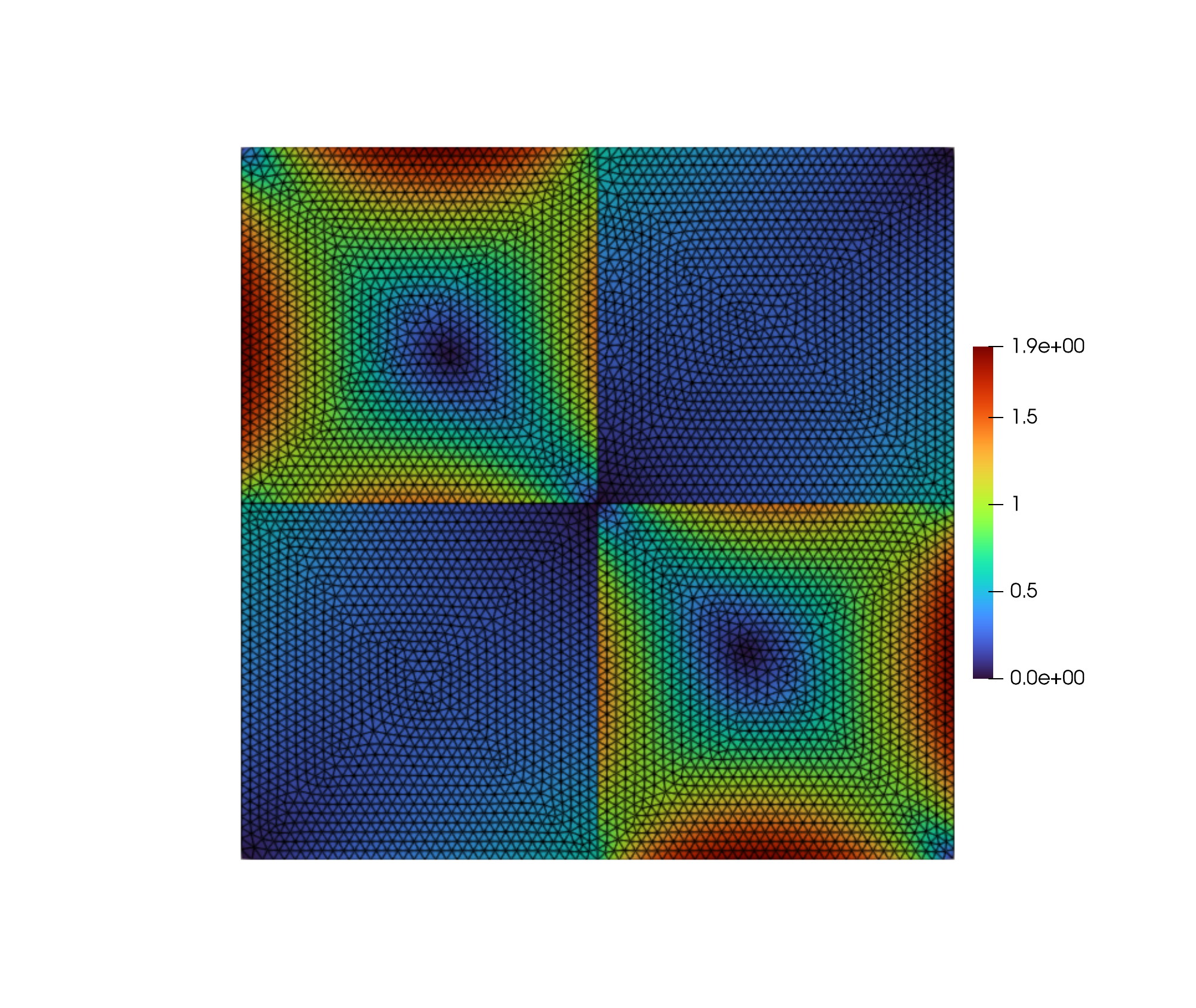}
  }
  \caption{
    The top row displays the solution predictions 
    and the bottom row shows the magnitudes of their gradients. The columns show
     PINN prediction, DPG-NN prediction, and the fine-grid reference ground truth,
    respectively.    
  }
  \label{fig:vanillaparaview}
\end{figure}

\subsection{Ablation study}
\label{ssec:ablation}

The degrees of freedom needed to construct a parameter-to-solution
surrogate consists of its trainable weights $\theta$ as well as the finite
element dimensions represented by the hyperparameter~$h$. We expect that reducing $h$
reduces the best approximation error and unlocks better potential
minima during training. We also anticipate that
expanding the NN architecture
enhances its representational capacity and may reduce the
generalization error by better fitting the nonlinear
parameter-to-solution maps.

To view the effects of hyperparameters $h$ and $\theta$, we
now conduct an ablation study that consists of training several models
with different budgets of degrees of freedom over a fixed, demanding
$\chi^2_{\bar \alpha, \sigma}$ distribution with
$\bar\alpha = (0.1, 5.0, 5.0, 0.1)$ and $\sigma = 1.0.$
To ensure that the training set is reliable in representing the continuous 
parameter space, we consider a large  number of training samples,
namely~8,192 samples.
Furthermore, to ensure that the training process of each model is
carried out until exhaustion, rather than utilizing the prior default
of uniformly fixed number of epochs, we now use a dynamic, online
stopping criterion based on the model's performance over a validation
set. This validation set consists of a fixed, independent random
sample of 1,024 instances drawn from the same distribution as the
training set. Specifically, every 100 epochs, we calculate the average
loss of the model's predictions over this validation set. The training
process terminates either when this average loss fails to improve for
10 consecutive 100-epoch intervals, or when a maximum limit of 20,000
epochs is reached. 
Keeping the remaining settings at default (see
\S\ref{ssec:default-parameters}), we trained a set of 20 models using
varying numbers of hidden layers $l \in \{1, 2, 4, 8, 16\}$ and meshsizes
$h\in \{0,5, 0.25, 0.125, 0.0625\}$, all utilizing the $\Ldpg$ loss.

\begin{figure}
  \centering
  \begin{minipage}{0.48\textwidth}
    \centering
    \begin{tikzpicture}
      \begin{semilogyaxis}[
        width=\linewidth,   
        height=8cm,         
        xlabel={Total (NN and FE) DoFs},
        ylabel ={$  \| u_{0,\theta} - u_*\|_{L_2}^2
          +   \| q_{0,\theta} - q_*\|_{L_2}^2$},
        grid=both,
        grid style={dashed,gray!40},
        major grid style={solid,gray!60},
        minor tick num=1,
        legend style={at={(0.03,0.03)},anchor=south west, draw=none, fill=none},
        tick align=outside,
        tick pos=left,
        cycle list name=color list,
        every axis plot/.append style={thick},
        ymin=1e-6,
        ymax=1e-1,
        x tick label style={font=\footnotesize},
        y tick label style={font=\footnotesize},
        ylabel style={font=\footnotesize},
        xlabel style={font=\footnotesize},
        legend style={font=\footnotesize},
        cycle list={
          {tabred,mark=*},
          {taborange,mark=x},
          {tabgreen,mark=+},
          {tabblue,mark=o},
          {tabpurple,mark=star}
        },
        ]

        \addplot+ table [x=DoF, y=Error, col sep=comma] {error_vs_DOFS_nhid-1.csv};
        \addlegendentry{$l=1$}

        \addplot+ table [x=DoF, y=Error, col sep=comma] {error_vs_DOFS_nhid-2.csv};
        \addlegendentry{$l=2$}

        \addplot+ table [x=DoF, y=Error, col sep=comma] {error_vs_DOFS_nhid-4.csv};
        \addlegendentry{$l=4$}

        \addplot+ table [x=DoF, y=Error, col sep=comma] {error_vs_DOFS_nhid-8.csv};
        \addlegendentry{$l=8$}

        \addplot+ table [x=DoF, y=Error, col sep=comma] {error_vs_DOFS_nhid-16.csv};
        \addlegendentry{$l=16$}

      \end{semilogyaxis}
    \end{tikzpicture}
  \end{minipage}\hfill
  \begin{minipage}{0.48\textwidth}
    \centering
    \begin{tikzpicture}
      \begin{semilogyaxis}[
        width=\linewidth,   
        height=8cm,         
        xlabel={Total (NN and FE) DoFs},
        ylabel={$\Ldpg(\uu^\dpg_\theta)$}, 
        grid=both,
        grid style={dashed,gray!40},
        major grid style={solid,gray!60},
        minor tick num=1,
        legend style={at={(0.03,0.03)},anchor=south west, draw=none, fill=none},
        tick align=outside,
        tick pos=left,
        cycle list name=color list,
        every axis plot/.append style={thick},
        ymin=0.002,
        ymax=0.2,
        x tick label style={font=\footnotesize},
        y tick label style={font=\footnotesize},
        ylabel style={font=\footnotesize},
        xlabel style={font=\footnotesize},
        legend style={font=\footnotesize},
        cycle list={
          {tabred,mark=*},
          {taborange,mark=x},
          {tabgreen,mark=+},
          {tabblue,mark=o},
          {tabpurple,mark=star}
        },
        ]

        \addplot+ table [x=DoF, y=Error, col sep=comma] {error_vs_DOFS_nhid-1_loss.csv};
        \addlegendentry{$l=1$}

        \addplot+ table [x=DoF, y=Error, col sep=comma] {error_vs_DOFS_nhid-2_loss.csv};
        \addlegendentry{$l=2$}

        \addplot+ table [x=DoF, y=Error, col sep=comma] {error_vs_DOFS_nhid-4_loss.csv};
        \addlegendentry{$l=4$}

        \addplot+ table [x=DoF, y=Error, col sep=comma] {error_vs_DOFS_nhid-8_loss.csv};
        \addlegendentry{$l=8$}

        \addplot+ table [x=DoF, y=Error, col sep=comma] {error_vs_DOFS_nhid-16_loss.csv};
        \addlegendentry{$l=16$}

      \end{semilogyaxis}
    \end{tikzpicture}
  \end{minipage}
  
  \caption{Results of ablation study: as NN depth $l \in \{1, 2, 4, 8, 16\}$ is
    expanded mesh size $h\in \{0.5, 0.25, 0.125, 0.0625\}$ is
    decreased, the mean of squares of the $L_2$-error in DPG
    predictions $u_{0,\theta}$ and $ q_{0, \theta}$ over a random test set
    decrease (left), and so does the $\Ldpg$ loss of the DPG-NN
    predictions (right).}
  \label{fig:ablation_fixed_nhidden}
\end{figure}

The results are presented in Figure~\ref{fig:ablation_fixed_nhidden}.
It features two sets of curves. Each curve progresses through a set of
models with (the above-mentioned) decreasing values of $h$ while keeping
the same number of hidden layers $l$, and plots the post-training
accuracy of each model as a function of the total degrees of freedom
generated by the $h$ and $\theta$~settings.  The accuracy is evaluated
over a test set of $100$ samples drawn from the same distribution as
the training set, in two different metrics.  In the right plot, the
accuracy metric is the mean over the test samples of the variationally
correct DPG loss function.  In the left plot, the accuracy metric is
the mean squared $L_2$ error in the first two components of the DPG-NN
prediction
$\uu^\dpg_\theta = (q_{0, \theta}, u_{0, \theta}, \hat u_\theta,
\hat q_\theta)$, namely. Letting
$\uu_* = (q_*, u_*, \hat u_* , \hat q_*)$ denote a high-fidelity DPG
finite element solution calculated using the finest considered mesh
size ($h=0.0625$), the left curves show how the error metric 
$
  \| u_{0,\theta} - u_*\|_{L_2}^2 +
  \| q_{0,\theta} - q_*\|_{L_2}^2
$
varies with $h$ and $\theta$. We observe that in both metrics, the
accuracy increases as the number of degrees of freedom increase. The loss values track the error but are higher than the error (as expected from  Proposition~\ref{prop:VC-DPG}  since loss also includes
errors in interface variables).    For smaller
meshsizes, Figure~\ref{fig:ablation_fixed_nhidden} shows that
deeper NNs  perform better. For example, the flattening of the $l=1$ curves
seems to indicate that the expressiveness of a shallow network with just one hidden layer is insufficient for significant accuracy gains as $h$ is decreased.

\section{Improved parameter robustness using DPG ideas}
\label{sec:param-robust}

The proofs of variational correctness of both the FOSLS loss function
and the DPG loss function depended on an inequality independent of
discretizations, namely~\eqref{eq:A-bdd-below}. Even if this might
suggest that it is impossible to remove the $\alpha$~dependence, we
are able to ameliorate the $\alpha$ dependence in the DPG case, as we
shall see in this section.  As motivation, we refer again to the
antecedent of the ultraweak DPG method given in
Example~\ref{eg:uw}. There, the optimal test norm
of~\eqref{eq:opti-test-norm}, $\| A_\alpha^* \y\|_{L_2}$ for any $\y$
in the (unbroken) graph space yielded an isometry of the operator
induced by the corresponding weak formulation, which yielded an
attractive characterization of error in the $\alpha$-independent
$L_2$-norm in~\eqref{eq:uw-error-residual} with no $\alpha$-dependent
constants. When moving to the DPG setting, in the broken graph norm
of~\eqref{eq:Y-dpg-norm}, we are forced to add an $L_2$-term to make a
norm, namely in
$(\| A_\alpha^* \y\|_{L_2}^2 + \| \y\|_{L_2}^2)^{1/2}$, for $\y$ in
the broken graph space $\Ydpg$. Consequently, we no longer
have~\eqref{eq:uw-error-residual} to characterize the error in an
$\alpha$-independent norm. In order to retain as much $\pp$-robustness as possible, in this section, we weaken the previous DPG norm toward the optimal test norm of~\eqref{eq:opti-test-norm}, and then see to what extent we can modify resulting fiber loss functionals so as to obtain sharp bounds for the $L_2$-errors of the bulk terms (interior variables).

\subsection{Loss functionals built with a weaker test space norm}

Instead of \eqref{eq:Y-dpg-norm}, we now endow the test space
$\Ydpg = \Hdivoh \times H^1(\oh)$ by the following norm with an
additional scaling factor $s>0$,
\[
  \| (\tau, \nu) \|_{\cl Y, s}^2  =
  \sum_{K \in \oh}
  \left( \| A_\alpha^*(\tau,\nu) \|_{L_2(K)}^2 + s^{-2} \| (\tau, \nu) \|_{L_2(K)}^2
  \right).
\]
The corresponding inner product is denoted by
$(\cdot, \cdot)_{\cl Y, s}$.  The DPG variant using this inner product
is defined exactly as in \eqref{eq:DPG-method} but replacing
$(\cdot, \cdot)_{\cl Y}$ with $(\cdot, \cdot)_{\cl Y, s}$.  Such a method was
first studied in~\cite{GopalMugaOliva14}.
Analogous
to \eqref{eq:DPG-loss}, we define $\vepzs \in \Ydpg$ by 
\begin{subequations}
  \label{eq:DPG-loss-s}
  \begin{equation}
    \label{eq:DPG-loss-s-prep}
    (\vepzs, \y)_{\cl Y, s} = (F, \y)_{L_2} - b_\alpha^\dpg (\w, \y), \qquad
    \text{ for all } \y \in \Ydpg,
  \end{equation}
and a new $s$-dependent fiber loss  function by 
  \begin{equation}
    \label{eq:DPG-loss-s-fun}
    \hLsdpg(\w) := \| {\vepzs}\|_{\cl Y, s}^2, \qquad \w \in \Xhdpg, \quad \pp\in \pdom.
  \end{equation}
\end{subequations}
As $s$ increases, the test space norm becomes weaker (approaching~\eqref{eq:opti-test-norm}) and this new loss
function gains improved parameter robustness as shown in
Theorem~\ref{thm:dpg-loss-s} below. 

The drawback of the loss function $\hLsdpg$ is that it requires us to
solve the infinite-dimensional problem~\eqref{eq:DPG-loss-s-prep} for
$\varepsilon^\w_s$. As we shall see below,  we are able to get good results when using a
computable version obtained by replacing $\Ydpg$ with $\Yhdpg$ in \eqref{eq:DPG-loss-s-prep}, i.e,
in place of $\vepzs$, we use $\vepzhs \in \Yhdpg$ satisfying
\begin{subequations}
  \label{eq:computable-Ls}
  \begin{equation}
    \label{eq:computable-Ls-eps}
    (\vepzhs, \y)_{\cl Y, s} = (F, \y)_{L_2} - b_\alpha^\dpg (\w, \y), \qquad
    \text{ for all } \y \in \Yhdpg,
  \end{equation}
and
\begin{equation}
  \label{eq:computable-Ls-defn}
  \Lsdpg(\w) := \| \vepzhs\|_{\cl Y, s}^2, \qquad \w \in \Xhdpg.
\end{equation}
\end{subequations}
We will also consider the dual-parameter functional
\begin{equation}  
  \label{eq:computable-Ls1s2}
  \Lss{s_1}{s_2} (\w) := \frac{ s_2^2 \Lsa{s_1}(\w) - s_1^2 \Lsa{s_2}(\w) }{s_2^2  - s_1^2}
\end{equation}
for any two positive parameters {$s_1<s_2$}.

\subsection{Theoretical motivation}

We present a few theoretical observations that motivate  the new loss
function families defined above.

It will prove convenient to express
$\Xdpg = L_2(\om)^d \times L_2(\om) \times \Hoh\times H^{-1/2}(\d\oh)$
as a Cartesian product of the space of ``interior variables'' (or bulk variables) in
$\Xdpgo = L_2(\om)^d \times L_2(\om)$ and the space of ``interface variables'' 
$\Xdpgh = \Hoh\times H^{-1/2}(\d\oh)$,
i.e., any $\w^\dpg \in \Xdpg$ will be written as
\begin{equation}
  \label{eq:dpg-space-split}
  \w^\dpg \in \Xdpg = \Xdpgo \times \Xdpgh, \qquad  \w^\dpg = (\w_0, \hat \w),
  \text{ with } \w_0 \in \Xdpgo, \; \hat \w \in \Xdpgh.
\end{equation}
Any $\hat \w$ in $\Xdpgh$ is obtained as the trace of a function in
$\Ho\times \Hdiv$, so the following coset definition makes sense:
\[
  Z({\hat \w}) = \{ (r, w) \in \Hdiv \times \Ho : (\tr w, \trn r) = 
  \hat \w\}.
\]
By~\eqref{eq:A-bdd-below}, $\| A_\alpha \cdot \|_{L_2}$ is
a norm on $\Hdiv \times \Ho$,  allowing us to define a minimal
extension operator
\[
  E \hat \w = \argmin_{ (r, w) \in Z({\hat \w})} \| A_\alpha (r, w) \|_{L_2}.
\]
Writing any $\w^\dpg \in \Xdpg$ as in~\eqref{eq:dpg-space-split}, we
use the above extension to define a trial space norm by
\begin{equation}
  \label{eq:s-alpha-norm}
  \| \w^\dpg\|_{s,\pp}^2 \equiv \| (\w_0, \hat \w) \|_{s,\pp}^2 := \| \w_0 \|_{L_2}^2 + s^2 \| A_\alpha E  \hat \w\|_{L_2}^2.  
\end{equation}
Noting that  $Z(0)$ is a closed subspace of $\Xdpgo =L_2(\om)^d \times L_2(\om)$, 
we  also define the Peetre $K$-functional
\begin{equation}
  \label{eq:K-fun-defn}
  K(s, \w_0) = \inf_{\z \in Z(0)}
  \Big(\| \w_0 - \z\|_{L_2}^2  + s^2 \| A_\alpha \z \|_{L_2}^2 \Big)
\end{equation}
defined for any $\w_0 \in \Xdpgo$. In the next theorem, we also use the  following quantities for positive numbers $s, s_1,$ and~$s_2$:
\begin{gather}
  \label{eq:ksa}
  \ksa = \frac 1 2 \left( \frac{c^2_\alpha}{s^2}
    + \sqrt{ \frac{c_\alpha^4}{s^4}
      + 4 \frac{c^2_\alpha}{s^2}
    } \right),
  \\
  \label{eq:Ls1s2}
  \hLss{s_1}{s_2} (\w^\dpg) = \frac{ s_2^2 \hLs{s_1}(\w^{\dpg}) - s_1^2 \hLs{s_2}(\w^{\dpg}) }{s_2^2  - s_1^2},
  \\
  \label{eq:Ls1s2L}
  \hLssL{s_1}{s_2}(\w^{\dpg}) =
  \frac{ (1 + \ka{s_1})^{-1} s_2^2 \hLs{s_1}(\w^{\dpg}) - (1 +  \ka{s_2}) s_1^2 \hLs{s_2}(\w^{\dpg}) }
  {s_2^2  - s_1^2},
  \\
  \label{eq:Ls1s2R}
  \hLssR{s_1}{s_2}(\w^{\dpg}) =
  \frac{ (1 + \ka{s_1}) s_2^2 \hLs{s_1}(\w^{\dpg}) - (1 +  \ka{s_2})^{-1} s_1^2 \hLs{s_2}(\w^{\dpg}) }
  {s_2^2  - s_1^2},
\end{gather}
where  $c_\alpha$ is as in~\eqref{eq:A-bdd-below}.

\begin{theorem}
  \label{thm:dpg-loss-s}
  Let $s>0$, $ 0 < s_1 < s_2$, 
  $\alpha \in \cl Q$,  $\w_0 \in \Xdpgo,$  $\hat \w \in \Xdpgh$, 
  and  $\w^\dpg = (\w_0, \hat \w)$. 
  Let 
  $\x^\dpg = (\x_0, \hat \x) = (q, u, \hat u, \hat q_n)$ denote the exact solution of \eqref{eq:dpg-undiscretized-formulation}
 for the arbitrarily chosen $\alpha \in \cl Q$.
  Then the following results hold.
  \begin{enumerate}
  \item \label{item:dpg-loss-s-1} The fiber loss  functionals $\hLsdpg$ are
    variationally correct and satisfy the two-sided estimates
    \begin{equation}
      \label{eq:dpg-loss-s}
      (1+ \ksa)^{-1} \;\| \x^\dpg - \w^\dpg \|_{s, \pp}^2
      \;\le\;
      \hLsdpg(\w^\dpg)
      \;\le\;
      (1 + \ksa) \;\| \x^\dpg - \w^\dpg \|_{s, \pp}^2
    \end{equation}
    with the explicit constant $\ksa$ defined in~\eqref{eq:ksa}.

  \item \label{item:dpg-loss-s-2} The loss functional $\hLsdpg$ is related to the error through the $K$-functional in~\eqref{eq:K-fun-defn}:
    \begin{equation}
      \label{eq:K-fun}
      \hLsdpg(\w^\dpg) = s^2 \| A_\alpha E (\hat \x - \hat \w) \|_{L_2}^2
      +
      K(s, (\x_0 - \w_0) - E (\hat \x - \hat \w)).
    \end{equation}

  \item \label{item:dpg-loss-s-3} The functional $\hLss{s_1}{s_2}$
    in~\eqref{eq:Ls1s2} is non-negative.

  \item \label{item:dpg-loss-s-4} The interior variable error
    $\x_0 - \w_0$ satisfies
    \begin{equation}
      \label{eq:interior-var-2-bounds}
      \hLssL{s_1}{s_2}(\w^\dpg)  \le   \| \x_0 - \w_0 \|_{L_2}^2  \le \hLssR{s_1}{s_2}(\w^\dpg).
    \end{equation}

  \item \label{item:dpg-loss-s-5} For all $\alpha$
    admitting~\eqref{eq:A-bdd-below} with a fixed $c_\alpha$, as
    $s_1 \to \infty$, both the upper and lower bounds above converge
    to the non-negative functional $\hLss{s_1}{s_2}$,
    \[
      \hLssL{s_1}{s_2}(\w^{\dpg}) \to \hLss{s_1}{s_2}(\w^{\dpg}), \qquad
      \hLssR{s_1}{s_2}(\w^{\dpg}) \to \hLss{s_1}{s_2}(\w^{\dpg}).
    \]

  \end{enumerate}
\end{theorem}

Before turning to the proof of these somewhat technical facts, a 
few remarks on how to interpret the results in
Theorem~\ref{thm:dpg-loss-s} and explain their relevance are in order.  The advantage of the new
loss function $\hLsdpg$ is evident from~\eqref{eq:dpg-loss-s} of
Theorem~\ref{thm:dpg-loss-s}(\ref{item:dpg-loss-s-1}). Namely, even
when $c_\alpha$ is very large, it is possible to choose a scaling
factor $s$ such that $c_\alpha/s \lesssim 1$, thus making the constants
in \eqref{eq:dpg-loss-s} approximately independent of $\alpha$.  This
motivates the use of the loss function $ \Lsdpg$
defined in~\eqref{eq:computable-Ls} (since
$\hLsdpg$ is not computable).
Yet, a drawback
of~\eqref{eq:dpg-loss-s} is that the error is measured in the
$\alpha$-dependent norm $\| \cdot \|_{s,\pp}$, incorporating the (stronger) graph-norm through the element boundary terms. The estimates
of~\eqref{eq:interior-var-2-bounds} in
Theorem~\ref{thm:dpg-loss-s}(\ref{item:dpg-loss-s-4}) address this
drawback by providing upper and lower bounds of an
{\em $\alpha$-independent norm} of the error in the interior variable
approximation. 
The constants in those two estimates are explicit, 
but the upper and lower bounds use two distinct functionals
$\hLssL{s_1}{s_2}$ and $\hLssR{s_1}{s_2}$. These two functional values
are then proved to converge to the value of the functional
$\hLss{s_1}{s_2}$ in
Theorem~\ref{thm:dpg-loss-s}(\ref{item:dpg-loss-s-5}).  Together with
the non-negativity of $\hLss{s_1}{s_2}$ proved in
Theorem~\ref{thm:dpg-loss-s}(\ref{item:dpg-loss-s-3}), this provides a
strong rationale for using $\hLss{s_1}{s_2}$ to obtain an error
certificate in an $\alpha$-independent norm for part of the total error. This part of the error consists of the error in $q$ and error in $u$ and is therefore comparable to the $H^1$ error, the canonical norm in which error is measured
for the usual primal formulation (see Example~\ref{eg:standard-weak-form}).
Again, since
$\hLss{s_1}{s_2}$ is not computable, we have resorted to the use
of $\Lss{s_1}{s_2}$ in~\eqref{eq:computable-Ls1s2} in some experiments to follow.
Multiple tests, reported in the remainder of this section, show the practical
effectiveness of these strategies.
Note that if we were to restrict ourselves to only collections of
$\alpha$ with a bounded $c_\alpha$ (of \eqref{eq:A-bdd-below}), then
the argument of the proof of Proposition~\ref{prop:VC-DPG} using the
same Fortin operator would prove the variational correctness of
$\Lsdpg$ with constants that depend on $c_\alpha$.  It is currently an open
question to obtain an analogue of~\eqref{eq:dpg-loss-s} with explicit
constants for $\Lsdpg$ instead of $\hLsdpg$.
We now prove each item of Theorem~\ref{thm:dpg-loss-s}.

\begin{proof}[Proof of Theorem~\ref{thm:dpg-loss-s}:]
  For any $\w^{\dpg} \in \Xdpg$, let 
  \[
    \| \w^{\dpg} \|_{b_\alpha^\dpg}
    = \sup_{0 \ne \y \in \Ydpg} \frac{|b_\alpha^\dpg(\w^{\dpg}, \y)|}{\| \y\|_{\cl Y, s}}.
  \]
  Since $b_\alpha^\dpg$ is bounded and satisfies the inf-sup condition (see
  \eqref{eq:cty-infsup}), this is a norm on $\Xdpg$.

  {\em Proof of~items~(\ref{item:dpg-loss-s-1})--(\ref{item:dpg-loss-s-2}):}
  By~\cite[Theorem~7.9, Propositions~7.10--7.11]{DPGacta},
  whenever~\eqref{eq:A-bdd-below} holds, we have
  \begin{equation}
    \label{eq:s-norm-K}
    \| \w^{\dpg} \|_{b_\alpha^\dpg} = s^2 \| A_\alpha E \hat \w \|_{L_2}^2 + K(s, \w_0 - E \hat \w),
  \end{equation}
  and 
  \begin{equation}
    \label{eq:equiv-optim}
    (1 + \ksa)^{-1} \| \w^{\dpg} \|_{s,\pp}^2 \le \| \w^{\dpg}\|_{b_\alpha^\dpg}^2 \le (1 + \ksa) \| \w^{\dpg} \|_{s,\pp}^2
  \end{equation}
  for all $\w^{\dpg} \in \Xdpg$. Now observe that \eqref{eq:DPG-loss-s} implies 
  \begin{align*}
    \| \x^{\dpg} - \w^{\dpg}\|_{b_\alpha^\dpg}
    & =  \sup_{0 \ne \y \in \Ydpg} \frac{|b_\alpha^\dpg(\x^{\dpg} - \w^{\dpg}, \y)|}{\| \y\|_{\cl Y, s}}
      =  \sup_{0 \ne \y \in \Ydpg} \frac{|\ell^\dpg(\y)  - b_\alpha^\dpg (\w^{\dpg}, \y)|}{\| \y\|_{\cl Y, s}}
    \\
    & = \sup_{0 \ne \y \in \Ydpg} \frac{|(\vepzs, \y)_{\cl Y, s}|}{\| \y\|_{\cl Y, s}}
      = \| \vepzs\|_{\cl Y, s} = \hLsdpg(\w^{\dpg}).
  \end{align*}
  Hence~\eqref{eq:dpg-loss-s} and~\eqref{eq:K-fun} follow
  from~\eqref{eq:equiv-optim} and~\eqref{eq:s-norm-K}, respectively,
  with $\x^{\dpg} -\w^{\dpg}$ in  place of $\w^{\dpg}$.

  {\em Proof of~item~(\ref{item:dpg-loss-s-3}):} Since $0< s_1 < s_2,$
  \[
    s_1^{-2}\| \w - \z\|_{L_2}^2  +  \| A_\alpha \z \|_{L_2}^2
    \ge
    s_2^{-2} \| \w - \z\|_{L_2}^2 + \| A_\alpha \z \|^2_{L_2} 
  \]
  for any $\w_0 \in \Xdpgo$ and $\z \in Z(0)$.  
  By~\eqref{eq:K-fun-defn}, this implies
  \begin{equation}
    \label{eq:K-mon}
    \frac{ K(s_1, \w_0)}{s_1^2}  \ge \frac{ K(s_2, \w_0) }{s_2^2}
  \end{equation}
  for any $\w_0 \in \Xdpgo$.
  Let $\w_0' = \x_0 - \w_0 - E(\hat \x - \hat \w)$ in $\Xdpgo$.   By \eqref{eq:K-fun},
  \[
    \frac{\hLs{s_1}(\w^\dpg)}{s_1^2} =
    \| A_\alpha E (\hat \x - \hat \w) \|_{L_2}^2
    +
    \frac{ K(s_1, \w_0')}{s_1^2}.
  \]
  The same equality holds with $s_2$ also. Subtracting,
  \[
    \frac{\hLs{s_1}(\w^{\dpg})}{s_1^2}  - \frac{\hLs{s_2}(\w^{\dpg})}{s_2^2}
    = \frac{ K(s_1, \w_0')}{s_1^2} - \frac{ K(s_2, \w_0')}{s_2^2},
  \]
  which is nonnegative by~\eqref{eq:K-mon}. Hence
  $\hLss{s_1}{s_2} (\w^{\dpg}) \ge 0$.

  {\em Proof of item~(\ref{item:dpg-loss-s-4}):} Let
  $e_0 = \| \x_0 - \w_0\|_{L_2}$, 
  $\hat e = \| A_\alpha ( \hat \x - \hat \w) \|_{L_2}$, and
  $l_i = \hLs{s_i}(\w^\dpg)$.
  Then
  by~\eqref{eq:dpg-loss-s},
  \begin{align*}
    \frac{e_0^2 }{s_1^2} + \hat e^2
    & \,\le \,(1 + \ka{s_1}) \;\frac{ l_1}{s_1^2},
    \\
    \frac{e_0^2 }{s_2^2} + \hat e^2
    & \,\ge \,(1 + \ka{s_2} )^{-1}
      \frac{ l_2}{ s_2^2}.
  \end{align*}
  Subtracting, 
  \begin{align*}
    e_0^2 \left( \frac{1}{s_1^2} -    \frac{1 }{s_2^2} \right)
    \le
    \frac{     (1 + \ka{s_1}) \,l_1}{s_1^2}
    -
    \frac{ (1 + \ka{s_2} )^{-1} l_2 }{ s_2^2}.
  \end{align*}
  Rearranging, we obtain
  \[
    e_0^2 \le  \hLssR{s_1}{s_2}(\w^\dpg),
  \]
  which is the upper inequality of
  \eqref{eq:interior-var-2-bounds}. The lower bound
  of~\eqref{eq:interior-var-2-bounds} is proved similarly.

  {\em Proof of~item~(\ref{item:dpg-loss-s-5}):} Holding $c_\alpha$
  fixed, we see from the definition of $\ksa$ in~\eqref{eq:ksa} that
  \[
    \lim_{s \to \infty} \ksa = 0.
  \]
  Hence the expressions for $\hLssL{s_1}{s_2}(\w^\dpg)$ and
  $\hLssR{s_1}{s_2}(\w^\dpg)$ in \eqref{eq:Ls1s2L}--\eqref{eq:Ls1s2R}
  converge to the expression for $\hLssL{s_1}{s_2}(\w^\dpg)$
  in~\eqref{eq:Ls1s2}.
\end{proof}

\subsection{Robustness of the $s$-based DPG loss function}
\label{ssec:robustness-s-based}

We train an NN (with the same architecture and other parameters
mentioned previously) using the new DPG fiber loss function $\Lsdpg$
(in the correspondingly modified mean squared
loss~\eqref{mean-squared}).  The NN prediction for any given parameter $\alpha$
now consists of an full solution prediction,
interior solution prediction and interface solution prediction, which
are respectively denoted by $\x_{\theta,s}^\dpg(\pp)$, $\x_{0,\theta,s}^\dpg(\pp)$ and
$ \hat \x_{\theta,s}^\dpg(\pp)$. The corresponding exact finite element
solutions are denoted by $\x_{h, s}^\dpg(\alpha)$, $\x_{0, h, s}^\dpg(\alpha)$ 
and $ \hat \x_{h, s}^\dpg(\alpha)$,
respectively.  Notice that when $s=1$, the loss function coincides
with the previous DPG loss $\Ldpg$ from \eqref{eq:DPG-loss}, a case we
have already discussed at length.  The $s=1$ case,  the prior FOSLS
solution $\xfosls(\alpha)$ from \eqref{eq:FOSLS-method}, and its NN prediction
$\x_\theta^\fosls(\alpha)$ (previously seen in~\eqref{eq:dpg-fosls-theta}),
will offer comparators as we now study variations with~$s$.

We examine the following error measures of interest:
\begin{equation}
  \label{eq:error-quantities}
\begin{aligned}
  &\hat e^\fosls(\pp) := \|A_{\alpha}\big(\x^\fosls_\theta(\pp) - \x^\fosls_h(\pp)\big)\|_{L_2}^2,
  && e_{0}^\fosls(\pp) := \|\x^\fosls_\theta(\pp) - \x^\fosls_h(\pp)\|_{L_2}^2,
  \\
  &\hat e_{s}^\dpg(\pp) : = \|A_{\alpha}\big(\hat \x^\dpg_{\theta,s}(\pp) - \hat \x^\dpg_{h,s}(\pp)\big)\|_{L_2}^2,
  && e_{0,s}^\dpg(\pp) := \|\x^\dpg_{0,\theta, s}(\pp) - \x^\dpg_{0,h, s}(\pp)\|_{L_2}^2.
\end{aligned}  
\end{equation}
We  use these quantities (not in any learning algorithm but) to study the effectiveness and $\alpha$-robustness of
the computable loss functionals to estimate the error in NN
predictions. We drop the $\pp$-argument when it is understood from context.
The exact solution $\x$ is unavailable and is absent from the computable errors
in \eqref{eq:error-quantities}. Nonetheless, 
we see, by means of the triangle inequality, that the FOSLS loss
$\Lfosls(\w) = \| A_\alpha \w - F\|_{L_2}^2 = \| A_\alpha (\w - \x)\|_{L_2}^2 $
bounds the first quantity $\hat e^\fosls$
in~\eqref{eq:error-quantities} as follows:
\begin{equation}
  \label{eq:trg-flsls}
    \hat e^\fosls = \|A_{\alpha}(\x^\fosls_\theta - \x^\fosls_h)\|_{L_2}^2 \leq 2\big[\Lfosls(\x^\fosls_\theta) + \Lfosls(\x^\fosls_h)\big].
\end{equation}
Since both sides of this inequality are computable once we have computed FOSLS solutions on chosen test sets in $\pdom$, we can use it to check the reliability
of the loss function.
For an  analogous measure  for the DPG $s$-dependent loss,
recall the $\|\cdot \|_{s,\pp}$-norm introduced in \eqref{eq:s-alpha-norm}
and observe that
\[
  \|\x^\dpg_{\theta,s} - \x^\dpg_{h,s}\|_{s,\pp}^2
  = e_{0,s}^\dpg + s^2 \hat e_{s}^\dpg.
\]
Then, \eqref{eq:dpg-loss-s} of Theorem \ref{thm:dpg-loss-s} and a new
application of the triangle inequality gives 
\begin{equation}
  \label{eq:trg-dpg}
  e_{0,s}^\dpg + s^2 \hat e_{s}^\dpg
  \leq 2 (1 + k_{s, \alpha}) \left [ \hLsdpg(\x^\dpg_{\theta,s})
    + \hLsdpg(\x^\dpg_{h,s}) \right].  
\end{equation}
Thus, just as in the FOSLS case~\eqref{eq:trg-flsls}, once we have
computed DPG solutions for test sets in $\pdom$, both
$e_{0,s}^\dpg + s^2 \hat e_{s}^\dpg$ and the loss functional values
in~\eqref{eq:trg-dpg} can be computed without knowing the exact
solution.

Recall that training on mean-squared losses like \eqref{mean-squared} approximates the solution $u(x,\pp)$ as a function of $x$ and $\pp$ {\em in expectation}, hence in a mean-squared sense. This does not imply beforehand any worst-case accuracy control over $\pdom$. On the other hand, it is well known that under the given assumptions the parameter-to-solution map is even {\em holomorphic}, despite the fact that for each $\pp\in\pdom$ the solution as a function of $x$ may have very low regularity. This suggests that worst-case errors and  the  mean-squared errors with respect to $\pp$ may not be that different. Subsequent experiments are to shed some light on this issue as well.

With these considerations in mind,  we propose the following experiment.  Let
$\cl Q_{M}^{\bar \alpha, \sigma}$ be a collection of $\alpha$
samples of size $M$ drawn from the same distribution described in
Section~\ref{sec:initial-results} with mean $\bar \alpha$ and standard
deviation $\sigma$.  Let
$\cl Q_{m}^{\bar \alpha, \sigma}$ denote its subset of the
first $m \le M$ samples. For any function $e(\alpha)$ over the samples, we denote
its {\em cumulative maximum} function by 
\[
  (\cmax e)(m) =
  \max_{\alpha \in \cl Q_{m}^{\bar \alpha, \sigma}} e(\alpha),
  \qquad m = 1, 2, \dots, M, 
\]
i.e., $  (\cmax e)(m) =  \|e\|_{\ell_\infty(\cl Q_{m}^{\bar \alpha, \sigma})}$.
Now, turning to \eqref{eq:trg-flsls}, we are motivated to define 
\begin{equation*}
  \hat \rho^{\fosls} := \cmax
  \left(\frac{\hat e^\fosls}{\Lfosls(\x^\fosls_\theta) + \Lfosls(\x^\fosls_h)}\right)
\end{equation*}
a quantity which cannot be greater than two in view of \eqref{eq:trg-flsls}.
Similarly,  \eqref{eq:trg-dpg}  motivates the definition of
\begin{equation}
  \label{eq:rho-dpg-s}
  \rho_{s}^{\dpg} := \cmax\left(
    \frac{e_{0,s}^\dpg + s^2 \hat e_{s}^\dpg}{\Lsdpg(\hat \x^\dpg_{\theta,s}) + \Lsdpg(\hat \x^\dpg_{h,s})}\right).
\end{equation}
If we had used $\hLsdpg$ instead of $\Lsdpg$, then in view of
\eqref{eq:trg-dpg}, the right hand side above would have been bounded
by $2\cmax (1 + k_{s, \alpha}),$ which for large $s$ approaches the
value of 2 in view of~\eqref{eq:ksa}.  We are interested in observing
if this is the case in practice for the above computable
version. For completeness and to facilitate
comparisons, we also introduce the following analogous quantities for
FOSLS approximations,
\begin{equation*}
  \rho_{0}^{\fosls} := \cmax
  \left(
    \frac{e_{0}^\fosls}{\Lfosls(\x^\fosls_\theta) + \Lfosls(\x^\fosls_h)}
  \right), \qquad
  \rho^{\fosls} := \cmax
  \left( 
    \frac{e_{0}^\fosls + \hat e^\fosls}{\Lfosls(\x^\fosls_\theta)
      + \Lfosls(\x^\fosls_h)}\right).
\end{equation*}
We report these quantities observed from our experiment as we draw the
random samples of~$\alpha$.

\begin{figure} 
\centering
\begin{tikzpicture}
\begin{loglogaxis}[
    width=15cm,
    height=8cm,
    xlabel={sample number $m$ from $1, \dots, M$},
    ylabel={error to loss ratio}, 
    legend pos=north west,
    xmajorgrids = false,
    xminorgrids = false,
    ymajorgrids = true,
    yminorgrids = false,
    legend cell align=left,
]

\addplot+[mark=none, color=orange, thick, mark=none] table [x index=0, y index=1, col sep=comma] {fosls_l2_error_data.csv};
\addlegendentry{$\rho_{0}^{\fosls}(m)$}

\addplot+[const plot, thick, color=crimson, mark=none] table [x index=0, y index=1, col sep=comma] {fosls_Anorm_error_data.csv};
\addlegendentry{$\hat \rho^\fosls(m)$}
\end{loglogaxis}
\end{tikzpicture}
\caption{Cumulative maxima of FOSLS  error to loss ratios, 
  from a random collection of $M=10^4$ samples of $\alpha$,
  increase in the $\alpha$-independent metric.
}
\label{fig:FOSLS_L2_error_ratios}
\end{figure}

Each  experiment drawing random
samples of $\alpha$ forms the set  $\cl Q_{M}^{\bar \alpha, \sigma}$
with $M \geq 10000$, $\bar \alpha = (10^{-1}, 1, 1, 10^{-1})$ and
$\sigma = 0.5$.  To establish baseline cases, we begin with the
FOSLS measures $\rho^{\fosls}$ and $\hat \rho^{\fosls}$. The
latter is bounded by two as seen above, whereas the theory  does
not provide any $\alpha$-independent bound for the former. The
experimental results in Figure~\ref{fig:FOSLS_L2_error_ratios} are in
agreement: they show that $\hat \rho^\fosls$ remains bounded across
samples while $\rho^{\fosls}$ grows as the random choices pick up
samples of $\alpha$ from the distribution tail.  In other
words, the ratios of squares of error norms to loss values do not
appear to be bounded robustly in $\alpha$ for FOSLS when the norm
contains the $\alpha$-independent $L_2$-norm of the error,
i.e., $L_2$ errors appear to suffer most from a worsening conditioning
  of the FOSLS operator.
 

\begin{figure}
\centering
\begin{tikzpicture}
\begin{loglogaxis}[
    width=15cm,
    height=8cm,
    xlabel={sample number $m$ from $1, \dots, M$},
    ylabel={error to loss ratio}, 
    legend pos=north west,
    xmajorgrids = false,
    xminorgrids = false,
    ymajorgrids = true,
    yminorgrids = false,
    legend cell align=left,
]

\addplot+[mark=none, color=crimson, thick, mark=none] table [x index=0, y index=1, col sep=comma] {fosls_s_regularization_data.csv};
\addlegendentry{$\rho^\fosls(m)$}

\addplot+[const plot, thick, color=navy, mark=none] table [x index=0, y index=1, col sep=comma] {dpg_s_regularization_s1_data.csv};
\addlegendentry{$\rho_{s=1}^\dpg(m)$}

\addplot+[const plot, thick, color=royalblue, mark=none] table [x index=0, y index=1, col sep=comma] {dpg_s_regularization_s10_data.csv};
\addlegendentry{$\rho_{s=10}^\dpg(m)$}

\addplot+[const plot, thick, color=darkcyan, mark=none] table [x index=0, y index=1, col sep=comma] {dpg_s_regularization_s100_data.csv};
\addlegendentry{$\rho_{s=100}^\dpg(m)$}

\end{loglogaxis}
\end{tikzpicture}
\caption{Cumulative maxima of  DPG  error to loss ratios,
  from a random collection of $M=10^4$ samples of $\alpha$,
  increase less for larger $s$.}
\label{fig:s-regularization_plot}
\end{figure}

In the $s$-dependent DPG case, we observe better robustness for high
values of~$s$. Plots of $\rho^\dpg_{s}$ for $s=1, 10,$ and $100$
are shown in Figure~\ref{fig:s-regularization_plot} (together with the
baseline plot of $\rho^\fosls$ from
Figure~\ref{fig:FOSLS_L2_error_ratios}). As mentioned previously, we are interested
in seeing if these ratios would be bounded, in practice,
by $2\cmax (1 + k_{s, \alpha}),$ a value which
approaches two as $s \to \infty$ for bounded ranges of $\alpha$. We see in
Figure~\ref{fig:s-regularization_plot} that the  graph of
$\rho^\dpg_{s}$ remains flat with values close to one when
$s=100$, while for lower values
of $s$, it climbs as extreme values of $\alpha$ are picked up on
repeated sampling. Of course, we expect there are values of $\alpha$
which will make even $s=100$ insufficient to obtain a ratio close to
one. They may be picked up upon even further sampling. Nonetheless,
Figure~\ref{fig:s-regularization_plot} conveys the practical utility
of larger values of~$s$.

Conditioning and numerical stability will limit us from increasing $s$
unboundedly.  Indeed, the loss computation requires the
calculation of $\vepzhs \in \Yhdpg$, which involves the
inversion of a block diagonal matrix whose condition number
can deteriorate as $s$ increases.  If a range of $\alpha$ is known, then
a practical strategy would be to find an $s$-value that
works for that range and address the numerical stability issues for
that $s$~value.  Such $\alpha$-ranges may be available for specific
applications like porous media flow, but since this is
application-specific, we do not pursue it here. Instead, we proceed
to discuss a general strategy that can help us identify appropriate
$s$-values when an $\alpha$-range of interest is available.

\subsection{Training a DPG NN with $s$ and $\alpha$}
\label{subsec:(s,alpha)-surface}

The ultraweak DPG setting allows for  tuning the test space topology 
using the scaling factor~$s$. This opens the possibility of
letting $s$ be yet another random parameter while training neural
networks. Having such an NN allows one to guess $s$ values that may be
appropriate for an identified range of $\alpha$ values of interest.

To execute this idea in practice, we consider a new NN trained using 
the scaled DPG fiber loss  function
\begin{equation}
\label{eq:L-s-alpha}
\cl L_{s,\alpha} = \frac{1}{s^2} \Lsdpg.
\end{equation}
The scaling by $1/s^2$ in \eqref{eq:L-s-alpha} is motivated
by~\eqref{eq:dpg-loss-s} of
Theorem~\ref{thm:dpg-loss-s}(\ref{item:dpg-loss-s-1}), which implies
that for any $\w^\dpg=(\w_0, \hat \w) \in \Xdpg$,
\[
  (1 + \ksa)^{-1}\,\frac{\hLsdpg(\w^\dpg)}{s^2} 
  \;\le\;
  \frac{\| \x_0  - \w_0\|_{L_2}^2}{s^2}   + \| A_\alpha E (\hat \x -\hat \w)\|_{L_2}^2
  \;\le\;
  (1 + \ksa) \,\frac{\hLsdpg(\w^\dpg)}{s^2}.
\]
Without the $1/s^2$ scaling, the $\alpha$-dependent part of the error
norm, namely $\|A_\alpha E (\hat \x -\hat \w)\|_{L_2}^2 $ would get
amplified by $s^2$, so for large $s$, the loss $\hLsdpg(\w^\dpg)$
becomes extremely large and less sensitive to the $\alpha$-independent
error norm $\| \x_0 - \w_0\|_{L_2}^2$.

Using~\eqref{eq:L-s-alpha}, 
the training process is carried out sampling both $\alpha$ and $s$.
We select the distribution for the random variable $(\alpha, s)$ to be
the product of the $\chi^2_{\bar\alpha,\sigma}$-distribution (described in Section~\ref{sec:initial-results})
and the uniform distribution of $s$-values in the interval $[c, d]$.
Let $\cl Q_M^{\bar\alpha, \sigma, [c, d]}$ denote a collection
of $M$ samples of $(\alpha, s)$ drawn from this distribution.  The
training process is performed over $M=1024$
samples and $5000$ epochs at the learning rate of $0.001$.
The parameters of the random distribution
are set by
\begin{equation}
  \label{eq:s-alpha-expt-prm}
  \bar \alpha = (10^{-1}, 1, 1, 10^{-1}), \quad 
  \sigma = 0.5, \quad 
  c = 1, \quad d = 1000,
\end{equation}
and the NN is trained with the loss \eqref{mean-squared}
where the fiber loss is replaced by $\cl L_{s,\alpha}$
from~\eqref{eq:L-s-alpha}.

\begin{figure}
  \centering  
  \begin{tikzpicture}

    \begin{axis}[
      width=14cm,
      height=7cm,
      xlabel={$s$},
      ymode=log,
      xmode=normal,
      xmajorgrids = false,
      xminorgrids = false,
      ymajorgrids = true,
      yminorgrids = false,
      legend pos=south east,
      legend cell align=left,
      axis y line*=left,
      x tick label style={font=\footnotesize},
      y tick label style={font=\footnotesize},
      ylabel style={font=\footnotesize},
      xlabel style={font=\footnotesize},
      legend style={font=\footnotesize},
      ylabel={Loss \& Error},
      ]

      \addplot+[color=blue, thick, mark=none] 
      table [x index=0, y index=1, col sep=comma] 
      {e0_vs_s_alpha0-0.6_alpha1-1.0.csv};
      \addlegendentry{Error $e_{0,s}^\dpg(\pp)$}

      \addplot+[const plot, ultra thick, densely dotted,
      color=dgreen, mark=none] 
      table [x index=0, y index=1, col sep=comma] 
      {loss_vs_s_alpha0-0.6_alpha1-1.0.csv};
      \addlegendentry{Loss $\Lsdpg$}

    \end{axis}

    \begin{axis}[
      width=14cm,
      height=7cm,
      ymode=log,
      axis line style={purple, very thick},
      ylabel style={purple},
      tick style={purple, thick},
      axis y line*=right,
      axis x line=none,
      legend style={at={(0.98,0.7)},anchor=south east},
      x tick label style={font=\footnotesize},
      y tick label style={font=\footnotesize},
      ylabel style={font=\footnotesize},
      xlabel style={font=\footnotesize},
      legend style={font=\footnotesize},
      ylabel={Accuracy of Gram matrix inversion},
      ]

      \addplot+[color=purple, very thick, dash pattern=on 8pt off 2pt,
      mark=none] 
      table [x index=0, y index=1, col sep=comma] 
      {RRinv_vs_s_alpha0-0.6_alpha1-1.0.csv};
      \addlegendentry{$\|I - R_{s, \pp} R_{s,\pp}^{-1}\|_{F}$}

    \end{axis}

  \end{tikzpicture}
  \caption{Errors, losses,  and conditioning}
  \label{fig:err-cond-s}
\end{figure}

After training, the errors in solution predictions behave much the
same way as detailed previously for $s=1$, but improve for some values
of $s$. To quickly describe the trend, consider a specific example of
$\alpha = (0.6, 1, 1, 0.6)$ (which has the structure
$\bar \alpha + \sigma$ over the first and last components), and
various $s$ values. We use as an error metric the quantity
$e_{0,s}^\dpg(\alpha)$ introduced in~\eqref{eq:error-quantities},
which compares the $L_2$ error between the bulk variables of the
surrogate prediction and a DPG finite element approximation over the
same mesh used for training. In Figure~\ref{fig:err-cond-s}, we
observe how $e_{0,s}^\dpg(\alpha)$ as well as the loss
$\cl L_{s,\alpha}(\uu^\dpg_\theta)$ varies as a function of $s$. The
remaining curve in Figure~\ref{fig:err-cond-s} concerns the
$s$-dependent Gram matrix of the inner product of norm
in~\eqref{eq:computable-Ls-defn} (needed to compute $\Lsdpg$), which
we denote by $R_{s,\pp}$. They are block diagonal matrices (with small
diagonal blocks for the lowest order case) whose conditioning depends
on~$s$. When dense linear algebra packages are used to invert its
small diagonal blocks in double precision, we measure the accuracy of
inversion by the Frobenius norm of
$\|I - R_{s, \pp} R_{s,\pp}^{-1}\|_{F}$ using the returned
inverse. The dashed curve in Figure~\ref{fig:err-cond-s} shows that
the number of digits of accuracy that might be potentially lost is
modest and can only be of concern for extremely high values of~$s$.

We summarize three takeaways from our numerical experience with
$\Lsdpg$ (with various examples of $\pp$) of which
Figure~\ref{fig:err-cond-s} is a typical example. First, the values of the loss
function $\Lsdpg$  closely track the squared error. Second,
there appears to be an optimal $s$ region where the error is lowest,
and since it coincides with the region where the computable loss
function is lowest, we can inexpensively locate it by sampling
$\Lsdpg$ over $s$. Third, local ill-conditioning is not an issue
anywhere near the optimal $s$ region as indicated by the accuracy of
inversion of the $s$-dependent Riesz matrices.

\subsection{NN trained with two-parameter DPG loss}

\begin{figure}
\centering
\begin{tikzpicture}
\begin{loglogaxis}[
    width=15cm,
    height=8cm,
    xlabel={sample number $m$ from $1, \dots, M$},
    ylabel={$L_2$-error},
    legend pos=north west,
    xmajorgrids = false,
    xminorgrids = false,
    ymajorgrids = true,
    yminorgrids = false,
    legend cell align=left,
]

\addplot+[mark=none, color=orange, thick, mark=none] table [x index=0, y index=1, col sep=comma] {fosls_e0_error.csv};
\addlegendentry{$\cmax(e^\fosls_{0})$}

\addplot+[const plot, thick, color=blue, mark=none] table [x index=0, y index=1, col sep=comma] {dpg_s1s2_e0_error-fixed.csv};
\addlegendentry{$\cmax(e^{\dpg{(s_1,s_2)}}_{0})$}
\end{loglogaxis}
\end{tikzpicture}
\caption{Cumulative maxima of
  $\alpha$-independent-error-norms obtained using the new
  DPG loss $\Lss{s_1}{s_2}$ with
  $s_1=50, s_2=100$, and the  FOSLS loss, over $M=100,000$ samples of
  $\alpha$, show the enhanced robustness of the former. }
\label{fig:e0-FOSLS-DPGs1s2_L2_error_ratios}
\end{figure}

Recall Example~\ref{eg:uw} which suggested the possibility of
obtaining sharp error control for $u$ and $q$ in the $\alpha$-independent $L_2$-norm (thus  controlling the solution of \eqref{laplaceDP} in a model-compliant norm).
The extent to which this might be realizable for the DPG interior
variables is suggested by
Theorem~\ref{thm:dpg-loss-s}(\ref{item:dpg-loss-s-4})--(\ref{item:dpg-loss-s-5}),
which hints 
 training an NN (again using the loss~\eqref{mean-squared} but now) with the two-parameter fiber loss  functions $\Lss{s_1}{s_2}$ 
introduced in~\eqref{eq:computable-Ls1s2}. If we have an $\alpha$ 
data set range, its maximal and minimal ranges can be used to determine 
an appropriate $s_1 >0$ (say by sampling $\Lsdpg$
as in Figure~\ref{fig:err-cond-s}).
Then setting some $s_2 > s_1$, 
it would be natural to train $\Lss{s_1}{s_2}$ over a fixed box
$[\min(\alpha), \max(\alpha)]^{m_\alpha}$. However, in this
subsection, to keep the tests challenging and be application-agnostic
(without presuming any knowledge of minimum and maximum of $\alpha$),
we will continue to use the unbounded $\chi^2_{\bar\alpha, \sigma}$
distribution as before, setting
$\bar \alpha = (10^{-1}, 1, 1, 10^{-1})$, $\sigma = 0.5,$ and fixing
$s_1=50, s_2=100$.  The associated NN is trained setting with a
learning rate $0.001$, $5000$ epochs, $M = 1024$ training samples, and
a batch size of $32$ samples.

To test the performance of the trained NN,
we compute for each sampled $\alpha$ the DPG finite
element solution for a fixed $s=1$, take its interior part in $\Xdpgo$, denoting it by  $\x^{\dpg(s=1)}_{0,h}$, 
and compare it with the
NN prediction for the same $\alpha$, denoted by $\x^{\dpg(s_1,s_2)}_{0,\theta}$, obtained by substituting 
the fiber losses $\Lss{s_1}{s_2}$ into the lifted loss~\eqref{mean-squared}. The comparison is done using the {$\alpha$-independent} $L_2$ error norm
\begin{equation*}
  e_{0}^{\dpg(s_1,s_2)} :=
  \|\x^{\dpg(s_1,s_2)}_{0,\theta} -  \x^{\dpg(s=1)}_{0,h}\|_{L_2}^2,
\end{equation*}
which we view as a function of $\alpha$.
This error is compared against the analogous FOSLS error 
$e_{0}^{\fosls}$ introduced before in~\eqref{eq:error-quantities}.
Results from an experiment
carried out over $100,000$ random samples of $\alpha$, taken from the previous
distribution, are displayed in
Figure~\ref{fig:e0-FOSLS-DPGs1s2_L2_error_ratios}.  Clearly, the DPG error in the $\alpha$-independent norm shows 
$\alpha$-robust performance, while the other does not.
These results show how tunable DPG loss functions gain enhanced parameter robustness.


\section*{Acknowledgments}

This work was supported in part by the NSF under an FRG grant (DMS-2245077, DMS-2245097) and DMS-2012469. The work also benefited from the activities organized under the NSF RTG program (grants DMS-2136228 and DMS-2038080). The authors also acknowledge funding by the Deutsche Forschungsgemeinschaft (DFG, German Research Foundation)---project number 442047500---through the Collaborative Research Center ``Sparsity and Singular Structures'' (SFB 1481).

\bibliography{bibliography} 
\bibliographystyle{siam}

\end{document}